\documentclass[a4paper,12pt,reqno]{amsart}
\usepackage{amssymb}
\usepackage{amsmath}
\usepackage{anyfontsize}
\usepackage{hyperref}
\usepackage{mathtools}
\usepackage{mathrsfs}
\usepackage{bbm}
\usepackage{tikz}
\usetikzlibrary{commutative-diagrams}
\usepackage{tikz-cd}
\usepackage{bm}
\usepackage{bm}
\usepackage{geometry}
\usepackage{colordvi}
\usepackage{color}

\DeclareMathOperator{\ovep}{\overline{\varepsilon}}

\allowdisplaybreaks

\numberwithin{equation}{section}

\theoremstyle{definition}
\newtheorem{theorem}{Theorem}[section]
\newtheorem{corollary}[theorem]{Corollary}
\newtheorem{lemma}[theorem]{Lemma}
\newtheorem{definition}[theorem]{Definition}
\newtheorem{proposition}[theorem]{Proposition}
\newtheorem{example}[theorem]{Example}
\newtheorem{remark}[theorem]{Remark}
\newtheorem{que}[theorem]{Question}

\begin{document}
\title[Homological properties of quantum groups at roots of unity]{Homological properties of quantum groups governed by small quantum groups}
\author{Yimin Huang}
\address{School of Mathematical Sciences, Fudan University, Shanghai 200433, China}
\email{21110180008@m.fudan.edu.cn}

\author{Tiancheng Qi}
\address{School of Mathematical Sciences, Fudan University, Shanghai 200433, China}
\email{tcqi21@m.fudan.edu.cn}

\author{Quanshui Wu}
\address{School of Mathematical Sciences, Fudan University, Shanghai 200433, China}
\email{qswu@fudan.edu.cn}

\author{Ruipeng Zhu}
\address{School of Mathematics, Shanghai University of Finance and Economics, Shanghai 200433, China}
\email{zhuruipeng@sufe.edu.cn}



\begin{abstract}

We develop a general characteristic-free framework for studying the homological properties of a broad class of module-finite Hopf algebras. This framework makes it possible to reduce the study of the homological properties of many quantum groups at roots of unity and their multiparameter deformations to the study of the corresponding small quantum groups. For any affine Hopf algebra $H$ admitting a large central Hopf subalgebra $C$, we prove that its left homological integral space, in the sense of Lu-Wu-Zhang, is isomorphic as a bimodule to the left integral space of the identity fiber algebra, which is a finite-dimensional Hopf algebra. Consequently, $H$ is a symmetric Frobenius extension of $C$ if and only if the corresponding identity fiber algebra is unimodular and the square of the antipode of $H$ is inner, thus providing an effective criterion for the Calabi-Yau property of $H$. For a broad class of quantum groups at roots of unity, an appropriate large central Hopf subalgebra can be chosen such that the identity fiber algebra is the corresponding small quantum group. Therefore, some homological properties of these big quantum groups 
are governed by those of their corresponding small quantum groups. Assuming that the base field is algebraically closed, we prove that $H$ is unimodular if and only if, for some (equivalently, every) maximal ideal $\mathfrak{m}$ of $C$, the category of finite-dimensional representations of the fiber algebra at $\mathfrak{m}$ is unimodular in the sense of Yadav as a module category over the finite tensor category of finite-dimensional representations of the identity fiber algebra. As an application, we prove that all Andruskiewitsch-Angiono-Yakimov large quantum groups 
are affine noetherian unimodular Artin-Schelter Gorenstein Hopf algebras. We also give a necessary and sufficient condition for these large quantum groups 
to be Calabi-Yau. 

\end{abstract}


\thanks{This research has been supported by NSFC (Grant No. 12471032) and NSFC (Grant No. 12301052).}

\maketitle	

\section*{Introduction}
Quantum groups at roots of unity, along with their multiparameter deformations, are closely related to numerous areas of mathematics and have a wide range of applications. These include the construction of \textit{quantum invariants} of links and $3$-manifolds \cite{MR1091619,MR1354257,MR1191386,MR3556349}, (semisimple and non-semisimple) \textit{modular tensor categories} (MTCs) \cite{MR1862634,MR2286123,MR4498166,MR4652904}, and (semisimple and non-semisimple) \textit{topological quantum field theories} (TQFTs) \cite{MR4788097,MR2674592,MR3539369,MR1292673}, among others. The structure theory \cite{MR1013053,MR1066560,MR2178656}, representation theory \cite{MR1124981,MR1296515,MR2178656,MR4600057}, and homological properties \cite{MR1482982,MR2054387} of quantum groups at roots of unity have received considerable attention and been extensively studied over the past three decades.


Let $\mathfrak{g}$ be a complex semisimple Lie algebra, and let $\epsilon\in\mathbb{C}$ be a primitive root of unity whose order satisfies specified restrictions depending on $\mathfrak{g}$.
The De Concini-Kac (DK for short) quantized enveloping algebra $\mathcal{U}_{\epsilon}(\mathfrak{g})$ \cite{MR1103601} is defined as the specialization at $q=\epsilon$ of the unrestricted integral form of the Drinfeld-Jimbo quantized enveloping algebra $\mathcal{U}_{q}(\mathfrak{g})$ \cite{MR934283,MR797001}. The representation theory of $\mathcal{U}_{\epsilon}(\mathfrak{g})$ shares many notable similarities with that of the universal enveloping algebra of a restricted Lie algebra. A key feature of the quantum group $\mathcal{U}_{\epsilon}(\mathfrak{g})$ is that it contains a canonical \textit{large} central Hopf subalgebra $C_{\epsilon}(\mathfrak{g})$ (see \eqref{eq-canceHopfsubalgof-Uepg} for its definition), in the sense that $\mathcal{U}_{\epsilon}(\mathfrak{g})$ is a finitely generated $C_{\epsilon}(\mathfrak{g})$-module. The Hopf quotient 
$$\mathcal{U}_{\epsilon}(\mathfrak{g})/C_{\epsilon}(\mathfrak{g})^{+}\mathcal{U}_{\epsilon}(\mathfrak{g})$$ is precisely the Lusztig small quantum group $\mathfrak{u}_{\epsilon}(\mathfrak{g})$ \cite{MR1013053,MR1066560}. Here $C_{\epsilon}(\mathfrak{g})^{+}$ denotes the augmentation ideal of the Hopf algebra $C_{\epsilon}(\mathfrak{g})$. It is well known that the small quantum group $\mathfrak{u}_{\epsilon}(\mathfrak{g})$ is a normal Hopf subalgebra of the Lusztig quantized enveloping algebra at $\epsilon$ \cite{MR1013053,MR1066560}, obtained by specializing the restricted integral form of $\mathcal{U}_{q}(\mathfrak{g})$. Similarly, the Hopf quotient of the De Concini-Kac-Procesi (DKP for short) quantized Borel subalgebra $\mathcal{U}_{\epsilon}^{\geq 0}(\mathfrak{g})$ \cite{MR1351503} with respect to its canonical large central Hopf subalgebra coincides with the small quantum Borel subgroup $\mathfrak{u}_{\epsilon}^{\geq 0}(\mathfrak{g})$; see \S \ref{subsec-QBsgrps}.

The homological properties of the quantum groups $\mathcal{U}_{\epsilon}(\mathfrak{g})$ and $\mathcal{U}_{\epsilon}^{\geq 0}(\mathfrak{g})$ are well understood. In \cite{MR2054387}, Chemla computed the rigid dualizing complex of 
$\mathcal{U}_{\epsilon}(\mathfrak{g})$, and thereby established that $\mathcal{U}_{\epsilon}(\mathfrak{g})$ is a Calabi-Yau algebra in the sense of Ginzburg \cite{Ginz2007CY}. On the other hand, although $\mathcal{U}_{\epsilon}^{\geq 0}(\mathfrak{g})$ is skew Calabi-Yau, it is not Calabi-Yau. Its Nakayama automorphism was determined by Brown-Gordon-Stroppel in \cite{MR2379091}. See \S \ref{subsec-sCYrigG} for background on these notions. Note that the Lusztig small quantum group $\mathfrak{u}_{\epsilon}(\mathfrak{g})$ is a \textit{unimodular} and, moreover, \textit{factorizable} \cite[Corollary A.3.3]{MR1354257} Hopf algebra, whereas the small quantum Borel subgroup  $\mathfrak{u}_{\epsilon}^{\geq 0}(\mathfrak{g})$ is \textit{neither}. This observation motivates the following question:

\begin{que}
Are the homological properties of the \textit{big quantum groups} at roots of unity governed by the corresponding \textit{small quantum groups}?\label{que-big-small-hp-intro}
\end{que}

Before proceeding with Question \ref{que-big-small-hp-intro}, we briefly return from the quantum setting to the classical setting. Let $(\mathfrak{r},[p])$ be a finite-dimensional restricted Lie algebra over a field of characteristic $p>0$. The Hopf quotient of its universal enveloping algebra $\mathcal{U}(\mathfrak{r})$ by the ideal generated by the augmentation ideal of the \textit{$p$-center} is precisely the restricted enveloping algebra $\mathfrak{u}(\mathfrak{r})$; see \cite[p.114]{MR1898492}. It is well known that $\mathcal{U}(\mathfrak{r})$ is skew Calabi-Yau. In \cite{MR1762922}, Yekutieli described the rigid dualizing complex of the universal enveloping algebra of an arbitrary finite-dimensional Lie algebra. It follows from Yekutieli's result that $\mathcal{U}(\mathfrak{r})$ is Calabi-Yau if and only if $\operatorname{tr}(\operatorname{ad}_{x})=0$ for all $x\in \mathfrak{r}$, where $\operatorname{ad}_{x}\colon \mathfrak{r}\to \mathfrak{r}$ denotes the adjoint action of $x$ on $\mathfrak{r}$. A classical theorem of Schue \cite{MR185045} states that $\operatorname{tr}(\operatorname{ad}_{x})=0$ for all $x\in \mathfrak{r}$ if and only if the restricted enveloping algebra $\mathfrak{u}(\mathfrak{r})$ is a symmetric Frobenius algebra, which is equivalent to $\mathfrak{u}(\mathfrak{r})$ being unimodular. So the universal enveloping algebra $\mathcal{U}(\mathfrak{r})$ is Calabi-Yau if and only if the restricted enveloping algebra $\mathfrak{u}(\mathfrak{r})$ is unimodular.

Motivated by the preceding discussion, we consider Question \ref{que-big-small-hp-intro} in a more general context.

\begin{que}
Let $H$ be an affine (i.e., finitely generated as a $\mathbbm{k}$-algebra) Hopf algebra over a field $\mathbbm{k}$, and let $C$ be a central Hopf subalgebra of $H$ such that $H$ is finitely generated as a $C$-module. Denote  the augmentation ideal of $C$ by $C^{+}$. Are the homological properties of $H$ governed by the finite-dimensional Hopf quotient $H/C^{+}H$?\label{modulefinite-identfiber-intro}
\end{que}

It is worth noting that, in addition to the DK quantized enveloping algebras $\mathcal{U}_{\epsilon}(\mathfrak{g})$, the DKP quantized Borel subalgebras $\mathcal{U}_{\epsilon}^{\geq 0}(\mathfrak{g})$, and the universal enveloping algebras $\mathcal{U}(\mathfrak{r})$ of finite-dimensional restricted Lie algebras $(\mathfrak{r},[p])$, the framework considered in Question \ref{modulefinite-identfiber-intro} encompasses several other significant classes of Hopf algebras. These include all De Concini-Lyubashenko quantized coordinate rings $\mathcal{O}_{\epsilon}(G)$ at roots of unity $\epsilon$ \cite{MR1296515}, all Andruskiewitsch-Angiono-Yakimov (AAY for short) large quantum groups $U_{\mathfrak{q}}$ whose braiding matrices $\mathfrak{q}$ belong to a one-parameter family \cite{MR4600057}, and many Artin-Schelter (AS for short) Gorenstein Hopf algebras of low Gelfand-Kirillov dimension \cite{MR2661247,MR2732991,MR2320655,MR4042591,MR4977721}.

\begin{remark}\label{rmk-affmofiHopf-repCH}

Suppose the base field $\mathbbm{k}$ is algebraically closed of characteristic zero, and let the pair $(H,C)$ be in the setting of Question \ref{modulefinite-identfiber-intro}. Then the maximal spectrum $\operatorname{maxSpec}C$ is an algebraic group with respect to the convolution product, see \S \ref{subsec-fiberalgg-mfhopf} for the details. By \cite[Theorem 1.1 (1)]{HMQW2026chev}, the associated sheaf $\widetilde{H}$ of $H$ as a $C$-module defines a \textit{vector bundle over the algebraic group} $\operatorname{maxSpec} C$\textit{ with fibers $H/\mathfrak{m}H$ at each point }$\mathfrak{m}\in \operatorname{maxSpec} C$. The bundle of algebras $$\{H/\mathfrak{m}H\mid \mathfrak{m}\in \operatorname{maxSpec} C\}$$ plays a central role in the study of the representation theory of $H$; see \cite{MR1898492,MR1863398,MR1989650,MR1676211,MR1288995} for example. By \cite[Theorem 1.1 (2)]{HMQW2026chev}, the pair $(H,C)$, equipped with the Hattori-Stallings trace map, is a Cayley-Hamilton Hopf algebra in the sense of De Concini-Procesi-Reshetikhin-Rosso \cite{MR2178656}. This enables the study of the representation theory of $H$ within the general framework of Cayley-Hamilton Hopf algebras. Consequently, the representation theory of many significant quantum groups at roots of unity can be investigated in a unified framework rather than on a case-by-case basis. For recent developments in this direction, we refer to \cite{MR4891381,HMQW2026chev,HQWZ2026chev}.
\end{remark}

This paper provides an affirmative answer to Question \ref{modulefinite-identfiber-intro} from the perspective of \textit{homological integrals} in the sense of Lu-Wu-Zhang \cite{MR2320655} (see Definition \ref{def-homolointeLWZ} and the discussion below), thereby establishing a criterion for $H$ to be Calabi-Yau in terms of the finite-dimensional Hopf algebra $H/C^{+}H$; see Theorem \ref{thm-fghopfhomol-unimosCY}. Furthermore, we establish a connection between the \textit{unimodularity} of $H$ and \textit{unimodular module categories} in the sense of Yadav \cite{MR4632574}; see Theorem \ref{thm-intromodufihipf-fibunim8}. Our results apply to numerous important quantum groups at roots of unity and their multiparameter deformations, including all AAY large quantum groups \cite{MR4600057}, thus answering Question \ref{que-big-small-hp-intro}.

In the remainder of this introduction, we assume that $H$ is an affine Hopf algebra over the base field $\mathbbm{k}$, of \textit{arbitrary} characteristic unless otherwise stated, and that $H$ is finitely generated as a module over a central Hopf subalgebra $C$. Then $H$ is a \textit{module-finite $C$-algebra} in the sense of \cite[p. xv]{MR2080008}. By \cite[Theorem 0.2]{MR1938745} and \cite[Corollary 2]{MR2271358}, $H$ is a noetherian AS-Gorenstein Hopf algebra (see Definition \ref{def-ASGore-hopf}) with bijective antipode. In this setting, the \textit{homological integral} is considered in the sense of Lu-Wu-Zhang \cite{MR2320655}: if $d$ is the self-injective dimension of $H$, then any nonzero element of
$$\textstyle \int_{ H}^{l}\coloneqq \operatorname{Ext}_{H}^{d}({_{H}\mathbbm{k}}, {_{H}H})
$$
is called a \textit{left homological integral} of $H$, and $\textstyle \int_{H}^{l}$ is called the \textit{left homological integral space}. The definition of the right homological integral is symmetric, and the right homological integral space is denoted by $\textstyle \int_{H}^{r}$. We call $H$ \textit{unimodular} \cite{MR2320655} if $\textstyle \int_{H}^{l}\cong \mathbbm{k}$ as $H$-$H$-bimodules. For finite-dimensional Hopf algebras, the homological integral coincides with the classical integral, and this notion of unimodularity agrees with the usual condition that the spaces of left and right integrals coincide. The \textit{integral order} \cite{MR2320655} of $H$ is defined by
$$\operatorname{io}(H)
=
\inf\left\{
n\ge 1 \,\bigm|
\left(\textstyle \int_{H}^{l}\right)^{\otimes n}\cong \mathbbm{k}\text{ as\ }H\text{-}H\text{-bimodules}
\right\}.$$
By \cite[Lemma 2.3]{MR2320655}, $H$ is unimodular if and only if $\operatorname{io}(H)=1$. If $H$ is module-finite, then $\operatorname{io}(H)$ is necessarily finite; see \cite[Theorem 2.3 (b)]{MR2661247} or \cite[Theorem B (d)]{MR4891381}.

Homological integrals are powerful tools in the study of noetherian AS-Gorenstein Hopf algebras. Notable applications include
\begin{itemize}
\item[$\bullet$] establishing a ``Maschke theorem'' for noetherian affine PI Hopf algebras \cite{MR2320655};
\item[$\bullet$] describing the Nakayama automorphism of AS-Gorenstein Hopf algebras \cite{MR2437632};
\item[$\bullet$] providing an invariant for classifying noetherian affine prime Hopf algebras of low Gelfand-Kirillov dimension (see \cite{MR2320655,MR2661247,MR3490761,MR4042591} for a partial list);
\item[$\bullet$] providing a lower bound, optimal in a certain sense, for the level of the lowest discriminant subvariety of an affine Cayley-Hamilton Hopf algebra \cite{MR4891381}.
\end{itemize}

As observed in Remark \ref{rmk-affmofiHopf-repCH}, if the base field $\mathbbm{k}$ is algebraically closed (including in positive characteristic), then $\operatorname{maxSpec} C$ is a group under convolution, with identity element $\mathfrak{m}_{\overline{\varepsilon}} \coloneqq C^{+}$. For any $\mathfrak{m}\in \operatorname{maxSpec}C$, the finite-dimensional algebra $H/\mathfrak{m}H$ is called the \textit{fiber algebra} of $H$ at $\mathfrak{m}$. The coproduct of $H$ endows the \textit{identity fiber algebra} $H/\mathfrak{m}_{\ovep}H$ with the structure of a finite-dimensional Hopf algebra, and each fiber algebra $H/\mathfrak{m}H$ with the structure of an $H/\mathfrak{m}_{\ovep}H$-$H/\mathfrak{m}_{\ovep}H$-bicomodule algebra; see \S \ref{subsec-fiberalgg-mfhopf} for details. Then $H/\mathfrak{m}_{\ovep}H\text{-mod}$ is a finite tensor category, and $H/\mathfrak{m}H\text{-mod}$ can be naturally regarded as a module category over the finite tensor category $H/\mathfrak{m}_{\ovep}H\text{-mod}$. Moreover, \cite[Proposition 3.8]{HMQW2026chev} proves that $H/\mathfrak{m}H\text{-mod}$ is an indecomposable exact module category over $H/\mathfrak{m}_{\ovep}H\text{-mod}$.


We now state our first result, which relates the homological integral of $H$ to the integral of the identity fiber $H/\mathfrak{m}_{\ovep}H$. Consequently, the computation of homological integrals for a large class of, often infinite-dimensional, module-finite Hopf algebras is reduced to the computation of integrals for the finite-dimensional Hopf algebra $H/\mathfrak{m}_{\ovep}H$. In \cite[Question III. 4.8]{MR1898492}, Brown and Goodearl asked under which conditions $C\subseteq H$ is a symmetric Frobenius extension. We establish a necessary and sufficient condition for $H$ to be a symmetric Frobenius extension of $C$, thereby obtaining a criterion for $H$ to be Calabi-Yau in terms of the identity fiber $H/\mathfrak{m}_{\ovep}H$.

\begin{theorem} \label{thm-fghopfhomol-unimosCY}
Let $H$ be an affine Hopf algebra over a field $\mathbbm{k}$, with a central Hopf subalgebra $C$ such that $H$ is finitely generated as a $C$-module. Let $S$ be the antipode of $H$. Then
\begin{itemize}
\item[(a)] $\textstyle \int_{H}^{l} \cong \textstyle \int_{H/\mathfrak{m}_{\ovep}H}^{l}$ and $\textstyle \int_{H}^{r} \cong \textstyle \int_{H/\mathfrak{m}_{\ovep}H}^{r}$ as $H$-$H$ bimodules.
\item[(b)] $H$ is unimodular if and only if $H/\mathfrak{m}_{\ovep}H$ is unimodular.
\item[(c)] $W_{l}(\chi)S^{2}$ is the Nakayama automorphism of $H$, where $\chi:H\to\mathbbm{k}$ is the character of $H$ induced by the distinguished group-like element of the dual Hopf algebra $(H/\mathfrak{m}_{\ovep}H)^{\circ}$, and $W_{l}(\chi)$ is the left winding automorphism associated with $\chi$.
\item[(d)] The extension $C\subseteq H$ is a symmetric Frobenius extension if and only if $S^{2}$ is inner and $H/\mathfrak{m}_{\ovep}H$ is unimodular. In particular, $H$ is Calabi-Yau if and only if
\begin{itemize}
\item[(i)] $H$ has finite global dimension, 
\item[(ii)] $S^{2}$ is an inner automorphism of $H$, and
\item[(iii)] $H/\mathfrak{m}_{\ovep}H$ is unimodular.
\end{itemize}
\item[(e)] If $\mathbbm{k}$ is an algebraically closed field of characteristic zero, $H$ is semiprime, and $H$ has the Chevalley property (i.e., the tensor product of any two irreducible $H$-modules is completely reducible as an $H$-module), then $H$ is a Calabi-Yau algebra.
\end{itemize}
\end{theorem}

\begin{remark}
In \cite{MR2678828}, He-Van Oystaeyen-Zhang proved that a noetherian AS-Gorenstein Hopf algebra is Calabi-Yau if and only if it is a unimodular AS-regular Hopf algebra and the square of its antipode is inner. Therefore, the necessary and sufficient condition for $H$ to be Calabi-Yau in Theorem \ref{thm-fghopfhomol-unimosCY} (d) also follows from the theorem of He-Van Oystaeyen-Zhang together with Theorem \ref{thm-fghopfhomol-unimosCY} (b).
\end{remark}

\begin{remark}
If, in the assumptions of Theorem \ref{thm-fghopfhomol-unimosCY}, one requires only that $C$ be a normal Hopf subalgebra of $H$, the conclusion may fail, even when $C$ is commutative and $H$ is finitely generated as a $C$-module; see Example \ref{eg-infdihdralgrp}. Therefore, Theorem \ref{thm-fghopfhomol-unimosCY} \textit{does not necessarily hold for commutative-by-finite Hopf algebras in the sense of Brown-Couto} \cite{MR4201485}.
\end{remark}

Theorem \ref{thm-fghopfhomol-unimosCY} is proved in \S \ref{subsec-pfofintrm-thm}. This theorem 
provides a new approach to establishing the Calabi-Yau property of $\mathcal{U}_{\epsilon}(\mathfrak{g})$ and to computing the Nakayama automorphism of $\mathcal{U}_{\epsilon}^{\geq 0}(\mathfrak{g})$ via the corresponding small quantum groups (see Propositions \ref{prop-DKquantgrp-CY} and \ref{prop-nakaya-quantborel-BGS}). Note that the quantum group $\mathcal{U}_{\epsilon}^{\geq 0}(\mathfrak{g})$ is a skew Calabi-Yau Hopf algebra for which $S^{2}$ is inner, but it is not unimodular. In the setting of Theorem \ref{thm-fghopfhomol-unimosCY}, there exist unimodular skew Calabi-Yau Hopf algebras for which $S^{2}$ is not inner (see Example \ref{eg-geralizedsuzuki-algunimodular}).

Since the group algebra of a finite group is always a symmetric Frobenius algebra, Theorem \ref{thm-fghopfhomol-unimosCY} implies that the group algebra $\mathbbm{k}[G]$ of a finitely generated \textit{central-by-finite} group $G$ is a unimodular Hopf algebra; see Proposition \ref{prop-fggrp-io1ifcbf}. This shows that homological integrals can provide a necessary condition for a finitely generated group to be central-by-finite. Theorem \ref{thm-fghopfhomol-unimosCY} is also applied to study the homological properties of the algebra $\mathcal{H}_{LY}=H(e_{\pm 1},f_{\pm 1},u,v)$ \cite[Definition 5.1]{MR4977721}, which arises in the study of Hopf algebras with the dual Chevalley property of discrete corepresentation type. By \cite[Proposition 5.4]{MR4977721}, $\mathcal{H}_{LY}$ is a non-pointed affine module-finite Hopf algebra with the dual Chevalley property of discrete corepresentation type. Using Theorem \ref{thm-fghopfhomol-unimosCY}, we prove that $\mathcal{H}_{LY}$ is a unimodular AS-Gorenstein Hopf algebra of infinite global dimension and 
a symmetric Frobenius extension of its canonical large central Hopf subalgebra; see Example \ref{eg-Liu-Yualg}.

In Theorem \ref{thm-noethAS-intebimoduiso-un}, we establish a stronger version of Theorem \ref{thm-fghopfhomol-unimosCY} (a), in which $H$ is allowed to be an arbitrary noetherian AS-Gorenstein Hopf algebra and $C$ to be an arbitrary central Hopf subalgebra of $H$. This result facilitates the study of the homological properties of certain non-PI noetherian Hopf algebras through Theorem \ref{thm-noethAS-intebimoduiso-un}. For example, if $q\in \mathbb{C}^{\times}$ is not a root of unity, Theorem \ref{thm-noethAS-intebimoduiso-un} allows us to deduce the Calabi-Yau property of the standard Drinfeld double quantum group $\mathfrak{D}_{q}(\mathfrak{g})$ from that of $\mathcal{U}_q(\mathfrak{g})$; see Theorem \ref{thm-DDQGrp-CY}. 


The study of (semisimple and non-semisimple) MTCs and their applications has long been an active area of research and continues to attract considerable attention. A fundamental problem is the construction of new MTCs from a given one.
One effective approach is to consider the category of \textit{local modules} (see \cite[Definition 4.10]{MR4652904}) over a ``nice'' Frobenius algebra object in an MTC; see \cite{MR1936496,MR4652904} and the references therein. Inspired by this idea, Yadav \cite{MR4632574} introduced the notion of a \textit{unimodular module category} (see Definition \ref{def-unimmodu-excat}) for exact module categories over a finite tensor category $\mathcal{C}$, thereby obtaining a functorial construction of Frobenius algebra objects in the Drinfeld center of $ \mathcal{C}$. Since $H/\mathfrak{m}H\text{-mod}$ is an indecomposable exact module category over the finite tensor category $H/\mathfrak{m}_{\ovep}H\text{-mod}$, one may consider its unimodularity 
as a module category in the sense of Yadav. This motivates our second result, which provides equivalent characterizations of the unimodularity of $H$ in terms of the unimodularity, in the sense of Yadav, of the module categories $H/\mathfrak{m}H\text{-mod}$ associated with the fiber algebras $H/\mathfrak{m}H$.

\begin{theorem}
Let $H$ be an affine Hopf algebra over an algebraically closed field $\mathbbm{k}$ that is finitely generated as a module over a central Hopf subalgebra $C$.\label{thm-intromodufihipf-fibunim8} 
\begin{itemize}
\item[(a)]The following are equivalent:
\begin{itemize}
\item[(i)] $H$ is unimodular;
\item[(ii)]For any $\mathfrak{m}\in \operatorname{maxSpec}C$, $H/\mathfrak{m}H\text{-mod}$ is unimodular as an indecomposable exact module category over $H/\mathfrak{m}_{\ovep}H\text{-mod}$;
\item[(iii)]There exists $\mathfrak{m}\in \operatorname{maxSpec}C$ such that $H/\mathfrak{m}H\text{-mod}$ is unimodular as an indecomposable exact module category over $H/\mathfrak{m}_{\ovep}H\text{-mod}$.
\end{itemize}
\item[(b)]If $H$ is Calabi-Yau (e.g., all De Concini-Kac quantized enveloping algebra $\mathcal{U}_{\epsilon}(\mathfrak{g})$ at roots of unity $\epsilon$), then $H/\mathfrak{m}H\text{-mod}$ is unimodular as an indecomposable exact module category over $H/\mathfrak{m}_{\ovep}H\text{-mod}$ for all $\mathfrak{m}\in \operatorname{maxSpec}C$.
\end{itemize}
\end{theorem}

Theorem \ref{thm-intromodufihipf-fibunim8} follows from Theorem \ref{thm-unimodularcomodualg-fiber}. Moreover, the argument extends to the more general setting of Hopf group coalgebras of finite type in the sense of Turaev-Virelizier \cite{MR2674592,MR1903398}; see Remark \ref{rmk-unimocat-Hopfcogralg}. By Theorem \ref{thm-intromodufihipf-fibunim8}, all unimodular affine Hopf algebras with a large central Hopf subalgebra naturally give rise to Yadav's unimodular module categories. This includes the following infinite-dimensional Hopf algebras: 
\begin{itemize}
\item[$\bullet$] all DK quantized enveloping algebras $\mathcal{U}_{\epsilon}(\mathfrak{g})$ at roots of unity;
\item[$\bullet$] group algebras of finitely generated central-by-finite groups (see Proposition \ref{prop-fggrp-io1ifcbf});
\item[$\bullet$] the algebra $\mathcal{H}_{LY} $ (see Example \ref{eg-Liu-Yualg});
\item[$\bullet$] all AAY large quantum groups $U_{\mathfrak{q}}$ whose braiding matrices $\mathfrak{q}$ belong to a one-parameter family (see Theorem \ref{thm-intro-AAYlargeqgrps} below);
\item[$\bullet$] a large class of multiparameter quantized enveloping algebras $\mathcal{U}_{\mathfrak{q}}(\mathfrak{g})$, including the Drinfeld double quantum groups $\mathfrak{D}_{\epsilon}(\mathfrak{g})$ at roots of unity (see Example \ref{eg-multiparamqgrp-AAYlarge}).
\end{itemize}

In \cite{MR4600057}, Andruskiewitsch-Angiono-Yakimov developed a general Poisson-geometric framework for the uniform, rather than case-by-case, study of the representation theory of all contragredient quantum supergroups at roots of unity, as well as a broad class of multiparameter quantized enveloping algebras at roots of unity.  For this purpose, they considered the quantum doubles $U_{\mathfrak{q}}$ (called \textit{large quantum groups}) of the bosonizations of all distinguished pre-Nichols algebras \cite{MR3459702} associated with braiding matrices $\mathfrak{q}$ belonging to a one-parameter family, 
such that the corresponding Nichols algebra of diagonal type is finite-dimensional, condition \eqref{eq-ffc-cartroot-largech7} holds (see \cite[Eq. (4.26)]{MR4600057}), and the additional technical conditions listed in \cite[\S 4.1]{MR4600057} are satisfied. When $\mathfrak{q}$ is genuinely of finite Cartan type \cite[p.522]{MR4164719}, a large class of the multiparameter quantized enveloping algebras $\mathcal{U}_{\mathfrak{q}}(\mathfrak{g})$ considered in \cite{MR4506530,MR2642565} arise as special cases of $U_{\mathfrak{q}}$. For $\mathfrak{q}$ of non-Cartan type, the framework of Andruskiewitsch-Angiono-Yakimov includes multiparameter quantum groups at roots of unity associated with finite-dimensional simple contragredient Lie superalgebras, as well as quantizations at roots of unity in characteristic zero of certain simple Lie algebras in positive characteristic \cite{MR4600057}.

As in \cite{MR4600057}, the AAY large quantum groups $U_{\mathfrak{q}}$ considered in this paper are defined over $\mathbb{C}$. The difference is that we require only that the Nichols algebra corresponding to the braiding matrix $\mathfrak{q}$  be finite-dimensional and that condition \eqref{eq-ffc-cartroot-largech7} hold. The definition of the large quantum group $U_{\mathfrak{q}}$ and the relevant background are recalled in \S \ref{sec-largeqgrp-setup}. Under condition \eqref{eq-ffc-cartroot-largech7}, the quantum group $U_{\mathfrak{q}}$ has a canonical large central Hopf subalgebra $Z_{\mathfrak{q}}$, constructed by Andruskiewitsch-Angiono-Yakimov in \cite[\S 4.5]{MR4600057}. When $\mathfrak{q}$ is genuinely of finite Cartan type and of integral type (see Example \ref{eg-GGsmquangroupmu-9}), the identity fiber of $U_{\mathfrak{q}}$ with respect to the central Hopf subalgebra $Z_{\mathfrak{q}}$ is the small multiparameter quantum group considered by García-Gavarini in \cite[\S 7.3.2]{MR4506530}. The Lusztig small quantum group $\mathfrak{u}_{\epsilon}(\mathfrak{g})$ arises as a Hopf quotient for a specific choice of $\mathfrak{q}$; see Remark \ref{rmk-GGsmallquantumgrp-lusz}. To endow the pair $(U_{\mathfrak{q}},Z_{\mathfrak{q}})$ with a nontrivial Poisson order structure in the sense of Brown-Gordon \cite{MR1989650}, Andruskiewitsch-Angiono-Yakimov impose the additional conditions on the braiding matrix $\mathfrak{q}$ listed in \cite[\S 4.1]{MR4600057}. The geometry of the Poisson algebraic group $\operatorname{maxSpec} Z_{\mathfrak{q}}$ was studied in detail in \cite{MR4600057}, where its symplectic foliation was described in \cite[Theorem 8.2]{MR4600057}. By \cite[Theorem 20]{MR3459702}, if $\mathfrak{q}\in \text{M}_{\theta}(\mathbb{C})$, then the Gelfand-Kirillov dimension of $U_{\mathfrak{q}}$ is given by $|\mathcal{O}^{\mathfrak{q}}|+2\theta$, where $\mathcal{O}^{\mathfrak{q}}$ denotes the set of \textit{Cartan roots} of $\mathfrak{q}$ \cite{MR3459702}; see \eqref{eq-cartanroots-set}. Writing $\mathcal{O}_{+}^{\mathfrak{q}}
=\mathcal{O}^{\mathfrak{q}}\cap \mathbb{N}^{\theta}$, we have $|\mathcal{O}^{\mathfrak{q}}|=2|\mathcal{O}_{+}^{\mathfrak{q}}|.$ Note that the braiding matrix $\mathfrak{q}$ induces a canonical $\mathbb{Z}$-bilinear form $\mathfrak{q}:\mathbb{Z}^{\theta}\times \mathbb{Z}^{\theta}\to\mathbb{C}^{\times}$.
Our third result applies the preceding results to all AAY large quantum groups (see Corollary \ref{cor-ASgorenlargequan8-eregucart6} and Theorem \ref{thm-largqungrp-CY}). We prove that every large quantum group $U_{\mathfrak{q}}$ is a unimodular Noetherian affine AS-Gorenstein Hopf algebra and establish a necessary and sufficient condition on $\mathfrak{q}$ for 
$U_{\mathfrak{q}}$ to be Calabi-Yau. In particular, the bundle of fiber algebras of $U_{\mathfrak{q}}$ over $\operatorname{maxSpec} Z_{\mathfrak{q}}$ gives rise to a family of unimodular indecomposable exact module categories in the sense of Yadav.


\begin{theorem}\label{thm-intro-AAYlargeqgrps}
Let $\mathfrak{q}=(q_{ij})_{\theta\times \theta}$ be a braiding matrix such that the corresponding Nichols algebra is finite-dimensional and that the condition \eqref{eq-ffc-cartroot-largech7} holds. Let $U_{\mathfrak{q}}$ denote the corresponding Andruskiewitsch-Angiono-Yakimov large quantum group.
\begin{itemize}
\item[(a)] The large quantum group $U_{\mathfrak{q}}$ is a unimodular noetherian affine Artin-Schelter Gorenstein Hopf algebra of injective dimension $ |\mathcal{O}^{\mathfrak{q}}|+2\theta$. 
\item[(b)] The large quantum group $U_{\mathfrak{q}}$ is Calabi-Yau if and only if $\mathfrak{q}$ is of Cartan type, and there exist $\alpha,\beta\in \mathbb{Z}^{\theta}$ such that 
$$\mathfrak{q}(\alpha,\alpha_{j})\mathfrak{q}(\alpha_{j},\beta)=q_{jj}^{-1}$$ 
for all $1\leq j\leq \theta$.
\item[(c)] For any $\mathfrak{m}\in \operatorname{maxSpec}Z_{\mathfrak{q}}$, the category $U_{\mathfrak{q}}/\mathfrak{m}U_{\mathfrak{q}}\text{-mod}$ is a unimodular indecomposable exact module category over $U_{\mathfrak{q}}/\mathfrak{m}_{\ovep}U_{\mathfrak{q}}\text{-mod}$.
\end{itemize}
\end{theorem}

\begin{remark}
Let the braiding matrix $\mathfrak{q}$ satisfy the assumptions of Theorem \ref{thm-intro-AAYlargeqgrps}. If, in addition, $\mathfrak{q}$ is genuinely of finite Cartan type, then Theorem \ref{thm-intro-AAYlargeqgrps} provides a criterion for the Calabi-Yau property of a large class of multiparameter quantized enveloping algebras $\mathcal{U}_{\mathfrak{q}}(\mathfrak{g})$ considered in \cite{MR4506530,MR2642565}; see Corollaries \ref{cor-PHRmulparaquangro-regulardoman} and \ref{cor-mulQGrp-whenCY}. There exist multiparameter quantized enveloping algebras $\mathcal{U}_{\mathfrak{q}}(\mathfrak{g})$ that are not Calabi-Yau (Example \ref{eg-finitegen-maynotS2iner}). Moreover, Corollary \ref{cor-ASgorenlargequan8-eregucart6} (d) establishes that the AAY large quantized Borel algebra $U_{\mathfrak{q}}^{\geq 0}$ is Artin-Schelter regular if and only if $\mathfrak{q}$ is of Cartan type. The quantum groups $\mathcal{U}_{\epsilon}^{\geq 0}(\mathfrak{g})$ arise as special cases of $U_{\mathfrak{q}}^{\geq 0}$; see Example \ref{eq-positpart-Lztsmaqugrp9-nichols}.
\end{remark}

The paper is organized as follows. In \S \ref{section 1}, we review background material on (skew) Calabi-Yau algebras, rigid Gorenstein algebras, Hopf Galois extensions, and Stefan's spectral sequence. In \S \ref{sec-homolointTHM}, we prove Theorem \ref{thm-fghopfhomol-unimosCY}, establish a stronger version of the result, and present several examples illustrating these results. In \S \ref{sec-applica-UqgUqb-Dqg}, we apply Theorems \ref{thm-fghopfhomol-unimosCY} and \ref{thm-noethAS-intebimoduiso-un} to investigate the homological properties of the quantum groups $\mathcal{U}_{\epsilon}(\mathfrak{g}), \mathcal{U}_{\epsilon}^{\geq 0}(\mathfrak{g})$, and $\mathfrak{D}_{q}(\mathfrak{g})$. In \S \ref{sec-AAYlargequantumGrp}, we review the relevant background on AAY large quantum groups, including Nichols algebras of diagonal type, and prove Theorem \ref{thm-intro-AAYlargeqgrps} (a) and (b). In \S \ref{sec-yadaunimodular-cat}, after recalling basic notions concerning unimodular module categories in the sense of Yadav, we prove Theorem \ref{thm-intromodufihipf-fibunim8}.


\section{Preliminaries}\label{section 1}
In this section, after fixing some basic notation, we recall some preliminaries on (skew) Calabi-Yau algebras, rigid Gorenstein algebras, Hopf Galois extensions and Stefan's spectral sequence. The main references are \cite{MR1358765,MR2437632,MR3250287,MR4413364,MR5000225}.
\subsection{Notation and conventions}
Throughout this paper, $\mathbbm{k}$ is a fixed field and all algebras are over $\mathbbm{k}$. An unadorned $\otimes$ denotes $\otimes_{\mathbbm{k}}$. All modules over algebras considered here are left modules. Let $A$ be an algebra over $\mathbbm{k}$. The category of left $A$-modules is denoted by $A\operatorname{-Mod}$, and the category of finite-dimensional left $A$-modules is denoted by $A\operatorname{-mod}$.

We write ${\bf D}(A)$ for the derived category of $A\operatorname{-Mod}$. The derived Hom and tensor functors are denoted by $\operatorname{RHom}_{A}$ and $\otimes_{A}^{L}$, respectively. We refer to \cite{MR1269324} for standard background on derived categories and derived functors.


For a left $A$-module $M$, $\operatorname{p.dim}_{A}M$ and $\operatorname{inj.dim}_{A}M$ denote the projective dimension and injective dimension of $M$, respectively. The left and right global dimensions of $A$ are denoted by $\operatorname{l.gl.dim}A$ and $\operatorname{r.gl.dim}A$, respectively. When $\operatorname{l.gl.dim}A=\operatorname{r.gl.dim}A$ (for instance, when $A$ is noetherian), the common value is denoted by $\operatorname{gl.dim}A$.  The opposite algebra of $A$ is denoted by $A^{op}$. If $A$ is noetherian and both $\operatorname{inj.dim}_{A}A $ and $\operatorname{inj.dim}_{A^{op}}A $ are finite, then $\operatorname{inj.dim}_{A}A =\operatorname{inj.dim}_{A^{op}}A$ \cite[Lemma A]{MR244325}. In this case, we say that the noetherian algebra $A$ has \textit{finite injective dimension}. The enveloping algebra $A\otimes A^{op}$ of $A$ is denoted by $A^{e}$. Any $A$-$A$ bimodule $M$ can be identified with a left $A^{e}$-module via $(a\otimes b)m=amb$, and with a right $A^{e}$-module via $m(a\otimes b)=bma$. We will use these identifications freely.

For background on Hopf algebras and tensor categories, see \cite{MR1243637,MR4164719,MR3242743}. The coproduct, counit, and antipode of a Hopf algebra are denoted by $\Delta,\varepsilon$ and $S$, respectively. For a Hopf algebra $H$ and $h\in H$, we use Sweedler's notation 
$$\Delta(h)=\sum_{(h)} h_{(1)}\otimes h_{(2)} \in H \otimes H.$$
Similarly, for a left $H$-comodule $(V,\delta)$ and an element $v\in V$, we write 
$$\delta(v)=\sum_{(v)}v_{(-1)}\otimes v_{(0)}\in H\otimes V.$$ Let $\mathcal{C}$ be a finite tensor category over $\mathbbm{k}$ and let $\mathcal{M}$ be a finite left $\mathcal{C}$-module category.   If $A$ is an algebra in $\mathcal{C}$, the category of right $A$-modules in $\mathcal{C}$, denoted by $\text{Mod}_{\mathcal{C}}(A)$, naturally carries the structure of a left $\mathcal{C}$-module category \cite[Proposition 7.8.10]{MR3242743}. 




\subsection{Skew Calabi-Yau algebras and rigid Gorenstein algebras}\label{subsec-sCYrigG}
Suppose that $A$ is an algebra over $\mathbbm{k}$.  Let $M$ be an $A$-$A$-bimodule and let $\mu,\nu$ be algebra automorphisms of $A$. Denote by ${^{\nu}M}^{\mu}$ the twisted $A$-$A$-bimodule, which as a $\mathbbm{k}$-linear space is $M$, with the bimodule structure given by $a\cdot m\cdot b= \nu(a)m\mu(b)$ for all $a,b\in A$ and $m\in M$. If $\nu$ (resp., $\mu$) is the identity map of $A$, then ${^{\nu}M}^{\mu}$ is denoted by $M^{\mu}$ (resp., $^{\nu}M$). 

Suppose $\mu,\nu$ are algebra automorphisms of $A$. It is clear that $A^{\mu}\cong A^{\nu}$ as $A$-$A$-bimodules if and only if there exists an invertible element $u\in A$ such that $\mu(a)=u\nu(a)u^{-1}$ for all $a\in A$ (i.e., $\mu$ and $\nu$ differ by an inner automorphism of $A$).

\begin{definition}
[\cite{MR2437632,MR3250287}]Let $A$ be an algebra over $\mathbbm{k}$. 
\begin{itemize}
\item[(1)] $A$ is called \textit{homologically smooth} over $\mathbbm{k}$ if $A$ is perfect as an $A^{e}$-module (i.e., $A$ has a finitely generated $A^e$-projective resolution of finite length).
\item[(2)] $A$ is called \textit{skew Calabi-Yau} if $A$ is homologically smooth, and there is an integer $d$ and an algebra automorphism $\nu $ of $A$ such that
\begin{equation}\label{eq-defsCYExtgrp1}
\operatorname{Ext}_{A^{e}}^{i}(A,A^{e})\cong
\begin{cases}
0, & i\neq d,\\
A^{\nu}, &i=d,
\end{cases}
\end{equation}
as $A$-$A$ bimodules.
\item[(3)] $A$ is called \textit{rigid Gorenstein} if $A$ is noetherian with finite injective dimension $d$, and there is an algebra automorphism $\nu $ of $A$ such that \eqref{eq-defsCYExtgrp1} holds.
\end{itemize}
\end{definition}
The algebra automorphism $\nu$ in \eqref{eq-defsCYExtgrp1} is called the \textit{Nakayama automorphism} of $A$, which is unique up to an inner automorphism of $A$. When the Nakayama automorphism of a skew Calabi-Yau algebra $A$ is inner, then $A$ is a \textit{Calabi-Yau algebra} in the sense of Ginzburg \cite{Ginz2007CY}. If $A$ is skew Calabi-Yau, then the integer $d$ in \eqref{eq-defsCYExtgrp1} is $\operatorname{p.dim}_{A^{e}}A$, see \cite[Lemma 4.4]{MR4413364}. 
If a homologically smooth algebra $A$ is rigid Gorenstein with injective dimension $d$, then $A$ is skew Calabi-Yau and $\operatorname{p.dim}_{A^{e}}A=\operatorname{inj.dim}_{A}A=\operatorname{gl.dim}A=d$. There are noetherian Calabi-Yau algebras that are not rigid Gorenstein, see \cite[Example 4.15]{MR4413364}.



\subsection{Hopf Galois extensions and Stefan's spectral sequence}\label{subsec-UMact-stefan}
Let $H$ be a Hopf algebra 
and let $B$ be a left $H$-comodule algebra. The extension ${^{coH}\!B}\subseteq B$ is said to be \textit{left $H$-Galois} if the canonical map $$\beta: B\otimes_{{^{coH}\!B}}B\to  H\otimes B, a\otimes b\mapsto \sum_{(a)}a_{(-1)}\otimes a_{(0)}b$$ is bijective, where 
$${^{coH}\!B}=\left\{a\in B  \bigm| \sum_{(a)}a_{(-1)}\otimes a_{(0)}= 1\otimes a\right\}$$
is the subalgebra of coinvariants of the $H$-comodule $B$. Given a left $H$-Galois extension ${^{coH}\!B}\subseteq B$, we say that $B$ is a left \textit{$H$-Galois object} if ${^{coH}\!B}=\mathbbm{k}$. Right $H$-Galois extensions and right $H$-Galois objects are defined analogously.

In this subsection, we assume that $A\subseteq B$ is a left $H$-Galois extension with the canonical map $\beta:B\otimes_{A}B\to H\otimes B$. Consider the $\mathbbm{k}$-linear map 
\begin{equation}\label{eq-babti-taudef}
\tau:H\to B\otimes_{A}B,h\mapsto \beta^{-1}(h\otimes 1).
\end{equation}
For any $h\in H$, we  use the notation
$$\tau(h)=\sum\tau^{1}(h)\otimes_{A} \tau^{2}(h)\in B\otimes_{A}B.$$
By \cite[Remark 3.4]{MR1098989}, for any $h,k\in H$ and any $a\in A$, one has
\begin{equation}\label{eq-lefgalexumac-for1}
\sum a\tau^{1}(h)\otimes_{A} \tau^{2}(h)=\sum \tau^{1}(h)\otimes_{A} \tau^{2}(h)a\in B\otimes_{A}B,
\end{equation}
\begin{equation}\label{eq-lGeu-botbarims7}
\sum \tau^{1}(hk)\otimes_{A}\tau^{2}(hk)=\sum \tau^{1}(h)\tau^{1}(k)\otimes_{A}\tau^{2}(k)\tau^{2}(h)\in B\otimes_{A}B.
\end{equation}
It follows from \eqref{eq-lefgalexumac-for1} and \eqref{eq-lGeu-botbarims7} that $B\otimes_{A}B$ has a right $H$-module structure given by
\begin{equation}\label{eq-botimabriHmo-stu1}
(x\otimes_{A}y)\leftharpoonup h\coloneqq \sum x\tau^{1}(h)\otimes_{A} \tau^{2}(h)y \in B\otimes_{A}B
\end{equation}
for all $x,y\in B$ and $h\in H$. Thus $B\otimes_{A}B$ is a $B^{e}$-$H$-bimodule, with the left $B^{e}$-module structure induced by the outer action and the right $H$-module structure given by \eqref{eq-botimabriHmo-stu1}. Unless stated otherwise, \textit{the $B^{e}$-$H$-bimodule structure on $B\otimes_{A}B$ will always be understood to be the one defined above}. Moreover, considering the $B^{e}$-$H$-bimodule structure on $H\otimes B$ defined by
$$(b\otimes b^{\prime})\rightharpoonup  (h\otimes x)\leftharpoonup k\coloneqq \sum_{(b)}b_{(-1)}hk\otimes b_{(0)}xb^{\prime}$$
for all $b,b^{\prime},x\in B, h,k\in H$, the canonical map $\beta:B\otimes_{A}B\to H\otimes B$ is a $B^{e}$-$H$-bimodule isomorphism. Since $A^{e}$ is a subalgebra of $B^{e}$, $B^{e}$ admits a canonical $B^{e}$-$A^{e}$-bimodule structure. Consequently, $B^{e}\otimes_{{A}^{e}}A$ may be regarded as a left $B^{e}$-module. Then
\begin{equation}\label{eq-ABetim-bbaiso1}
B^{e}\otimes_{A^{e}}A\to B\otimes_{A}B, \, (x\otimes y)\otimes_{A^{e}} a\mapsto xa\otimes_{A}  y
\end{equation}
is an isomorphism of left $B^{e}$-modules. 

Any left $B^{e}$-module $M$ can be naturally regarded as an $A$-$A$-bimodule. Moreover, the center $Z_{A}(M)$ of $M$, viewed as an $A$-$A$-bimodule is canonically isomorphic to $\operatorname{Hom}_{A^{e}}(A,M)$ via the map $Z_{A}(M)\to \operatorname{Hom}_{A^{e}}(A,M), \, x\mapsto (a\mapsto ax)$. We use this isomorphism to identify $Z_{A}(M)$ with $\operatorname{Hom}_{A^{e}}(A,M)$. We now show that $Z_{A}(M)=\operatorname{Hom}_{A^{e}}(A,M)$ carries a canonical $H$-module structure. Indeed, there are canonical isomorphisms:
$$
\operatorname{Hom}_{A^{e}}(A,M) \cong  \operatorname{Hom}_{A^{e}}(A,\operatorname{Hom}_{B^{e}}(B^{e},M))  \cong  \operatorname{Hom}_{B^{e}}( B^{e}\otimes_{A^{e}} A,M)
 \overset{\eqref{eq-ABetim-bbaiso1}}{\cong} \operatorname{Hom}_{B^{e}}(B\otimes_{A} B,M).
$$
Hence, the $B^{e}$-$H$-bimodule structure on $B\otimes_{A}B$ induces a left $H$-module structure on 
$\operatorname{Hom}_{A^{e}}(A,M)$. This $H$-module structure 
is known as the \textit{Ulbrich-Miyashita action} \cite{MR992778,MR1358765}. In particular, for any complex $Y^{\bullet}$ of left $B^{e}$-modules that is bounded below, $\operatorname{RHom}_{A^{e}}(A,Y^{\bullet})$ is a complex of left $H$-modules.

We will later require the following result, which reformulates \textit{Stefan's spectral sequence} \cite{MR1358765} in terms of derived functors; see also \cite[Theorem 2.3]{MR5000225}.
\begin{theorem}
[\text{\cite[Theorem 3.3]{MR1358765}}]Let $A\subseteq B$ be a left $H$-Galois extension such that $B$ is flat as both a left and a right $A$-module. Then for any left $B^{e}$-module $M$, one has\label{thm-stefanspsq}
$$\operatorname{RHom}_{B^{e}}(B,M)\cong \operatorname{RHom}_{H}(\mathbbm{k},\operatorname{RHom}_{A^{e}}(A,M)).$$
\end{theorem}


\section{Homological Integrals of module-finite Hopf algebras}\label{sec-homolointTHM}

In this section, we recall background material and basic facts concerning Artin-Schelter Gorenstein Hopf algebras. We then prove Theorem \ref{thm-fghopfhomol-unimosCY}, along with the stronger result, Theorem \ref{thm-noethAS-intebimoduiso-un}. These results are illustrated with several examples. 
Throughout this section, unless otherwise stated, the base field $\mathbbm{k}$ is arbitrary.
\subsection{Artin-Schelter Gorenstein Hopf algebras and homological integrals}
In this subsection, we recall the definition and basic properties of Artin-Schelter Gorenstein Hopf algebras and their homological integrals in the sense of Lu-Wu-Zhang \cite{MR2320655}. 
\begin{definition}
[\cite{MR1482982,MR2437632}] Let $H$ be a noetherian Hopf algebra.
 \label{def-ASGore-hopf} Then $H$ is called \textit{Artin-Schelter Gorenstein} (or \textit{AS-Gorenstein} for short) of dimension $d$, if
\begin{itemize}
\item[(1)] $\operatorname{inj.dim}_{H}H=\operatorname{inj.dim}_{H^{op}}H=d$;
\item[(2)] $\dim_{\mathbbm{k}}\operatorname{Ext}_{H}^{d}(\mathbbm{k},H)=\dim_{\mathbbm{k}}\operatorname{Ext}_{H^{op}}^{d}(\mathbbm{k},H)=1$;
\item[(3)] $\text{Ext}_{H}^{i}(\mathbbm{k},H)=\operatorname{Ext}_{H^{op}}^{i}(\mathbbm{k},H)=0$ for all $i\neq d$.
\end{itemize}
If, in addition, $H$ has finite global dimension, it is called \textit{Artin-Schelter regular} (or \textit{AS-regular} for short).
\end{definition}

By \cite[Theorem 0.2 (1)]{MR1938745}, any noetherian affine PI Hopf algebra is AS-Gorenstein, which answers \cite[Question A]{MR1676211} affirmatively. 
Now the Brown-Goodearl conjecture asserts that every noetherian Hopf algebra is AS-Gorenstein.
The antipode of any noetherian AS-Gorenstein Hopf algebra is bijective \cite[Corollary 0.3]{MR3856132}. For AS-Gorenstein Hopf algebras, Lu-Wu-Zhang introduced the notion of homological integral as a homological tool for studying infinite-dimensional Hopf algebras. 
\begin{definition}
[\text{\cite[Definition 1.1]{MR2320655}}]Let $H$\label{def-homolointeLWZ} be an AS-Gorenstein Hopf algebra of dimension $d$. Any nonzero element in $\operatorname{Ext}_{H}^{d}(\mathbbm{k},H)$ (resp., $\operatorname{Ext}_{H^{op}}^{d}(\mathbbm{k},H)$) is called a left (resp., right) \textit{homological integral} of $H$. We write $\textstyle \int_{H}^{l}=\operatorname{Ext}_{H}^{d}(\mathbbm{k},H)$ (resp., $\textstyle \int_{H}^{r}=\operatorname{Ext}_{H^{op}}^{d}(\mathbbm{k},H)$), which is called the left (resp. right) \textit{homological integral space} of $H$. 
\end{definition}

In what follows, assume that $H$ is noetherian AS-Gorenstein. Let $G(H^{\circ})$ denote the set of group-like elements of the finite dual $H^{\circ}$ of the Hopf algebra $H$, which coincides with the set of characters of $H$ (i.e., algebra homomorphisms from $H$ to $\mathbbm{k}$). Clearly, $G(H^{\circ})$ forms a group under the convolution product.  Since $H$ is AS-Gorenstein, the left homological integral space $\textstyle \int_{H}^{l}$ is a $1$-dimensional $H$-$H$-bimodule with the trivial left $H$-module structure. Hence, there exists a unique character $\psi\in G(H^{\circ})$  such that \begin{equation}\label{eq-rightinchara}
x\cdot h=x\psi(h)
\end{equation}
for all $x\in \textstyle \int_{H}^{l}$ and $h\in H$. The character $\psi$ in \eqref{eq-rightinchara} is referred to as the (left) \textit{integral character}. In this paper, we consider only the integral character defined by left homological integrals; the right homological integral character is defined analogously. The \textit{integral order} of $H$, denoted by $\operatorname{io}(H)$, is the minimal positive integer $n$ such that $(\textstyle \int_{H}^{l})^{\otimes n}\cong \mathbbm{k}$ as $H$-$H$-bimodules (if no such $n$ exists, we set $\operatorname{io}(H)=+\infty$). The integral order can be defined analogously using the right homological integral space. Since the antipode of $H$ is bijective, the integral orders defined using the left  and the right homological integral spaces coincide \cite[Lemma 2.1]{MR2320655}. Moreover, $\operatorname{io}(H)$ can be computed via the integral character:
\begin{lemma}
[\text{\cite[Lemma 2.3]{MR2320655}}]Assume that $H$ is a noetherian AS-Gorenstein Hopf algebra with bijective antipode. Then $\operatorname{io}(H)$ equals the order of the integral character in $G(H^{\circ})$.\label{lem-ioHcomputed-ordergrpch}
\end{lemma}
The Hopf algebra $H$ is called \textit{unimodular} \cite[Definition 1.2]{MR2320655} if $\textstyle \int_{H}^{l}\cong \mathbbm{k}$ as $H$-$H$-bimodules. If $H$ is finite-dimensional, this definition reduces to the classical one. Clearly, $H$ is unimodular if and only if the left integral character of $H$ coincides with the counit $\varepsilon$.
By Lemma \ref{lem-ioHcomputed-ordergrpch}, $H$ is unimodular if and only if $\text{io}(H)=1$.

Recall that the left winding automorphism $W_{l}(\chi): H \to H$ and right winding automorphism  $W_{r}(\chi): H \to H$ associated to $\chi\in G(H^{\circ})$ are defined respectively by:
$$W_{l}(\chi)(h)\coloneqq \sum_{(h)}\chi(h_{(1)})h_{(2)} \text{ and\ } W_{r}(\chi)(h)\coloneqq \sum_{(h)}h_{(1)}\chi(h_{(2)}), \forall h\in H.$$
It follows from the Brown-Zhang Theorem \cite{MR2437632} that the Nakayama automorphisms of noetherian AS-Gorenstein Hopf algebras are closely related to the homological integrals.
\begin{theorem}
[\cite{MR2437632,MR3856132}] Let $H$ be a noetherian Hopf algebra. Then $H$ is AS-Gorenstein if and only if $H$ is rigid Gorenstein. In these cases, if $\chi$ is the left integral character of $H$, then $W_{l}(\chi)S^{2} =S^{2}W_{l}(\chi) $ is a Nakayama automorphism of $H$.\label{thm-BZ-nakaASgoren}
\end{theorem}
\begin{corollary}
Let $H$ be a noetherian AS-Gorenstein Hopf algebra. Then $H$ is unimodular and $S^{2}$ is inner if and only if the Nakayama automorphism of $H$ is inner.\label{cor-nakainniff-unimodular}
\end{corollary}
\begin{proof}
It suffices to show the if part. Suppose the Nakayama automorphism of $H$ is given by $\nu(h)=uhu^{-1}$ for all $h\in H$, where $u$ is a unit in $H$. 

Let $\chi $ be the left integral character of $H$. By Theorem \ref{thm-BZ-nakaASgoren}, $ \sum_{(h)}S^{2}(\chi(h_{(1)})h_{(2)})=uhu^{-1}$. Applying the counit $\varepsilon$ to both sides yields
$$\chi(h)=\sum_{(h)}\varepsilon(\chi(h_{(1)})h_{(2)})=\varepsilon(h)$$
for all $h\in H$. Therefore, $H$ is unimodular, and $S^{2}$ is inner.
\end{proof}

Let $A$ be a ring, and let $C$ be a central subring of $A$. The ring extension $C\subseteq A$ is called a \textit{Frobenius extension} if $A$ is finitely generated projective as a $C$-module and there exists an $A$-$C$-bimodule isomorphism $A\cong \operatorname{Hom}_{C}(A,C)$. In this case, $A$ is called a \textit{Frobenius $C$-algebra}. When $C\subseteq A$ is a Frobenius extension, there exists an algebra automorphism $\nu$ of $A$, unique up to inner automorphisms, such that there is an $A$-$A$ bimodule isomorphism ${^{\nu}A}\cong \operatorname{Hom}_{C}(A,C)$. Such an algebra automorphism $\nu$ is called the \textit{Nakayama automorphism} of the Frobenius extension $C\subseteq A $. The Frobenius extension $C\subseteq A$ is called \textit{symmetric} if the Nakayama automorphism is inner (or equivalently, there is an $A$-$A$-bimodule isomorphism $A\cong \operatorname{Hom}_{C}(A,C)$). For additional background on Frobenius extension, we refer the reader to \cite{MR1690111}. 

Now assume that $H$ is an affine Hopf algebra with a central Hopf subalgebra $C$ such that $H$ is a finitely generated $C$-module. Then $C\subseteq H$ is a Frobenius extension (see for example \cite[Corollary III. 4.7]{MR1898492}) and $H$ is rigid Gorenstein. Consequently, $H$ admits both a ``Nakayama automorphism'' as a Frobenius $C$-extension and a ``Nakayama automorphism'' as a rigid Gorenstein Hopf algebra. By \cite[\S 2.5]{MR2379091}, the Nakayama automorphisms arising from these two perspectives coincide. Therefore, no ambiguity arises when referring to the Nakayama automorphisms of affine Hopf algebras admitting large central Hopf subalgebras. In \cite[Question III. 4.8]{MR1898492}, Brown and Goodearl asked when the Frobenius extension $C \subseteq H$ is symmetric.  Corollary \ref{cor-nakainniff-unimodular} accordingly provides the following criterion for determining when $C \subseteq H$ is symmetric.
\begin{corollary}
Let $H$ be an affine Hopf algebra, and let $C$ be a central Hopf subalgebra of $H$ such that $H$ is a finitely generated $C$-module. Then $C \subseteq H$ is a symmetric Frobenius extension if and only if $H$ is unimodular and $S^{2}$ is inner.\label{cor-symmfrobenexif-unS2inn}
\end{corollary}
\begin{remark}\label{rmk-recoOSthm}
For any finite-dimensional Hopf algebra $H$, Corollary \ref{cor-symmfrobenexif-unS2inn} (with $C=\mathbbm{k}$) shows that $H$ is a symmetric Frobenius algebra if and only if $H$ is unimodular and $S^{2}$ is inner. This recovers a theorem of Oberst-Schneider in \cite{MR347838}.
\end{remark}


For any Hopf algebra $H$, by a classical result of Lorenz-Lorenz \cite[\S 2.4]{MR1227522}, $\operatorname{l.gl.dim}H=\operatorname{p.dim}_{H}\mathbbm{k}$,  which is also equal to $\operatorname{p.dim}_{H^{e}}H$ and $\operatorname{r.gl.dim}H$ by \cite[Proposition A.1]{MR3681116}.
Hence, for any noetherian AS-regular Hopf algebra $H$, 
$$\operatorname{gl.dim}H=\operatorname{inj.dim}_{H}H=\operatorname{p.dim}_{H^{e}}H.$$ 

The following result will be used subsequently.
\begin{lemma}
[\text{\cite[Lemma 1.3]{MR3250287}}] Let $H$ be a noetherian Hopf algebra. Then $H$ is AS-regular if and only if $H$ is skew Calabi-Yau.\label{lem-noethHopf-sCYiffasre}
\end{lemma}
\begin{remark}\label{rmk-noethCYHopf-iffasreinns-un}
Let $H$ be a noetherian Hopf algebra. It was proved in \cite[Theorem 2.3]{MR2678828} that $H$ is Calabi-Yau if and only if $H$ is unimodular, AS-regular, and $S^{2}$ is inner. We remark that this can also be derived from Corollary \ref{cor-nakainniff-unimodular} and Lemma \ref{lem-noethHopf-sCYiffasre}.
\end{remark}
There exist noetherian affine PI skew Calabi-Yau Hopf algebras that are unimodular, yet for which $S^{2}$ is not inner. See Example \ref{eg-geralizedsuzuki-algunimodular} and Remark \ref{rmk-infiSuzukalg-prenilc}.



\subsection{Proof of Theorem \ref{thm-fghopfhomol-unimosCY}}\label{subsec-pfofintrm-thm}
In this subsection, we fix a noetherian AS-Gorenstein Hopf algebra $H$ and a central Hopf subalgebra $C$ of $H$ ($H$ is not necessarily finitely generated  over $C$). Then $H$ is faithfully flat over $C$ \cite[Theorem 3.3]{MR1228767}. It follows that $C$ is noetherian. By Molnar's Theorem \cite{MR376740}, $C$ is affine. In particular, $C$ is AS-Gorenstein. Since $C$ is central, $\overline{H}\coloneqq H/C^{+}H$ is a Hopf quotient of $H$. By the proof of \cite[Proposition 3.4.3]{MR1243637}, $C\subseteq H$ is an $\overline{H}$-Galois extension. We will use the notation introduced here without further comment in subsequent sections.



We will use the following canonical action of $\overline{H}$ on the $\operatorname{Hom}_{C}$ spaces of $H$-modules.
 Let $M,N$ be left $H$-modules and let $f\in \operatorname{Hom}_{C}(M,N)$. For any $\overline{h}\in \overline{H}=H/C^{+}H$, define
\begin{equation}\label{eq-actonhomsp-MNchbar}
(\overline{h}\rightharpoonup f)(m)\coloneqq \sum_{(h)}h_{(1)}f(S(h_{(2)})m)
\end{equation}
for all $m\in M$. Clearly, \eqref{eq-actonhomsp-MNchbar} turns $\text{Hom}_{C}(M,N)$ into a left $\overline{H}$-module. From the definition of the Ulbrich-Miyashita action in subsection \ref{subsec-UMact-stefan}, the following result follows immediately.
\begin{lemma}
Retain the preceding notation. The left $\overline{H}$-module structure on $\operatorname{Hom}_{C}(M,N)$ defined in \eqref{eq-actonhomsp-MNchbar} coincides with the $\overline{H}$-module structure induced by the Ulbrich-Miyashita action on the center of $\operatorname{Hom}_{\mathbbm{k}}(M,N)$, viewed as a $C^{e}$-module.\label{lem-useaconhom-gbUMact}
\end{lemma}

Hence, for any two left $H^{e}$-modules $M$ and $N$, $\operatorname{RHom}_{C}(M,N)$ is a complex of $\overline{H}$-$H^{e}$-bimodules. When $M=\mathbbm{k}$, for each integer $i\geq 0$, the induced left $\overline{H}$-module structure on the cohomology $\operatorname{Ext}_{C}^{i}(\mathbbm{k},N)$ coincides with the canonical $\overline{H}$-module structure on $\operatorname{Ext}_{C}^{i}(\mathbbm{k},N)$ considered in \cite[Lemma 5.2.1]{MR1201697}, with $B=H$ and $A=C$. 

In Theorem \ref{thm-noethAS-intebimoduiso-un}, we will prove that $\overline{H}$ is noetherian AS-Gorenstein and establish an $H$-$H$ bimodule isomorphism between the homological integral spaces of $H$ and $\overline{H}$. This is a stronger version of Theorem \ref{thm-fghopfhomol-unimosCY} (a). Theorem \ref{thm-fghopfhomol-unimosCY} (a)-(d) are consequences of Theorem \ref{thm-modulefinitehopf-integral}, and Theorem \ref{thm-fghopfhomol-unimosCY} (e) follows from Corollary \ref{cor-cheva-0algcl-CY}.

The following proposition is a key ingredient in the proof of Theorem \ref{thm-noethAS-intebimoduiso-un}. However, it is not the only approach to proving Theorem \ref{thm-noethAS-intebimoduiso-un}; as the proof of Theorem \ref{thm-noethAS-intebimoduiso-un} below demonstrates, one may alternatively use the spectral sequence considered in \cite[Proposition 5.3]{MR1201697}.
\begin{proposition}
Let $H$ be a noetherian Hopf algebra, and let $C$ be a central Hopf subalgebra of $H$. Then, for any two left $H^{e}$-modules $M$ and $N$, one has\label{prop-usebystespmathm1}, in ${\bf D}(H^{e})$,
$$\operatorname{RHom}_{H}(M,N)\cong \operatorname{RHom}_{\overline{H}}(\mathbbm{k},\text{RHom}_{C}(M,N)).$$
In particular, for the trivial bimodule $\mathbbm{k}$, $\operatorname{RHom}_{H}(\mathbbm{k},H)\cong \operatorname{RHom}_{\overline{H}}(\mathbbm{k},\text{RHom}_{C}(\mathbbm{k},H))$.
\end{proposition}
\begin{proof}
There are canonical isomorphisms 
\begin{align*}\operatorname{RHom}_{H}(M,N)&\cong \operatorname{RHom}_{H}(H\otimes_{H}^{L}M,N)\\
&\cong \operatorname{RHom}_{H^{e}}(H,\operatorname{Hom}_{\mathbbm{k}}(M,N)).
\end{align*}
For the second isomorphism, see, for example, \cite[Lemma 2.13 (2)]{MR4413364}. Since $H$ is faithfully flat over $C$, Lemma \ref{lem-useaconhom-gbUMact} and Theorem \ref{thm-stefanspsq} imply that
\begin{align*}
\operatorname{RHom}_{H^{e}}(H,\operatorname{Hom}_{\mathbbm{k}}(M,N))&\cong \operatorname{RHom}_{\overline{H}}(\mathbbm{k},\text{RHom}_{C^{e}}(C,\operatorname{Hom}_{\mathbbm{k}}(M,N)))\\
&\cong \operatorname{RHom}_{\overline{H}}(\mathbbm{k},\operatorname{RHom}_{C}(M,N))
\end{align*}
in ${\bf D}(H^{e})$. Hence, $\operatorname{RHom}_{H}(M,N)\cong \operatorname{RHom}_{\overline{H}}(\mathbbm{k},\text{RHom}_{C}(M,N)).$
\end{proof}

The following lemma is standard.
\begin{lemma}\label{lem-homtenflat-naiso} 
Let $R$ be a ring, $X$ a finitely presented left $R$-module, $Y$ an $R$-$R$-bimodule, and $Z$ a flat left $R$-module. Then, there is a  natural isomorphism
$$\gamma:\text{Hom}_{R}(X,Y)\otimes_{R}Z\to \text{Hom}_{R}(X,Y\otimes_{R}Z), \, f\otimes z\mapsto (x\mapsto f(x)\otimes z).$$
\end{lemma}

\begin{theorem}
Let $H$ be a noetherian AS-Gorenstein Hopf algebra of injective dimension $d$, and let $C$ be a central Hopf subalgebra of injective dimension $s$. Then the following hold:\label{thm-noethAS-intebimoduiso-un}
\begin{itemize}
\item[(a)] $\overline{H}$ is noetherian AS-Gorenstein of injective dimension $d-s$.
\item[(b)] $\textstyle \int_{H}^{l}\cong \textstyle \int_{\overline{H}}^{l}$ and $  \textstyle \int_{H}^{r}\cong \textstyle \int_{\overline{H}}^{r}$ as $H$-$H$ bimodules. In particular, $\operatorname{io}(H)=\operatorname{io}(\overline{H})$.
\item[(c)] $H$ is unimodular if and only if $\overline{H}$ is unimodular.
\item[(d)] Assume that $C$ is regular. If $\overline{H}$ is skew Calabi-Yau, then so is $H$. 
Moreover, if $S^{2}$ is inner and $\overline{H}$ is Calabi-Yau, then $H$ is Calabi-Yau.
\end{itemize}
\end{theorem}
\begin{proof}
(a) and (b) Let $W$ be an arbitrary $C$-module, regarded as a symmetric $C$-$C$ bimodule. Define the left $H$-module structure on $W\otimes_{C}H$ by
 \begin{equation}\label{eq-lefact-wotimesch}
 h\cdot (w\otimes_{C}k)\coloneqq w\otimes_{C}hk
 \end{equation}
 for all $h,k\in H$ and $w\in W$. Since $C$ is a central subalgebra of $H$, the action \eqref{eq-lefact-wotimesch} is well defined.
 
Let $\mathbbm{k}$ be the trivial $H$-$H$ bimodule. By \eqref{eq-actonhomsp-MNchbar}, for any integer $i$, $\operatorname{Ext}_{C}^{i}(\mathbbm{k},H)$ has a canonical $\overline{H}$-$H$ bimodule structure. By \eqref{eq-lefact-wotimesch}, $\operatorname{Ext}_{C}^{s}(\mathbbm{k},C)\otimes_{C}H$ may be regarded as an $\overline{H}$-$H$ bimodule.
 
We claim that
\begin{equation}\label{eq-extCkHiso-exttimesH}
\operatorname{Ext}_{C}^{i}(\mathbbm{k},H)\cong 
\begin{cases}
\operatorname{Ext}_{C}^{s}(\mathbbm{k},C)\otimes_{C}H, & i= s,\\
0, &i\neq s.
\end{cases}
\end{equation}

Let $0\to C\to I^{\bullet}$ be an injective resolution of $C$ as a $C$-module. Since $H$ is a projective $C$-module, Lemma \ref{lem-homtenflat-naiso} implies that $0\to {_{C}H} (\cong C\otimes_{C}H) \to  I^{\bullet}\otimes_{C}H$ is an injective resolution of $H$ as a $C$-module. For the trivial $H$-$H$ bimodule $\mathbbm{k}$ and any $C$-module $X$, the object $\operatorname{Hom}_{C}(\mathbbm{k},X)\otimes_{C} H$ can be regarded as an $\overline{H}$-$H$ bimodule via \eqref{eq-lefact-wotimesch}. 
By Lemma \ref{lem-homtenflat-naiso}, there is an isomorphism of $\overline{H}$-$H$ bimodule complexes $\operatorname{Hom}_{C}(\mathbbm{k},I^{\bullet}\otimes_{C}H)\cong \operatorname{Hom}_{C}(\mathbbm{k},I^{\bullet})\otimes_{C}H$. Therefore, \eqref{eq-extCkHiso-exttimesH} follows.

Since $\operatorname{Ext}_{C}^{s}(\mathbbm{k},C)$ is trivial as a $C$-module, \eqref{eq-extCkHiso-exttimesH} implies that 
\begin{equation}\label{eq-EHomckhi-hbar7}
\operatorname{RHom}_{C}(\mathbbm{k},H)\cong (\mathbbm{k}\otimes_{C}H)[-s]\cong \overline{H}[-s]
\end{equation}
in ${\bf D}(\overline{H}\otimes H^{op})$.
By Proposition \ref{prop-usebystespmathm1}, we obtain
\begin{equation}\label{eq-RhomisoHHbar-key}
\begin{aligned}
\operatorname{RHom}_{H}(\mathbbm{k},H)&\cong \operatorname{RHom}_{\overline{H}}(\mathbbm{k},\text{RHom}_{C}(\mathbbm{k},H))\\
&\cong \operatorname{RHom}_{\overline{H}}(\mathbbm{k},\overline{H})[-s].
\end{aligned}
\end{equation}
Once it has been proved that $\overline{H}$ is AS-Gorenstein, taking the $d$th cohomology of both sides on \eqref{eq-RhomisoHHbar-key} yields $\operatorname{inj.dim}_{\overline{H}}\overline{H}=d-s$ and $$\textstyle \int_{H}^{l}=\operatorname{Ext}_{H}^{d}(\mathbbm{k},H)\cong \operatorname{Ext}_{\overline{H}}^{d-s}(\mathbbm{k},\overline{H})=\textstyle \int_{\overline{H}}^{l}$$ as $H$-$H$ bimodules. The right integral case follows by applying \cite[Lemma 2.1]{MR2320655} to the isomorphism $\textstyle \int_{\overline{H}}^{l}\cong \textstyle \int_{H}^{l}$. We now prove that $\overline{H}$ is AS-Gorenstein, thereby completing the proof of (a) and (b). It suffices to show that $\overline{H}$ is left AS-Gorenstein, as the right case is symmetric. By \eqref{eq-RhomisoHHbar-key}, $\operatorname{inj.dim}_{\overline{H}}\overline{H}\geq d-s$ and $\dim_{\mathbbm{k}}\operatorname{Ext}_{\overline{H}}^{k}(\mathbbm{k},\overline{H})=\delta_{k,d-s}$. It remains to show that $\operatorname{inj.dim}_{\overline{H}}\overline{H}\leq d-s$. 

Consider the injective resolution of $_{H}H$: $0\to {_{H}H}\to  J^{\bullet}$. Since $H$ is flat over $C$, the sequence $0\to {_{C}H}\to J^{\bullet}$ is an injective resolution of $_{C}H$. Now one has canonical chain isomorphisms $\operatorname{Hom}_{H}(\overline{H},J^{\bullet })\cong \operatorname{Hom}_{H}(H\otimes_{C}\mathbbm{k},J^{\bullet })\cong \operatorname{Hom}_{C}( \mathbbm{k},\operatorname{Hom}_{H}(H,J^{\bullet }))\cong \operatorname{Hom}_{C}( \mathbbm{k},J^{\bullet}) $ of left $\overline{H}$-module complexes. So $\operatorname{Ext}_{C}^{i}(\mathbbm{k},H)\cong \operatorname{End}_{H}^{i}(\overline{H},H)$ as left $\overline{H}$-modules for all $i\geq 0$. It follows from \eqref{eq-EHomckhi-hbar7} that 
\begin{equation}\label{eq-RHomhbarh}
\operatorname{RHom}_{H}(\overline{H},H)\cong \overline{H}[-s]
\end{equation}
in ${\bf D}(\overline{H})$. For any left $\overline{H}$-module $M$, one has
\begin{align*}
\operatorname{RHom}_{H}(M,H)&\cong \operatorname{RHom}_{\overline{H}}(M,\operatorname{RHom}_{H}(\overline{H},H))\\
&\cong \operatorname{RHom}_{\overline{H}}(M,\overline{H})[-s]& (\text{By }\eqref{eq-RHomhbarh}).
\end{align*}
Since $\operatorname{inj.dim}_{H}H=d$, one has $\operatorname{Ext}_{\overline{H}}^{k}(M,\overline{H})=0$, for all $k>d-s$. So $\operatorname{inj.dim}_{\overline{H}}\overline{H}\leq d-s$.

(c) This is a direct consequence of (b).

(d) Assume that $C$ is regular and $\overline{H}$ is skew Calabi-Yau. Since $C\subseteq H$ is a faithfully flat $\overline{H}$-Galois extension, $H$ is skew Calabi-Yau by \cite[Theorem 2.15 (2)]{MR5000225}. The second conclusion follows immediately from Remark \ref{rmk-noethCYHopf-iffasreinns-un}.
\end{proof}
\begin{remark}
Recall that a Hopf subalgebra $K$ of $H$ is called a \textit{normal} Hopf subalgebra if it is closed under all left and right adjoint actions; that is, $\sum_{(h)}h_{(1)}kS(h_{(2)}),\sum_{(h)}S(h_{(1)})kh_{(2)}\in K$ for all $k\in K$ and $h\in H$. Indeed, if $H$ is a noetherian AS-Gorenstein Hopf algebra of injective dimension $d$, $K$ is a normal Hopf subalgebra of $H$ such that $H$ is faithfully flat as a left and right $K$-module (which implies that $K$ is noetherian), and $K$ is AS-Gorenstein of injective dimension $s$, then the Hopf quotient $H/K^{+}H$ is a noetherian AS-Gorenstein Hopf algebra of injective dimension $d-s$. This follows by adapting the argument in (a) and applying \cite[Proposition 5.6]{2026Claincheduality}.
\end{remark}

The following example shows that the assumption that $C$ is central in Theorem \ref{thm-noethAS-intebimoduiso-un} (b) is necessary. Indeed, under the hypotheses of Theorem \ref{thm-noethAS-intebimoduiso-un}, if $C$ is  assumed only to be a commutative normal Hopf subalgebra of $H$ (rather than central), then the conclusion of Theorem \ref{thm-noethAS-intebimoduiso-un} (b) does not hold in general.
\begin{example}
[\cite{MR1676211,MR2320655}]Consider the infinite dihedral group\label{eg-infdihdralgrp}
$$\mathbb{D}=\langle g,x\mid  g^{2}=1,gxg=x^{-1}   \rangle.$$
Assume that $\text{char}\mathbbm{k}\neq 2$. Then $H=\mathbbm{k}[\mathbb{D}]$ is AS-regular of GK-dimension $1$ and $\operatorname{io}(H)=2$ \cite[Example 4.6]{MR2320655}. Now $\langle x\rangle \subseteq \mathbb{D}$ is an infinite cyclic group. Clearly, $\langle x\rangle$ is a normal subgroup of $\mathbb{D}$ and $\mathbb{D}/\langle x\rangle\cong \mathbb{Z}/2\mathbb{Z}$. So $K=\mathbbm{k}[\langle x\rangle]$ is a commutative normal Hopf subalgebra of $H$, $H$ is a finitely generated free $K$-module and $H/K^{+}H\cong \mathbbm{k}[\mathbb{Z}/2\mathbb{Z}]$. Since $\text{char}\mathbbm{k}\neq 2$, $H/K^{+}H$, as a finite-dimensional semisimple Hopf algebra, satisfies $\text{io}(H/K^{+}H)=1$. This shows that $\text{io}(H/K^{+}H)\neq \text{io}(H)$. In fact, $H$ is not a finitely generated module over any central Hopf subalgebra, as pointed out in \cite[p. 56]{MR1676211}. 
\end{example}

Recall that a group is \textit{central-by-finite} if its center has finite index. As discussed in Example \ref{eg-infdihdralgrp}, Theorem \ref{thm-noethAS-intebimoduiso-un} provides a  necessary homological condition for a finitely generated group to be central-by-finite.

\begin{proposition}
Let $G$ be a finitely generated group. If $G$ is central-by-finite, then $\text{io}(\mathbbm{k}[G])=1$.\label{prop-fggrp-io1ifcbf}
\end{proposition}
\begin{proof}
Let $Z$ denote the center of $G$. Then $\mathbbm{k}[G/Z]$ is unimodular, and $\mathbbm{k}[Z]$ is a central Hopf subalgebra of $\mathbbm{k}[G]$. Therefore, the result follows directly from Theorem \ref{thm-noethAS-intebimoduiso-un}.
\end{proof}

Since every noetherian affine PI Hopf algebra is AS-Gorenstein, the following theorem follows immediately from Theorems \ref{thm-noethAS-intebimoduiso-un} and \ref{thm-BZ-nakaASgoren}, Corollary \ref{cor-symmfrobenexif-unS2inn}, and Remark \ref{rmk-noethCYHopf-iffasreinns-un}.

\begin{theorem}\label{thm-modulefinitehopf-integral}
Let $H$ be an affine Hopf algebra, and let $C$ be a central Hopf subalgebra of $H$ such that $H$ is finitely generated as a $C$-module. Let $\overline{H}=H/C^{+}H$ be the resulting finite-dimensional Hopf quotient. Denote by $\theta\in G(\overline{H}^{\circ})$ the distinguished group-like element of $\overline{H}^{\circ}$. Then
\begin{itemize}
\item[(a)] $  \textstyle \int_{H}^{l}\cong \textstyle \int_{\overline{H}}^{l}$ and $ \textstyle \int_{H}^{r}\cong \textstyle \int_{\overline{H}}^{r}$ as $H$-$H$ bimodules. In particular, $\operatorname{io}(H)=\operatorname{io}(\overline{H})$.
\item[(b)] $H$ is unimodular if and only if $\overline{H}$ is unimodular.
\item[(c)]The extension $C \subseteq H$ is a symmetric Frobenius extension if and only if $S^{2}$ is inner and $\overline{H}$ is unimodular. In particular, $H$ is Calabi-Yau if and only if
\begin{itemize}
\item[(i)]$H$ has finite global dimension, 
\item[(ii)]$S^{2}$ is an inner automorphism of $H$,
\item[(iii)]$H/\mathfrak{m}_{\ovep}H$ is unimodular.
\end{itemize}
\item[(d)] If $\check{\theta}\in G(H^{\circ})$ is the character of $H$ induced by $\theta$, then $W_{l}(\check{\theta})S^{2}$ is the Nakayama automorphism of $H$.
\end{itemize}
\end{theorem}

\begin{remark}
An affine Hopf algebra $H$ is called \textit{commutative-by-finite} \cite{MR4201485} if it is finitely generated as a module over a commutative normal Hopf subalgebra. Example \ref{eg-infdihdralgrp} shows that Theorem \ref{thm-modulefinitehopf-integral} does not hold in general for commutative-by-finite Hopf algebras.
\end{remark}

\begin{example}
[\text{\cite[Definition 5.1]{MR4977721}}] Let\label{eg-Liu-Yualg} $\mathcal{H}_{LY}=H(e_{\pm 1},f_{\pm 1},u,v)$ be the $\mathbbm{k}$-algebra generated by $u,v,e_{i},f_{i}$ for $i\in \mathbb{Z}$, subject to the following relations
$$1=e_{0}+f_{0},\  e_{i}e_{j}=e_{i+j},\  f_{i}f_{j}=f_{i+j},\   e_{i}f_{j}=f_{j}e_{i}=0,$$
$$e_{i}u=(-1)^{i}ue_{i},\ f_{i}u=(-1)^{i}uf_{i},\ e_{i}v=(-1)^{i}ve_{i},\  f_{i}v=(-1)^{i}vf_{i},$$
$$u^{2}=v^{2}=0,\  uv=-vu$$
for any $i,j\in \mathbb{Z}$. The coalgebra structure of $\mathcal{H}_{LY}$ is defined by
$$\Delta(e_{i})=e_{i}\otimes e_{i}+f_{i}\otimes f_{-i},\  \varepsilon(e_{i})=1,$$
$$\Delta(f_{i})=e_{i}\otimes f_{i}+f_{i}\otimes e_{-i},\ \varepsilon(f_{i})=0,$$
$$\Delta(u)=1\otimes u+u\otimes e_{1}+v\otimes f_{-1},\  \varepsilon(u)=0,$$
$$\Delta(v)=1\otimes v+u\otimes f_{1}+v\otimes e_{-1},\  \varepsilon(v)=0.$$
Then $\mathcal{H}_{LY}$ is a Hopf algebra with the dual Chevalley property \cite[Lemma 5.2]{MR4977721}. The antipode of $\mathcal{H}_{LY}$ is given by
$$S(e_{i})=e_{-i},\ S(f_{i})=f_{i},\  S(u)=-vf_{-1}-ue_{-1}, \ S(v)=-uf_{1}-ve_{1}.$$
Then $S^{2}(h)=(f_{-1}+e_{1})h(f_{-1}+e_{1})^{-1}=(f_{-1}+e_{1})h(f_{1}+e_{-1})$ for all $h\in H$.

The Hopf algebra $\mathcal{H}_{LY}$ is not pointed; its coradical was computed in \cite[\S 5]{MR4977721}. Moreover, by the proof of \cite[Lemma 5.2]{MR4977721},
\begin{equation}\label{eq-pbwbasis-LYalg}
\{e_{i}u^{t}v^{s},f_{i}u^{t}v^{s}\mid  i\in \mathbb{Z}, 0\leq t,s\leq 1\}
\end{equation}
is a $\mathbbm{k}$-basis of $\mathcal{H}_{LY}$. Let $\mathcal{C}_{LY}$ be the subalgebra of $\mathcal{H}_{LY}$ generated by $e_{2i},f_{2i},i\in \mathbb{Z}$. Then $\mathcal{C}_{LY}$ is a central Hopf subalgebra of $\mathcal{H}_{LY}$. It is clear that 
\begin{align*}
\mathbbm{k}[x,x^{-1}]\oplus \mathbbm{k}[y,y^{-1}]&\to  \mathcal{C}_{LY},\\
 (p(x,x^{-1}),q(y,y^{-1})) &\mapsto p(e_{2},e_{-2})e_{0}+q(f_{2},f_{-2})f_{0}
\end{align*}
is an algebra isomorphism. By \eqref{eq-pbwbasis-LYalg}, $\mathcal{H}_{LY}$ is a free $\mathcal{C}_{LY}$-module with $\mathcal{C}_{LY}$-basis
$$\{1=e_{0}+f_{0},e_{1}+f_{1},u,v,uv, (e_{1}+f_{1})u, (e_{1}+f_{1})v, (e_{1}+f_{1})uv\}.$$
So $\overline{\mathcal{H}_{LY}}=\mathcal{H}_{LY}/\mathcal{C}_{LY}^{+}\mathcal{H}_{LY} $ is an $8$-dimensional Hopf algebra. For each $h\in \mathcal{H}_{LY}$, denote by $\overline{h}$ the image of $h$ in $\overline{\mathcal{H}_{LY}}$. Then it is direct to verify that
$$\textstyle{\int}_{\overline{\mathcal{H}_{LY}}}^{l}=\textstyle{\int}_{\overline{\mathcal{H}_{LY}}}^{r}=\mathbbm{k}(\overline{e_{1}}+1)\overline{uv}.$$
Hence $\overline{\mathcal{H}_{LY}}$ is unimodular. By Theorem \ref{thm-modulefinitehopf-integral} (b), $\mathcal{H}_{LY}$ is unimodular. Since $S^{2}$ is inner, $\mathcal{H}_{LY}$ is a symmetric Frobenius extension of $\mathcal{C}_{LY}$ by Theorem \ref{thm-modulefinitehopf-integral} (c). Observe that $\mathcal{H}_{LY}$ has nonzero nilpotent ideals (such as the principal ideal generated by $u$), therefore, by \cite[Corollary 4.2]{MR719665} and \cite[Proposition 1.6]{MR1482982}, the global dimension of $\mathcal{H}_{LY}$ is infinite.


\end{example}

\begin{example}
[Infinite Suzuki algebras, \cite{MR1406046}] Assume that the base field $\mathbbm{k}$ is algebraically closed. Let\label{eg-geralizedsuzuki-algunimodular} $n,m$ be positive integers, let $\epsilon$ be a primitive $m$-th root of unity, let $m_{1},\ldots,m_{n} \geq 2$ be positive integers such that $\epsilon^{m_{k}}=1$ for all $1\leq k\leq n$, and let $\zeta_{k}\in \mathbbm{k}^{\times}$ be a primitive $m_{k}$-th root of unity for each $k$.  Let $H=\mathfrak{S}(n,m,\epsilon,\{m_{k}\}_{k=1}^{n},\{\zeta_{k}\}_{k=1}^{n})$ be the $\mathbbm{k}$-algebra generated by $x_{1},\ldots,x_{n},g_{1},\ldots,g_{n}$, subject to the relations
$$g_{i}g_{j}=g_{j}g_{i},\  g_{k}^{m_{k}}=1,\  x_{j}g_{i}=\epsilon g_{i}x_{j},\ x_{i}g_{j}=\epsilon^{-1}g_{j}x_{i},\ x_{k}g_{k}=\zeta_{k}g_{k}x_{k},\  x_{j}x_{i}=\epsilon x_{i}x_{j}$$
for $1\leq k\leq n$ and $1\leq i<j\leq n$. The coalgebra structure on $H$ is defined by
$$\Delta(x_{i})=x_{i}\otimes g_{i}+1\otimes x_{i},\  \varepsilon(x_{i})=0,\  \Delta(g_{i})=g_{i}\otimes g_{i},\  \varepsilon(g_{i})=1$$
for $1\leq i\leq n$. It follows that $H$ is a pointed Hopf algebra, referred to as the \textit{infinite Suzuki algebra}. By viewing the $\mathbbm{k}$-algebra $H$ as a smash product of a suitable finite group algebra and a quantum affine space (see Remark \ref{rmk-infiSuzukalg-prenilc}), one obtains that
\begin{equation}\label{eq-insuzualg-pbw}
\{x_{1}^{s_{1}}\cdots x_{n}^{s_{n}}g_{1}^{t_{1}}\cdots g_{n}^{t_{n}}\mid   s_{1},\ldots,s_{n}\in \mathbb{N},  0\leq t_{k}\leq m_{k}-1\text{ for\ } k=1,2,\ldots,n\}
\end{equation}
is a $\mathbbm{k}$-basis of $H$. Set $C=\mathbbm{k}[x_{1}^{m_{1}},\ldots,x_{n}^{m_{n}}]\subseteq H$. Then $C$ is a central Hopf subalgebra of $H$. By \eqref{eq-insuzualg-pbw}, $H$ is a free $C$-module with basis
$$\{x_{1}^{s_{1}}\cdots x_{n}^{s_{n}}g_{1}^{t_{1}}\cdots g_{n}^{t_{n}}\mid    0\leq s_{k}, t_{k}\leq m_{k}-1\text{ for\ } k=1,2,\ldots,n\}.$$
It follows that $H/C^{+}H$ is precisely the finite-dimensional Hopf algebra constructed by Suzuki in \cite{MR1406046}. By \cite[Lemma 2 (b)]{MR1406046}, $H/C^{+}H$ is unimodular if and only if $\zeta_{i}=\epsilon^{2i-n-1}$ for all $1\leq i\leq n$. By Theorem \ref{thm-modulefinitehopf-integral}(b), $H$ is unimodular if and only if $\zeta_{i}=\epsilon^{2i-n-1}$ for all $1\leq i\leq n$. In \cite[Theorem 6]{MR1406046}, Suzuki showed that, for suitable choices of positive integers $n,m,m_{1},\ldots,m_{n}$ and roots of unity $\epsilon,\zeta_{1},\ldots,\zeta_{n}$, the Hopf quotient $H/C^{+}H$ is a unimodular Hopf algebra whose antipode has a non-inner square. In this case, $H$ is a unimodular noetherian affine PI Hopf algebra such that $S^{2}$ is not inner.
\end{example}

\begin{remark}\label{rmk-infiSuzukalg-prenilc}
Keep the notation of Example \ref{eg-geralizedsuzuki-algunimodular}. Define the matrix $(q_{ij})_{n\times n}\in \text{M}_{n}(\mathbbm{k} )$ by
$$q_{ij}=\begin{cases}
\epsilon^{-1},&  i<j,\\
\zeta_{i}^{-1},& i=j,\\
\epsilon, &   i>j.
\end{cases}
$$
Consider the group $G=\oplus_{i=1}^{n}\mathbb{Z}/m_{k}\mathbb{Z}$, and let $g_{1},\ldots, g_{n}$ denote its canonical generators. Namely, $g_{k}$ is the element whose $k$-th coordinate is $\overline{1}\in \mathbb{Z}/m_{k}\mathbb{Z}$ and whose remaining coordinates are zero. Let $V=\mathbbm{k}x_{1}\oplus\cdots\oplus \mathbbm{k}x_{n}$ be a $G$-graded space with $\operatorname{deg}x_{k}=g_{k}$ for $1\leq k\leq n$, equipped with a $\mathbbm{k}[G]$-action defined by $g_{i}\cdot x_{j}=q_{ij}x_{j}$ for all $1\leq i,j\leq n$. Then $V$ is a left Yetter-Drinfeld module over $\mathbbm{k}[G]$. Set $\mathcal{R}(V)=T(V)/(x_{i}x_{j}-q_{ij}x_{j}x_{i}\mid i<j)$; as an algebra, $\mathcal{R}(V)$ is the quantum affine space $\mathcal{O}_{\epsilon^{-1}}(\mathbbm{k}^{n})$ at the root of unity $\epsilon^{-1}$. Then $\mathcal{R}(V)$ is a braided Hopf algebra in $_{\mathbbm{k}[G]}^{\mathbbm{k}[G]}\mathcal{YD}$ and a pre-Nichols algebra of $V$. For some background on pre-Nichols algebras and Nichols algebras, see \S \ref{sec-largeqgrp-setup}. Indeed, the coopposite of the infinite Suzuki algebra $H=\mathfrak{S}(n,m,\epsilon,\{m_{k}\}_{k=1}^{n},\{\zeta_{k}\}_{k=1}^{n})$ is precisely the bosonization $\mathcal{R}(V)\# \mathbbm{k}[G]$. Therefore, if $\operatorname{char}\mathbbm{k}=0$, then $H$ is skew Calabi-Yau by \cite[Proposition 7.3.1]{MR4030357}. 
 
Let $\mathcal{B}(V)$ be the Nichols algebra of $V$. Then $(H/C^{+}H)^{cop}$ is the bosonization $\mathcal{B}(V)\# \mathbbm{k}[G]$ by \cite[Example 1.10.13]{MR4164719}.
\end{remark}

\begin{proposition} \label{prop-modufin-idensemiCY-char0}
Let $H$ be an affine Hopf algebra, and let $C$ be a central Hopf subalgebra of $H$ such that $H$ is finitely generated as a $C$-module and $\overline{H}=H/C^{+}H$ is semisimple. Then
\begin{itemize}
 \item[(a)] $C \subseteq H$ is a symmetric Frobenius extension. Hence $H$ is unimodular and $S^{2}$ is inner.
 \item[(b)] If the base field $\mathbbm{k}$ is of characteristic zero, then $H$ is Calabi-Yau.
 \end{itemize}
\end{proposition}
\begin{proof}
(a) Recall that $C \subseteq H$ is an $\overline{H}$-Galois extension. By \cite[Example 3.15]{MR992778}, $H$ is a separable $C$-algebra. In fact, if we take $t\in \textstyle \int_{\overline{H}}^{l}$ satisfying $\varepsilon(t)=1$, then using \eqref{eq-lGeu-botbarims7} one can directly verify that the preimage of $t\otimes 1\in \overline{H}\otimes H$ under the Galois map $\beta:H\otimes_{C}H\to \overline{H}\otimes H $ gives a separability idempotent of $H$. By the Endo-Watanabe Theorem \cite{MR227211}, $C \subseteq H$ is a symmetric Frobenius extension.

(b) By \cite[Theorem 2.9 (8)]{MR4201485}, $H$ is AS-regular. It follows from (a) that the Nakayama automorphism of $H$ is inner. So $H$ is Calabi-Yau.
\end{proof}

Recall that a Hopf algebra $H$ is said to have the \textit{Chevalley property} \cite{MR1852304} if the tensor product of any two finite-dimensional irreducible $H$-modules is completely reducible. 
Note that every irreducible module over an affine PI algebra is finite-dimensional \cite[Theorem 13.10.3 (i)]{MR1811901}.

\begin{corollary}
Let $\mathbbm{k}$ be an algebraically closed field of characteristic zero, and let $H$ be an affine semiprime Hopf algebra with a central Hopf subalgebra $C$ such that $H$ is finitely generated as a $C$-module. If $H$ has the Chevalley property, then $H$ is Calabi-Yau.\label{cor-cheva-0algcl-CY}
\end{corollary}
\begin{proof}
By \cite[Theorem 0.3 (b)]{HQWZ2026chev}, $\overline{H}=H/C^{+}H$ is semisimple. Therefore, the result follows directly from Proposition \ref{prop-modufin-idensemiCY-char0} (b).
\end{proof}




We now apply Theorem \ref{thm-modulefinitehopf-integral} to the universal enveloping algebra $\mathcal{U}(\mathfrak{g})$ of a finite-dimensional restricted Lie algebra $\mathfrak{g}$, thereby relating its homological properties 
to the integrals of the restricted universal enveloping algebra $\mathfrak{u}(\mathfrak{g})$.

\begin{example}
Assume that the base field $\mathbbm{k}$ is of characteristic $p>0$. Let $(\mathfrak{g},[p])$ be a finite-dimensional restricted Lie algebra over $\mathbbm{k}$, where $[p]:\mathfrak{g}\to\mathfrak{g}$ is the $p$-map. Denote by $\mathcal{C}(\mathfrak{g})$ the $p$-center of the universal enveloping algebra $\mathcal{U}(\mathfrak{g})$, that is, 
$$\mathcal{C}(\mathfrak{g})=\mathbbm{k}\langle x^{p}-x^{[p]}\mid x\in \mathfrak{g}\rangle\subseteq \mathcal{U}(\mathfrak{g}).$$
It is well known (see, e.g., \cite[p.114]{MR1898492}) that $\mathcal{C}(\mathfrak{g})$ is a central Hopf subalgebra of $\mathcal{U}(\mathfrak{g})$, that $\mathcal{U}(\mathfrak{g})$ is a free $\mathcal{C}(\mathfrak{g})$-module of rank $p^{\dim_{\mathbbm{k}}\mathfrak{g}}$, and that the Hopf quotient $\mathcal{U}(\mathfrak{g})/\mathcal{C}(\mathfrak{g})^{+}\mathcal{U}(\mathfrak{g})$ is precisely the \textit{restricted enveloping algebra} $\mathfrak{u}(\mathfrak{g})$ of $\mathfrak{g}$. By Theorem \ref{thm-modulefinitehopf-integral} (a), $\textstyle{\int}_{\mathcal{U}(\mathfrak{g})}^{l}\cong  \textstyle{\int}_{\mathfrak{u}(\mathfrak{g})}^{l}$ as $\mathcal{U}(\mathfrak{g})$-$\mathcal{U}(\mathfrak{g})$ bimodules. Let $\text{ad}_{x}:\mathfrak{g}\to\mathfrak{g}$ be the adjoint action of $x\in \mathfrak{g}$. By \cite[Theorem A]{MR1762922} or \cite[Proposition 6.3]{MR2437632}, $\mathcal{U}(\mathfrak{g})$ is a noetherian AS-regular Hopf algebra of dimension $\dim_{\mathbbm{k}}\mathfrak{g}$, and the integral character $\eta: \mathcal{U}(\mathfrak{g})\to\mathbbm{k}$ of $\mathcal{U}(\mathfrak{g})$ is given by
$$\eta(x)=\text{tr}(\text{ad}_{x})$$
for all $x\in \mathfrak{g}$. Thus, since $S^{2}=\text{id}$, $\mathcal{U}(\mathfrak{g})$ is Calabi-Yau if and only if $\text{tr}(\text{ad}_{x})=0$ for all $x\in \mathfrak{g}$. It follows from Theorem \ref{thm-modulefinitehopf-integral} that the following conditions are equivalent.
\begin{itemize}
\item[(i)]$\mathcal{U}(\mathfrak{g})$ is unimodular;
\item[(ii)]$\mathcal{U}(\mathfrak{g})$ is Calabi-Yau;
\item[(iii)]$\text{tr}(\text{ad}_{x})=0$ for all $x\in \mathfrak{g}$;
\item[(iv)] $\mathfrak{u}(\mathfrak{g})$ is unimodular;
\item[(v)] $\mathfrak{u}(\mathfrak{g})$ is a symmetric Frobenius algebra.  
\end{itemize}
This observation is not new. As noted above, the equivalence (i) $\Leftrightarrow$ (ii) $\Leftrightarrow$ (iii) follows directly from  \cite[Theorem A]{MR1762922}. The equivalence (iii) $\Leftrightarrow$ (v) is a classical result  of Schue \cite{MR185045}. Finally, the equivalence (iv) $\Leftrightarrow$ (v) follows from $S^{2}=\text{id}$ and Remark \ref{rmk-recoOSthm}.
\end{example}


\section{Homological properties of the quantum groups \texorpdfstring{$\mathcal{U}_{\epsilon}(\mathfrak{g}), \,\mathcal{U}_{\epsilon}^{\geq 0}(\mathfrak{g})$}{bigQGrps} and \texorpdfstring{$\mathfrak{D}_{q}(\mathfrak{g})$}{DQGrp}}\label{sec-applica-UqgUqb-Dqg}
In this section, we apply the results of the previous section to the DK quantized enveloping algebras $\mathcal{U}_{\epsilon}(\mathfrak{g})$ at roots of unity, the DKP quantized Borel subalgebras $\mathcal{U}_{\epsilon}^{\geq 0}(\mathfrak{g})$ at roots of unity, and the Drinfeld double quantum groups $\mathfrak{D}_q(\mathfrak{g})$, whether the quantum parameter $q$ is a root of unity or generic. In \S \ref{subsec-DKquanenalgsCY}, we provide a concise new proof that $\mathcal{U}_{\epsilon}(\mathfrak{g})$ is Calabi-Yau. In \S \ref{subsec-QBsgrps}, we develop an effective method for computing the Nakayama automorphism of $\mathcal{U}_{\epsilon}^{\geq 0}(\mathfrak{g})$. Finally, in \S \ref{subsec-Drindouble-Grps}, we prove that $\mathfrak{D}_{q}(\mathfrak{g})$ is Calabi-Yau.

\subsection{Big quantized enveloping algebras at roots of unity}\label{subsec-DKquanenalgsCY}
In \cite[Theorem 3.3.2]{MR2054387}, Chemla computed the rigid dualizing complex of the DK quantized enveloping algebra $\mathcal{U}_{\epsilon}(\mathfrak{g})$ at a root of unity $\epsilon$, thereby showing that $\mathcal{U}_{\epsilon}(\mathfrak{g})$ is a Calabi-Yau algebra. This result was also established by Brown-Gordon-Stroppel \cite[\S 6.4]{MR2379091}, who constructed an explicit nondegenerate symmetric $\mathcal{C}_{\epsilon}(\mathfrak{g})$-bilinear form $\mathcal{U}_{\epsilon}(\mathfrak{g})\times \mathcal{U}_{\epsilon}(\mathfrak{g})\to \mathcal{C}_{\epsilon}(\mathfrak{g})$, where $\mathcal{C}_{\epsilon}(\mathfrak{g})$, as defined in \eqref{eq-canceHopfsubalgof-Uepg}, is the canonical central Hopf subalgebra of $\mathcal{U}_{\epsilon}(\mathfrak{g})$. The aim of this subsection is to provide a new and short proof of this result using Theorem \ref{thm-modulefinitehopf-integral}.


We first fix some necessary notation \cite{MR1898492,MR1359532}. Let $\mathfrak{g}$ be a finite-dimensional complex semisimple Lie algebra of rank $n$ with a Cartan subalgebra $\mathfrak{h}$, let $\Phi$ denote the associated root system with fixed base $ \Pi=\{\alpha_{1},\ldots,\alpha_{n}\}$, and let $W$ be the Weyl group of $\Phi$. Let $(-,-)$ denote the inner product on $\oplus_{i=1}^{n}\mathbb{R}\alpha_{i}$ induced by the Killing form, and let $(a_{ij})_{n\times n}$ denote the corresponding Cartan matrix. As usual, the short roots in each irreducible component of $\Phi$ are normalized so that $(\alpha,\alpha)=2$ for each short root $\alpha$. Then $d_{i}=(\alpha_{i},\alpha_{i})/2\in \{1,2,3\}$, and the matrix $(d_{i}a_{ij})_{n\times n}$ is symmetric. Let $N\geq 3$ be an odd positive integer, coprime to $3$ if $\mathfrak{g}$ has a factor of type $G_{2}$, and let $\epsilon \in \mathbb{C}$ be a primitive $N$-th root of unity.

Recall that the DK quantized enveloping algebra $\mathcal{U}_{\epsilon}(\mathfrak{g})$ \cite{MR1103601} is the $\mathbb{C}$-algebra defined by generators
$$E_{1},\ldots,E_{n},F_{1},\ldots,F_{n},K_{1}^{\pm 1},\ldots,K_{n}^{\pm 1}$$
and the following relations:
$$K_{i}K_{j}=K_{j}K_{i}, \ K_{i}K_{i}^{-1}=K_{i}^{-1}K_{i}=1,\  E_{i}F_{j}-F_{j}E_{i}=\delta_{ij}\frac{K_{i}-K_{i}^{-1}}{\epsilon^{d_{i}}-\epsilon^{-d_{i}}},$$
$$K_{i}E_{j}K_{i}^{-1}=\epsilon^{d_{i}a_{ij}}E_{j},\  K_{i}F_{j}K_{i}^{-1}=\epsilon^{-d_{i}a_{ij}}F_{j},$$
$$\sum_{k=0}^{1-a_{ij}}(-1)^{k}\left[\begin{smallmatrix}1-a_{ij}\\k\end{smallmatrix}\right]_{\epsilon^{d_{i}}}E_{i}^{1-a_{ij}-k}E_{j}E_{i}^{k}=0,\ (i\neq j),$$
$$\sum_{k=0}^{1-a_{ij}}(-1)^{k}\left[\begin{smallmatrix}1-a_{ij}\\k\end{smallmatrix}\right]_{\epsilon^{d_{i}}}F_{i}^{1-a_{ij}-k}F_{j}F_{i}^{k}=0,\ (i\neq j),$$
where the quantum binomial coefficient $\left[\begin{smallmatrix}m\\k\end{smallmatrix}\right]_{\epsilon^{d}}$ is defined by
$$[m]_{\epsilon^{d}}= \frac{\epsilon^{md}-\epsilon^{-md}}{\epsilon^{d}-\epsilon^{-d}},\   [m]_{\epsilon^{d}}!= [m]_{\epsilon^{d}}\cdots [1]_{\epsilon^{d}}\text{ and\ } \begin{bmatrix}m\\k\end{bmatrix}_{\epsilon^{d}}= \frac{[m]_{\epsilon^{d}}!}{[k]_{\epsilon^{d}}![m-k]_{\epsilon^{d}}!}$$
for those positive integers $k\leq m$ for which the denominators above are nonzero. The quantum group $\mathcal{U}_{\epsilon}(\mathfrak{g})$ is a Hopf algebra with the coalgebra structure and antipode given by 
$$\Delta(K_{i})=K_{i}\otimes K_{i},\  \Delta(E_{i})=E_{i}\otimes 1+K_{i}\otimes E_{i},\ \Delta(F_{i})=F_{i}\otimes K_{i}^{-1}+1\otimes F_{i},$$
$$\varepsilon(K_{i})=1,\ \varepsilon(E_{i})=0,\  \varepsilon(F_{i})=0,$$
$$S(K_{i})=K_{i}^{-1},\ S(E_{i})=-K_{i}^{-1} E_{i},\  S(F_{i})=-F_{i}K_{i}$$
for $1\leq i\leq n$.

Let $w_{0}$ be the unique element of longest  length in the Weyl group $W$, and we fix a reduced expression $w_{0}=s_{i_{1}}s_{i_{2}}\cdots s_{i_{\ell}}$, where $s_{k}$ denote the simple reflection corresponding to each simple root $\alpha_{k}$. Define $\beta_{k}=s_{i_{1}}\cdots s_{i_{k-1}}(\alpha_{i_{k}}),k=1,2,\ldots,\ell$. Then $\{\beta_{1},\ldots,\beta_{\ell}\}$ is precisely the set of positive roots of $\Phi$ with respect to $\Pi$. Consider the quantum root vectors
$$E_{\beta_{k}}=T_{i_{1}}\cdots T_{i_{k-1}}(E_{i_{k}}),\ F_{\beta_{k}}=T_{i_{1}}\cdots T_{i_{k-1}}(F_{i_{k}}),$$
where $k=1,2,\ldots,\ell$. Sometimes, for convenience, the generators $E_{i},K_{i},F_{i}$ are also written as $E_{\alpha_{i}},K_{\alpha_{i}},F_{\alpha_{i}}$, respectively. By \cite[Proposition 1.7 (c)]{MR1103601},
\begin{equation}\label{eq-Uqg-PBW}
\{F_{\beta_{\ell}}^{r_{\ell}}\cdots F_{\beta_{1}}^{r_{1}} K_{1}^{t_{1}}\cdots K_{n}^{t_{n}}E_{\beta_{1}}^{m_{1}}\cdots E_{\beta_{\ell}}^{m_{\ell}}\mid t_{1},\ldots,t_{n}\in \mathbb{Z},m_{1},\ldots,m_{\ell},r_{1},\ldots,r_{\ell}\in \mathbb{N}\}
\end{equation}
 is a $\mathbb{C}$-basis of $\mathcal{U}_{\epsilon}(\mathfrak{g})$. Consider the De Concini-Kac-Procesi \cite{MR1124981} central subalgebra
\begin{equation}\label{eq-canceHopfsubalgof-Uepg}
\mathcal{C}_{\epsilon}(\mathfrak{g})\coloneqq \mathbb{C}[K_{1}^{\pm N},\ldots,K_{n}^{\pm N},E_{\beta_{1}}^{N},\ldots,E_{\beta_{\ell}}^{N},F_{\beta_{1}}^{N},\ldots,F_{\beta_{\ell}}^{N}]\subseteq \mathcal{U}_{\epsilon}(\mathfrak{g}) .
\end{equation}
This is a central Hopf subalgebra of $\mathcal{U}_{\epsilon}(\mathfrak{g})$ by \cite[Proposition 5.6(d)]{MR1124981}. Moreover, it is clear from \eqref{eq-Uqg-PBW} that $\mathcal{U}_{\epsilon}(\mathfrak{g}) $ is a free $\mathcal{C}_{\epsilon}(\mathfrak{g})$-module of rank $N^{\dim_{\mathbb{C}}\mathfrak{g}}$. Note that
$$\mathcal{U}_{\epsilon}(\mathfrak{g})/\mathcal{C}_{\epsilon}(\mathfrak{g})^{+}\mathcal{U}_{\epsilon}(\mathfrak{g})\cong \mathcal{U}_{\epsilon}(\mathfrak{g})/(E_{\beta_{i}}^{N},F_{\beta_{i}}^{N},K_{j}^{N}-1)$$
is precisely Lusztig's small quantum group $\mathfrak{u}_{\epsilon}(\mathfrak{g})$ \cite{MR1013053}.  We now show that:
\begin{proposition}
[\cite{MR2054387,MR2379091}] The quantum group $ \mathcal{U}_{\epsilon}(\mathfrak{g})$ is Calabi-Yau.\label{prop-DKquantgrp-CY}
\end{proposition}
\begin{proof}
By \cite[Proposition 2.2]{MR1482982}, $ \mathcal{U}_{\epsilon}(\mathfrak{g})$ is a noetherian domain of finite global dimension. Since $G((\mathfrak{u}_{\epsilon}(\mathfrak{g}))^{\circ})$ is trivial, $\mathfrak{u}_{\epsilon}(\mathfrak{g})$ is unimodular. Thus $\mathcal{U}_{\epsilon}(\mathfrak{g})/\mathcal{C}_{\epsilon}(\mathfrak{g})^{+}\mathcal{U}_{\epsilon}(\mathfrak{g})$ is unimodular. Now $S^{2}$ is an inner automorphism of $\mathcal{U}_{\epsilon}(\mathfrak{g})$ by \cite[\S 4.9]{MR1359532}. Hence $ \mathcal{U}_{\epsilon}(\mathfrak{g})$ is Calabi-Yau by Theorem \ref{thm-modulefinitehopf-integral} (c).
\end{proof}

\begin{remark}
Assume that $N$ is coprime to the determinant of $(a_{ij})_{n\times n}$. It is well known (see \cite[Theorem 8.23]{MR1637096}) that the Lusztig small quantum group $\mathfrak{u}_{\epsilon}(\mathfrak{g})$ is a quasi-triangular Hopf algebra and admits a ribbon structure. As mentioned in the introduction, Lyubashenko proved in \cite[Corollary A.3.3]{MR1354257} that $\mathfrak{u}_{\epsilon}(\mathfrak{g})$ is a factorizable Hopf algebra (see also \cite[Example 6.6]{MR3996323}). It follows from \cite[Theorem 1.1]{MR3996323} that the finite tensor category $\mathfrak{u}_{\epsilon}(\mathfrak{g})\text{-mod}$ is a non-semisimple MTC in the sense of Kerler-Lyubashenko \cite{MR1862634}.
\end{remark}

\subsection{Big quantized Borel subalgebras at roots of unity}\label{subsec-QBsgrps}
Retain the assumptions and notation of \S \ref{subsec-DKquanenalgsCY}. Recall that the DKP quantized Borel subalgebra $\mathcal{U}_{\epsilon}^{\geq 0}(\mathfrak{g})$ is the subalgebra of $\mathcal{U}_{\epsilon}(\mathfrak{g})$ generated by $E_{i},K_{i}^{\pm 1}, 1\leq i\leq n$. It is clear that $\mathcal{U}_{\epsilon}^{\geq 0}(\mathfrak{g})$ is a Hopf subalgebra of $\mathcal{U}_{\epsilon}(\mathfrak{g})$, and
\begin{equation}\label{eq-Uqb-PBW}
\{ K_{1}^{t_{1}}\cdots K_{n}^{t_{n}}E_{\beta_{1}}^{m_{1}}\cdots E_{\beta_{\ell}}^{m_{\ell}}\mid t_{1},\ldots,t_{n}\in \mathbb{Z},m_{1},\ldots,m_{\ell} \in \mathbb{N}\}
\end{equation}
is a $\mathbb{C}$-basis of $\mathcal{U}_{\epsilon}^{\geq 0}(\mathfrak{g})$. By \cite[Proposition 5.6(d)]{MR1124981}, 
$$\mathcal{C}_{\epsilon}^{\geq 0}(\mathfrak{g})\coloneqq \mathbb{C}[K_{1}^{\pm N},\ldots,K_{n}^{\pm N},E_{\beta_{1}}^{N},\ldots,E_{\beta_{\ell}}^{N}]\subseteq \mathcal{U}_{\epsilon}^{\geq 0}(\mathfrak{g}) $$ is a Hopf subalgebra of $\mathcal{U}_{\epsilon}^{\geq 0}(\mathfrak{g}) $. Clearly, $\mathcal{U}_{\epsilon}^{\geq 0}(\mathfrak{g}) $ is a free $\mathcal{C}_{\epsilon}^{\geq 0}(\mathfrak{g})$-module of rank $N^{n+\ell}$. 

The subalgebra of $\mathcal{U}_{\epsilon}(\mathfrak{g})$ generated by $E_{i},1\leq i\leq n$, is denoted by $\mathcal{U}_{\epsilon}^{+}(\mathfrak{g})$, which is referred to as the \textit{positive part}, or the (positive)
\textit{quantized nilpotent subalgebra}, of
$\mathcal{U}_{\epsilon}(\mathfrak{g})$. The quantized nilpotent subalgebra $\mathcal{U}_{\epsilon}^{+}(\mathfrak{g})$ constitutes an important example of a
\textit{distinguished pre-Nichols algebra} in the sense of Angiono \cite{MR3459702}; see Example \ref{eq-positpart-Lztsmaqugrp9-nichols}. It is a standard fact that
$\mathcal{U}_{\epsilon}^{\geq 0}(\mathfrak{g})$
can be realized as the bosonization of the pre-Nichols algebra
$\mathcal{U}_{\epsilon}^{+}(\mathfrak{g})$
by the group algebra $\mathbb{C}[K_{1}^{\pm 1},\ldots,K_{n}^{\pm 1}]\cong \mathbb{C}[\mathbb{Z}^{n}]$. Similarly, one defines the positive part
$\mathfrak{u}_{\epsilon}^{+}(\mathfrak{g})$
of Lusztig's small quantum group
$\mathfrak{u}_{\epsilon}(\mathfrak{g})$.

In \cite[\S 7.2]{MR2379091}, Brown-Gordon-Stroppel constructed an explicit nondegenerate $\mathcal{C}_{\epsilon}^{\geq 0}(\mathfrak{g})$-bilinear form $\mathcal{U}_{\epsilon}^{\geq 0}(\mathfrak{g})\times \mathcal{U}_{\epsilon}^{\geq 0}(\mathfrak{g})\to  \mathcal{C}_{\epsilon}^{\geq 0}(\mathfrak{g}) $ in order to compute the Nakayama automorphism of $\mathcal{U}_{\epsilon}^{\geq 0}(\mathfrak{g})$. The aim of this subsection is to show that this computation can alternatively be recovered from Theorem \ref{thm-modulefinitehopf-integral}(d).



The finite-dimensional Hopf quotient $\mathcal{U}_{\epsilon}^{\geq 0}(\mathfrak{g})/(\mathcal{C}_{\epsilon}^{\geq 0}(\mathfrak{g}) )^{+}\mathcal{U}_{\epsilon}^{\geq 0}(\mathfrak{g})$ is the small quantum Borel subgroup $\mathfrak{u}_{\epsilon}^{\geq 0}(\mathfrak{g})$, which is a basic Hopf algebra (see, for example, \cite[Lemma 5.3]{MR4891381}).
The restriction on the order $N$ of $\epsilon$ as a primitive root of unity implies that the antipode of $\mathfrak{u}_{\epsilon}^{\geq 0}(\mathfrak{g})$ satisfies $S^{4} \neq \mathrm{id}$. Therefore, \cite[Theorem 3.3]{MR1739215} implies that $\mathfrak{u}_{\epsilon}^{\geq 0}(\mathfrak{g})$ admits no quasi-triangular structure.

For any $u\in \mathcal{U}_{\epsilon}^{\geq 0}(\mathfrak{g})$, let $\overline{u}$ denote the image of $u$ in the small quantum Borel subgroup $\mathfrak{u}_{\epsilon}^{\geq 0}(\mathfrak{g})$. It follows from \eqref{eq-Uqb-PBW} that
\begin{equation}\label{eq-pbw-smallqborel}
\{ \overline{K_{1}}^{t_{1}}\cdots \overline{K_{n}}^{t_{n}}\overline{E_{\beta_{1}}}^{m_{1}}\cdots \overline{E_{\beta_{\ell}}}^{m_{\ell}}\mid 0\leq t_{1},\ldots,t_{n},m_{1},\ldots,m_{\ell}\leq N-1\}
\end{equation}
is a $\mathbb{C}$-basis of $\mathfrak{u}_{\epsilon}^{\geq 0}(\mathfrak{g})$. Set $\delta=\beta_{1}+\beta_{2}+\cdots+\beta_{\ell}$. We claim that
\begin{equation}\label{eq-intofsmallquaBorel9}
t\coloneqq \sum_{0\leq i_{1}<\cdots<i_{n}\leq N-1}\epsilon^{-(i_{1}\alpha_{1}+\cdots+i_{n}\alpha_{n},\delta)}\overline{E_{\beta_{1}}}^{N-1}\cdots \overline{E_{\beta_{\ell}}}^{N-1}\overline{K_{1}}^{i_{1}}\cdots \overline{K_{n}}^{i_{n}}\neq 0\in \textstyle{\int}_{\mathfrak{u}_{\epsilon}^{\geq 0}(\mathfrak{g})}^{l}.
\end{equation}
Since $K_{i}E_{\beta_{j}}K_{i}^{-1}=\epsilon^{(\alpha_{i},\beta_{j})}E_{\beta_{j}}$ for all $1\leq i\leq n$ and $1\leq j\leq \ell$, it follows from  \eqref{eq-pbw-smallqborel} that $t\neq 0$.
For each $1\leq s\leq n$, we have
\begin{align*}
K_{s}t &= \sum_{0\leq i_{1}<\cdots<i_{n}\leq N-1}\epsilon^{-(i_{1}\alpha_{1}+\cdots+i_{n}\alpha_{n},\delta)}K_{s}\overline{E_{\beta_{1}}}^{N-1}\cdots \overline{E_{\beta_{\ell}}}^{N-1}\overline{K_{1}}^{i_{1}}\cdots \overline{K_{n}}^{i_{n}}\\
&= \sum_{0\leq i_{1}<\cdots<i_{n}\leq N-1}\epsilon^{-(i_{1}\alpha_{1}+\cdots+i_{n}\alpha_{n},\delta)+ (\alpha_{s},(N-1)\delta)} \overline{E_{\beta_{1}}}^{N-1}\cdots \overline{E_{\beta_{\ell}}}^{N-1}\overline{K_{1}}^{i_{1}}\cdots  \overline{K_{s}}^{i_{s}+1} \cdots \overline{K_{n}}^{i_{n}}\\
&= \sum_{0\leq i_{1}<\cdots<i_{n}\leq N-1}\epsilon^{-(i_{1}\alpha_{1}+\cdots+(i_{s}+1)\alpha_{s}+\cdots+i_{n}\alpha_{n},\delta) } \overline{E_{\beta_{1}}}^{N-1}\cdots \overline{E_{\beta_{\ell}}}^{N-1}\overline{K_{1}}^{i_{1}}\cdots  \overline{K_{s}}^{i_{s}+1} \cdots \overline{K_{n}}^{i_{n}}\\
&=t.
\end{align*}
Let $Q=\mathbb{Z}\Pi$ be the root lattice. Recall that the big quantum group $\mathcal{U}_{\epsilon} (\mathfrak{g})$ has a canonical $Q$-grading \cite[\S 4.7]{MR1359532} such that for each simple root $\alpha_{i}$, one has 
$$\operatorname{deg}E_{\alpha_{i}}=\alpha_{i},\operatorname{deg}F_{\alpha_{i}}=-\alpha_{i}\text{ and\ }\operatorname{deg}K_{\alpha_{i}}=0.$$
Then $\mathcal{U}_{\epsilon}^{\geq 0} (\mathfrak{g})$ is a $Q$-graded subalgebra of $\mathcal{U}_{\epsilon} (\mathfrak{g})$. Consequently, $\mathfrak{u}_{\epsilon}^{\geq 0}(\mathfrak{g})$, being the quotient of $\mathcal{U}_{\epsilon}^{\geq 0} (\mathfrak{g})$ by an ideal generated by homogeneous elements, inherits a natural $Q$-grading. Thus, by the $Q$-grading on $\mathfrak{u}_{\epsilon}^{\geq 0}(\mathfrak{g})$ and \eqref{eq-pbw-smallqborel}, it follows immediately that
\begin{equation}\label{eq-Eitimt=0rl-cha}
E_{i}t=tE_{i}=0
\end{equation}
for all $1\leq i\leq n$. Hence \eqref{eq-intofsmallquaBorel9} holds. It is also clear that
\begin{equation}\label{eq-tKs-tocomcharsmqb}
tK_{s}=\epsilon^{(\alpha_{s},\delta)}t=\epsilon^{(\alpha_{s},\alpha_{s})}t
\end{equation}
for any $1\leq s\leq n$. The final equality follows from the fact that the reflection associated with the simple root $\alpha_{s}$ permutes $\Phi^{+}-\{\alpha_{s}\}$, where $\Phi^{+}=\{\beta_{1},\ldots,\beta_{\ell}\}$.

By \eqref{eq-Eitimt=0rl-cha} and \eqref{eq-tKs-tocomcharsmqb}, the distinguished group-like element of $(\mathfrak{u}_{\epsilon}^{\geq 0}(\mathfrak{g}))^{\circ}$, denoted by $\tau$, is given by
 $$\tau(K_{i})=\epsilon^{(\alpha_{i},\alpha_{i})},\  \tau(E_{i})=0,$$
for $i=1,2,\ldots,n$. By \cite[\S 4.9]{MR1359532}, $S^{2}(K_{i})=K_{i}$ and $S^{2}(E_{i})=\epsilon^{-(\alpha_{i},\delta)}E_{i}=\epsilon^{-(\alpha_{i},\alpha_{i})}E_{i}$ for all $1\leq i\leq n$. Thus, Theorem \ref{thm-modulefinitehopf-integral} (d) implies the following.
 \begin{proposition}
[\cite{MR2379091}] The Nakayama automorphism $\nu$ of \label{prop-nakaya-quantborel-BGS}$\mathcal{U}_{\epsilon}^{\geq 0}(\mathfrak{g})$ is given by 
$$\nu (E_{i})=E_{i}\text{ and\ }\nu(K_{i})=\epsilon^{(\alpha_{i},\alpha_{i})}K_{i}$$ for all $1\leq i\leq n$.
 \end{proposition}
 This is consistent with the formula for the Nakayama automorphism given by Brown-Gordon-Stroppel in \cite[Theorem 7.2 (3)]{MR2379091}. It should be noted that \cite{MR2379091} considers the simply connected form of the quantized Borel subalgebra, whereas the computation presented here concerns the adjoint form of the quantized Borel subalgebra. The adjoint form is adopted here solely for convenience; the same discussion applies to any sublattice between the root lattice $Q$ and the weight lattice $P$ of the root system $\Phi$.

\subsection{Drinfeld double quantum groups}\label{subsec-Drindouble-Grps}
Throughout this subsection, the base field is $\mathbb{C}$. We retain the notation introduced in \S \ref{subsec-DKquanenalgsCY}. 

Let $q\in \mathbb{C}^{\times}$ satisfy one of the following conditions:
\begin{itemize}
\item[(i)] $q$ is not a root of unity;
\item[(ii)] $q$ is a primitive $N$-th root of unity, where $N\geq 3$ is an odd positive integer, and is coprime to $3$ if $\mathfrak{g}$ has a factor of type $G_{2}$.
\end{itemize}
Recall that the (standard) Drinfeld double quantum group $\mathfrak{D}_{q}(\mathfrak{g})$ \cite{MR934283} is the $\mathbb{C}$-algebra generated by $E_{1},\ldots,E_{n},F_{1},\ldots,F_{n},K_{1}^{\pm 1},\ldots,K_{n}^{\pm 1},L_{1}^{\pm 1},\ldots,L_{n}^{\pm 1}$, subject to the relations
$$K_{i}K_{j}=K_{j}K_{i}, \  K_{i}K_{i}^{-1}=K_{i}^{-1}K_{i}=1,\  L_{i}L_{j}=L_{j}L_{i}, \  L_{i}L_{i}^{-1}=L_{i}^{-1}L_{i}=1,$$
$$\ K_{i}E_{j}K_{i}^{-1}=q^{d_{i}a_{ij}}E_{j},\   K_{i}F_{j}K_{i}^{-1}=q^{-d_{i}a_{ij}}F_{j},\ L_{i}E_{j}L_{i}^{-1}=q^{-d_{i}a_{ij}}E_{j},\   L_{i}F_{j}L_{i}^{-1}=q^{d_{i}a_{ij}}F_{j}, $$

$$ K_{i}L_{j}=L_{j}K_{i},\ E_{i}F_{j}-F_{j}E_{i}=\delta_{ij}\frac{K_{i}-L_{i}}{q^{d_{i}}-q^{-d_{i}}},$$
$$\sum_{k=0}^{1-a_{ij}}(-1)^{k}\left[\begin{smallmatrix}1-a_{ij}\\k\end{smallmatrix}\right]_{q^{d_{i}}}E_{i}^{1-a_{ij}-k}E_{j}E_{i}^{k}=0,\ (i\neq j),$$
$$\sum_{k=0}^{1-a_{ij}}(-1)^{k}\left[\begin{smallmatrix}1-a_{ij}\\k\end{smallmatrix}\right]_{q^{d_{i}}}F_{i}^{1-a_{ij}-k}F_{j}F_{i}^{k}=0,\ (i\neq j),$$
where $1\leq i,j\leq n$. The Hopf algebra structure of $\mathfrak{D}_{q}(\mathfrak{g})$ is given by
\begin{equation}\label{eq-coalgDrindouquG1}
\Delta(E_{i})=E_{i}\otimes 1+K_{i}\otimes E_{i},\  \Delta(F_{i})=F_{i}\otimes L_{i}+1\otimes F_{i},
\end{equation}
\begin{equation}\label{eq-coalgDrindouquG2}
\Delta(K_{i})=K_{i}\otimes K_{i},\  \Delta(L_{i})=L_{i}\otimes L_{i}, \varepsilon(E_{i})=\varepsilon(F_{i})=0,\ \varepsilon(K_{i})=\varepsilon(L_{i})=1,
\end{equation}
\begin{equation}\label{eq-coalgDrindouquG3}
S(E_{i})=-K_{i}^{-1}E_{i},\  S(F_{i})=-F_{i}L_{i}^{-1},\ S(K_{i})=K_{i}^{-1},\ S(L_{i})=L_{i}^{-1}, 
\end{equation}
for $1\leq i\leq n$. By \cite[Corollary 5.4.4]{MR4164719}, $\mathfrak{D}_{q}(\mathfrak{g})$ is a pointed Hopf algebra. 

It is well known that $\mathfrak{D}_{q}(\mathfrak{g})$ is the quantum double of $\mathcal{U}_{q}^{\leq 0}(\mathfrak{g})$ and $\mathcal{U}_{q}^{\geq 0}(\mathfrak{g})$; see \cite[Proposition 8.1.6]{MR4164719} or \cite[\S 3.2]{MR1315966}. It follows that
\begin{equation}\label{eq-DDGrpDqg-PBW}
\{F_{\beta_{\ell}}^{r_{\ell}}\cdots F_{\beta_{1}}^{r_{1}}L_{1}^{s_{1}}\cdots L_{n}^{s_{n}} K_{1}^{t_{1}}\cdots K_{n}^{t_{n}}E_{\beta_{1}}^{m_{1}}\cdots E_{\beta_{\ell}}^{m_{\ell}}\mid s_{i},t_{i}\in \mathbb{Z},m_{j},r_{j}\in \mathbb{N}\}
\end{equation}
forms a $\mathbb{C}$-basis of $\mathfrak{D}_{q}(\mathfrak{g})$. Using \eqref{eq-DDGrpDqg-PBW}, a filtered argument analogous to those in \cite[Theorem 19]{MR3459702} and \cite[Proposition 10.1]{MR1288995} shows that $\mathfrak{D}_{q}(\mathfrak{g})$ is a noetherian domain.

Set $\mathfrak{T}_{q}(\mathfrak{g})=\mathbb{C}[(K_{1}L_{1})^{\pm 1},(K_{2}L_{2})^{\pm 1},\ldots,(K_{n}L_{n})^{\pm 1}]\subseteq \mathfrak{D}_{q}(\mathfrak{g})$. Then $\mathfrak{T}_{q}(\mathfrak{g})$ is a central Hopf subalgebra of $\mathfrak{D}_{q}(\mathfrak{g})$. Note that (see, e.g., \cite[Corollary 8.1.7]{MR4164719}) 
\begin{equation}\label{eq-Drinfeqdou-quo-Uqg}
\mathfrak{D}_{q}(\mathfrak{g})/(\mathfrak{T}_{q}(\mathfrak{g}))^{+}\mathfrak{D}_{q}(\mathfrak{g})\cong \mathcal{U}_{q}(\mathfrak{g})
\end{equation}
as Hopf algebras. We now apply Theorem \ref{thm-noethAS-intebimoduiso-un} to show that $\mathfrak{D}_{q}(\mathfrak{g})$ is Calabi-Yau.

\begin{theorem}
Let $\mathfrak{g}$ be a complex semisimple Lie algebra of rank $n$, and retain the preceding assumptions on $q\in \mathbb{C}^{\times}$. Then $\mathfrak{D}_{q}(\mathfrak{g})$ is a noetherian Calabi-Yau Hopf algebra.\label{thm-DDQGrp-CY}
\end{theorem}
\begin{proof}
Since $\mathfrak{T}_{q}(\mathfrak{g})$ is a Laurent polynomial algebra in $n$ variables, 
it is regular. It follows from \cite[Corollary 3.1.2 and Theorem 3.3.2]{MR2054387} that $\mathcal{U}_{q}(\mathfrak{g})$ is Calabi-Yau. Therefore, since $\mathfrak{D}_{q}(\mathfrak{g})$ is a faithfully flat $\mathcal{U}_{q}(\mathfrak{g})$-Galois extension of $\mathfrak{T}_{q}(\mathfrak{g})$, it is a noetherian skew Calabi-Yau Hopf algebra \cite[Theorem 2.15 (2)]{MR5000225}. By Lemma \ref{lem-noethHopf-sCYiffasre}, $\mathfrak{D}_{q}(\mathfrak{g})$ is AS-regular. 

By the argument in \cite[\S 4.9]{MR1359532}, $S^{2}$ is inner. Therefore, by \eqref{eq-Drinfeqdou-quo-Uqg} and Theorem \ref{thm-noethAS-intebimoduiso-un} (d), $\mathfrak{D}_{q}(\mathfrak{g})$ is a Calabi-Yau Hopf algebra.
\end{proof}
\begin{remark}\label{rmk-multquanenealg-generunim}
We note that, when the quantum parameter $q$ is not a root of unity, the unimodularity of $\mathfrak{D}_{q}(\mathfrak{g})$ can also be deduced from the results of Yu-Zhang in \cite{MR3148617}. In fact, we will show that all generic multiparameter quantum enveloping algebras \cite[Definition 7]{MR2642565} of complex semisimple Lie algebras are noetherian unimodular Hopf algebras. Consequently, the Drinfeld double quantum group $\mathfrak{D}_{q}(\mathfrak{g})$, being a special case of the two-parameter quantum group $\mathcal{U}_{q,q^{-1}}(\mathfrak{g})$ (see, for example, \cite[Remark 9 (1)]{MR2642565}), is unimodular.

Below, we outline a proof of the unimodularity of generic multiparameter quantum groups, which, to the best of our knowledge, has not previously been documented. 

Let $\mathfrak{q}=(q_{ij})_{n\times n}$ be a complex matrix whose entries are all nonzero. Suppose that $q_{ij}q_{ji}=q_{ii}^{a_{ij}}$ and $q_{ii}^{k}\neq 1$ for all $1\leq k\leq  -a_{ij}$ and $1\leq i\neq j\leq n$. Recall that $(a_{ij})_{n\times n}$ is the Cartan matrix of the fixed complex semisimple Lie algebra $\mathfrak{g}$. The matrix $\mathfrak{q}$ is said to be \textit{generic} if none of its entries is a root of unity. The multiparameter quantum group $\mathcal{U}_{\mathfrak{q}}(\mathfrak{g})$ is the $\mathbb{C}$-algebra generated by
$$E_{1},\ldots,E_{n},F_{1},\ldots,F_{n},K_{1}^{\pm 1},\ldots,K_{n}^{\pm 1},L_{1}^{\pm 1},\ldots,L_{n}^{\pm 1},$$
subject to the following relations:
$$K_{i}K_{j}=K_{j}K_{i}, \  K_{i}K_{i}^{-1}=K_{i}^{-1}K_{i}=1,\  L_{i}L_{j}=L_{j}L_{i}, \  L_{i}L_{i}^{-1}=L_{i}^{-1}L_{i}=1,$$
$$\ K_{i}E_{j}K_{i}^{-1}=q_{ij}E_{j},\   K_{i}F_{j}K_{i}^{-1}=q_{ij}^{-1}F_{j},\ L_{i}E_{j}L_{i}^{-1}=q_{ji}^{-1}E_{j},\   L_{i}F_{j}L_{i}^{-1}=q_{ji}F_{j}, $$

$$ K_{i}L_{j}=L_{j}K_{i},\ E_{i}F_{j}-F_{j}E_{i}=\delta_{ij}q_{ii}\frac{K_{i}-L_{i}}{q_{ii}-1},$$
$$\sum_{k=0}^{1-a_{ij}}(-1)^{k}\left(\begin{smallmatrix}1-a_{ij}\\k\end{smallmatrix}\right)_{q_{ii}}q_{ii}^{\frac{k(k-1)}{2}}q_{ij}^{k}E_{i}^{1-a_{ij}-k}E_{j}E_{i}^{k}=0,\ (i\neq j),$$
$$\sum_{k=0}^{1-a_{ij}}(-1)^{k}\left(\begin{smallmatrix}1-a_{ij}\\k\end{smallmatrix}\right)_{q_{ii}}q_{ii}^{\frac{k(k-1)}{2}}q_{ij}^{k}F_{i}^{k}F_{j}F_{i}^{1-a_{ij}-k}=0,\ (i\neq j),$$
where $1\leq i,j\leq n$, and the quantum binomial coefficient $\left(\begin{smallmatrix}m\\k\end{smallmatrix}\right)_{q^{d}}$ is defined by
$$(m)_{q}= \frac{q^{m}-1}{q-1},\   (m)_{q}!= (m)_{q}\cdots (1)_{q}\text{ and\ } \begin{pmatrix}m\\k\end{pmatrix}_{q }= \frac{(m)_{q}!}{(k)_{q}!(m-k)_{q}!}$$
for those positive integers $k\leq m$ for which the denominators above are nonzero. The Hopf algebra structure of $\mathcal{U}_{\mathfrak{q}}(\mathfrak{g})$ is given by \eqref{eq-coalgDrindouquG1}-\eqref{eq-coalgDrindouquG3}. 

For instance, when $(q_{ij})_{n\times n}=(q^{d_{i}a_{ij}})_{n\times n}$, the corresponding multiparameter quantum group coincides with $\mathfrak{D}_{q}(\mathfrak{g})$. Notably, every multiparameter quantized enveloping algebra $\mathcal{U}_{\mathfrak{q}}(\mathfrak{g})$ is a $2$-cocycle deformation of $\mathfrak{D}_q(\mathfrak g)$ \cite[Theorem 28]{MR2642565}.

Let $\mathcal{D}$ be a YD-datum of finite Cartan type, let $\lambda$ be a linking parameter for $\mathcal{D}$ (see \cite[\S 8.2]{MR4164719} for background) and let $U(\mathcal{D},\lambda)$ be the pointed Hopf algebra defined in \cite[Definition 8.3.1]{MR4164719}. By \cite[Proposition 8.4.5 and Lemma 8.4.9]{MR4164719} and \cite[Theorem 3.8]{MR3583293}, $\mathcal{U}_{\mathfrak{q}}(\mathfrak{g})$ can be viewed as a special case of $U(\mathcal{D},\lambda)$, as follows. 

Let $G=\mathbb{Z}K_{1}\oplus\cdots \oplus \mathbb{Z}K_{n}\oplus \mathbb{Z}L_{1}\oplus\cdots\oplus \mathbb{Z}L_{n}$ be a free abelian group of rank $2n$. For each $1\leq j\leq n$, define a character $\chi_{j}:G\to \mathbb{C}^{\times}$ by $\chi_{j}(K_{i})=q_{ij},  \chi_{j}(L_{i})=q_{ji} $. Consider the YD-datum $\mathcal{D}_{MQG}$ defined by the abelian group $G$, together with $(K_{1},\ldots,K_{n},L_{1},\ldots,L_{n})$ and $(\chi_{1},\ldots,\chi_{n},\chi_{1}^{-1},\ldots,\chi_{n}^{-1})$. Let $\widetilde{A}=\operatorname{diag}\{(a_{ij})_{n\times n},(a_{ij})_{n\times n}\}$ be the block diagonal matrix of size $2n$. Clearly, $\widetilde{A}$ is also a Cartan matrix of finite type.  Define $\lambda_{MQG}\in \text{M}_{2n}(\mathbb{C})$ as
$$\lambda_{MQG}=\begin{pmatrix}
0 & 0 & \cdots & 0 & \frac{q_{11}}{q_{11}-1} & 0 & \cdots & 0 \\
0 & 0 & \cdots & 0 & 0 & \frac{q_{22}}{q_{22}-1} & \cdots & 0 \\
\vdots & \vdots & \ddots & \vdots & \vdots & \vdots & \ddots & \vdots \\
0 & 0 & \cdots & 0 & 0 & 0 & \cdots & \frac{q_{nn}}{q_{nn}-1}\\[6pt]
\gamma_{1} & 0 & \cdots & 0 & 0 & 0 & \cdots & 0 \\
0 & \gamma_{2} & \cdots & 0 & 0 & 0 & \cdots & 0 \\
\vdots & \vdots & \ddots & \vdots & \vdots & \vdots & \ddots & \vdots \\
0 & 0 & \cdots & \gamma_{n} & 0 & 0 & \cdots & 0
\end{pmatrix},$$
where 
$\gamma_{j}=-q_{jj}^{2}/(q_{jj}-1)$ if $n+j$ and $j$ are not equivalent under the canonical equivalence relation on the index set $\{1,2,\ldots,2n\}$ induced by the Cartan matrix $\widetilde{A}$. Then, as Hopf algebras, $\mathcal{U}_{\mathfrak{q}}(\mathfrak{g})\cong U(\mathcal{D}_{MQG},\lambda_{MQG})$ \cite[Theorem 3.8]{MR3583293}. 

In \cite{MR3148617}, Yu and Zhang proved that, for any generic YD-datum $\mathcal{D}$ of finite Cartan type and any linking parameter $\lambda$ for $\mathcal{D}$, $U(\mathcal{D},\lambda)$ is a noetherian AS-regular Hopf algebra. They also gave a formula for its (homological) integral character.
If the matrix $\mathfrak{q}$ is generic, it follows readily from \cite[Theorem 2.2]{MR3148617} that the integral character of $U(\mathcal{D}_{MQG},\lambda_{MQG})$ is precisely its counit. 
Thus, all generic multiparameter quantum groups $\mathcal{U}_{\mathfrak{q}}(\mathfrak{g})$ are unimodular. Moreover, if there exist integers $t_{1},t_{2},\ldots,t_{n}$ such that
\begin{equation}\label{eq-CYformultqua-grp9}
\prod_{k=1}^{n}q_{ki}^{t_{k}}=q_{ii}^{-1}
\end{equation}
for all $1\leq i\leq n$, then it follows directly that $S^{2}$ is the inner automorphism given by $K_{1}^{t_{1}}\cdots K_{n}^{t_{n}}$. Therefore, any generic multiparameter quantum group $\mathcal{U}_{\mathfrak{q}}(\mathfrak{g})$ satisfying condition \eqref{eq-CYformultqua-grp9} for all $1\leq i\leq n$ 
is a noetherian Calabi-Yau Hopf algebra, by Remark \ref{rmk-noethCYHopf-iffasreinns-un}. For instance, if $(q_{ij})_{n\times n}=(q^{d_{i}a_{ij}})_{n\times n}$, then condition \eqref{eq-CYformultqua-grp9} is satisfied for every $1\leq i\leq n$. This provides an alternative proof that $\mathfrak{D}_{q}(\mathfrak{g})$
is Calabi-Yau.
\end{remark}

In \cite{MR4506530}, Garc\'ia-Gavarini constructed restricted and unrestricted integral forms of
$\mathcal{U}_{\mathfrak{q}}(\mathfrak{g})$, which were used to define the corresponding multiparameter quantum groups at roots of unity. The homological properties of unrestricted multiparameter quantum groups at roots of unity are discussed in the next section. In Theorem \ref{thm-largqungrp-CY}, we establish a criterion characterizing precisely when a large quantum group, in the sense of Andruskiewitsch-Angiono-Yakimov, is a Calabi-Yau algebra. This criterion applies to a large class of unrestricted multiparameter quantum groups at roots of unity (see Example \ref{eg-multiparamqgrp-AAYlarge}), including the Drinfeld double quantum group $\mathfrak{D}_{\epsilon}(\mathfrak{g})$, where $\epsilon$ is a root of unity. 


\section{Homological properties of the Andruskiewitsch-Angiono-Yakimov large quantum groups}\label{sec-AAYlargequantumGrp}
In \cite{MR4600057}, Andruskiewitsch-Angiono-Yakimov introduced a large class of Hopf algebras, termed \textit{large quantum groups}, defined as quantum doubles of bosonizations of all distinguished pre-Nichols algebras \cite{MR3459702} within a one-parameter family, with the aim of providing a unified approach to the representation theory of all contragredient quantum supergroups at roots of unity. The AAY large quantum groups include, as special cases, numerous multiparameter versions of the DK big quantum groups and big quantum supergroups at roots of unity. In this section, we apply Theorems \ref{thm-noethAS-intebimoduiso-un} and \ref{thm-modulefinitehopf-integral} to study these large quantum groups from a homological perspective.

Throughout this section, the base field is $\mathbb{C}$.

\subsection{Setup data and the large quantum groups \texorpdfstring{$U_{\mathfrak{q}}$}{large QGrps}}\label{sec-largeqgrp-setup}

For the reader's convenience, this subsection recalls necessary preliminaries on Nichols algebras of diagonal type before introducing 
large quantum groups in the sense of Andruskiewitsch-Angiono-Yakimov. The main references are \cite{MR4600057,MR3459702,MR3736568,MR4164719}. 

Recall that a \textit{braided vector space} $(V,c)$ is a pair consisting of a vector space $V$ and a linear automorphism $c:V\otimes V\to V\otimes V$ satisfying
$$(c\otimes \text{id})(\text{id}\otimes c)(c\otimes \text{id})=  (\text{id}\otimes c)(c\otimes \text{id})(\text{id}\otimes c)\in  \text{End}(V^{\otimes 3}).$$
It is well known (see \cite[Lemma 1.7.5]{MR4164719}) that for any positive integer $n\geq 2$, a braided vector space $(V,c)$ naturally gives rise to a representation of the braid group $\mathbb{B}_{n}$ on $V^{\otimes n}$.

Let $\theta$ be a positive integer, and let $(V,c)$ be a $\theta$-dimensional braided vector space. The pair $(V,c)$ is of \textit{diagonal type} if it has a basis $\{x_{i} \}_{i=1}^{\theta}$ such that $$c(x_{i}\otimes x_{j})=q_{ij}(x_{j}\otimes x_{i})$$
for all $1\leq i,j\leq \theta$, where $q_{ij}\in \mathbb{C}^{\times}$. The matrix $(q_{ij})_{\theta\times \theta}$ is called the \textit{braiding matrix} of $(V,c)$ with respect to the basis $\{x_{i} \}_{i=1}^{\theta}$. The braiding matrix $(q_{ij})_{\theta\times \theta}$ is said to be of \textit{Cartan type} if there exists a generalized Cartan matrix $(a_{ij})_{\theta\times \theta}$ such that $q_{ij}q_{ji}=q_{ii}^{a_{ij}}$ and $0\leq -a_{ij}<\text{ord}(q_{ii})$ for all $1\leq i\neq j\leq \theta$. If, moreover, $(a_{ij})_{\theta\times \theta}$ is of finite type, then $(q_{ij})_{\theta\times \theta}$ is said to be of \textit{finite Cartan type}. For instance, the matrix $\mathfrak{q}$ appearing in the definition of the multiparameter quantum group $\mathcal{U}_{\mathfrak{q}}(\mathfrak{g})$, as considered in Remark \ref{rmk-multquanenealg-generunim}, is of finite Cartan type.

Let $H$ be a Hopf algebra with bijective antipode. We denote by $_{H}^{H}\mathcal{YD}$ the category of left Yetter-Drinfeld modules over $H$ (for the definition, see for example \cite[\S 4.1]{MR4164719}), which has a natural monoidal structure. For any $V,W\in {_{H}^{H}\mathcal{YD}}$, the isomorphism
$$c_{V,W}^{\mathcal{YD}}:V\otimes W\to W\otimes V, v\otimes w\mapsto \sum v_{(-1)}w\otimes v_{(0)}$$
endows $_{H}^{H}\mathcal{YD}$ with a braiding (see also \cite[Theorem 4.1.3]{MR4164719}). In particular, for every Yetter-Drinfeld module $V$, $(V,c_{V,V}^{\mathcal{YD}})$ is a braided vector space. Henceforth, unless otherwise specified, each Yetter-Drinfeld module will be regarded as a braided vector space equipped with the canonical braiding $c_{V,V}^{\mathcal{YD}}$. It is well known that every $\theta$-dimensional braided vector space of diagonal type can be realized as a left Yetter-Drinfeld module over the group algebra $\mathbb{C}[\mathbb{Z}^{\theta}]$ in a natural way, see \eqref{eq-YDmodustr-groualgdiagtype} below.

For the definition and basic notions of Hopf algebras in $_{H}^{H}\mathcal{YD}$, see \cite[\S 4.1]{MR4164719}. Let $R$ be a Hopf algebra in $_{H}^{H}\mathcal{YD}$. For any $x,y\in R$, the \textit{braided adjoint action} of $x$ on $y$ is defined by
$$(\text{ad}_{c}x)y\coloneqq \mu_{R}(\mu_{R}\otimes S_{R})(\text{id}\otimes c_{R,R})(\Delta_{R}\otimes \text{id})(x\otimes y),$$
where $\mu_{R}$, $\Delta_{R}$, and $S_{R}$ denote the multiplication,
comultiplication, and antipode of $R$, respectively, and $c_{R,R}$
denotes the braiding on $R\otimes R$.

Let $V\in {_{H}^{H}\mathcal{YD}}$. Then the tensor algebra $T(V)=\oplus_{n\geq 0}T^{n}(V)$ naturally admits the structure of a (braided) $\mathbb{N}$-graded Hopf algebra in ${_{H}^{H}\mathcal{YD}}$. By \cite[Theorem 1.3.16]{MR4164719}, there exists a unique largest coideal contained in $\oplus_{n\geq 2} T^{n}(V)$, which we denote by $\mathcal{J}(V)$. As shown in the proof of \cite[Corollary 1.6.15]{MR4164719}, $\mathcal{J}(V)$ is a braided graded Hopf ideal of $T(V)$. The resulting braided graded Hopf quotient $\mathcal{B}(V)=T(V)/\mathcal{J}(V)$ is called the \textit{Nichols algebra} of $V$. In fact, the underlying graded algebra and coalgebra structures of the Nichols algebra $\mathcal{B}(V)$ depend only on the braided vector space $(V,c_{V,V}^{\mathcal{YD}})$. One can also define the notion of a Nichols algebra for any braided vector space, see \cite[Definition 7.1.1]{MR4164719}. The Nichols algebras of braided vector spaces of diagonal type are often referred to simply as \textit{Nichols algebras of diagonal type}. A \textit{pre-Nichols algebra} of $V$ is a quotient of $T(V)$ by a braided graded Hopf ideal contained in $\oplus_{n\geq 2} T^{n}(V)$.


Fix a positive integer $\theta$, and let $\mathbb{I}=\{1,2,\ldots,\theta\}$. Let $\{\alpha_{i}\}_{i=1}^{\theta}$ be the canonical basis of the free abelian group $\mathbb{Z}^{\theta}$. Let $(V,c)$ be a $\theta$-dimensional braided vector space of diagonal type, and let $\mathfrak{q}=(q_{ij})_{\theta\times \theta}\in \text{M}_{\theta}(\mathbb{C})$ be the braiding matrix with respect to a basis $\{x_{i} \}_{i=1}^{\theta}$. Then $V$ has a natural realization as a left Yetter-Drinfeld module over the group algebra $\mathbb{C}[\mathbb{Z}^{\theta}]$, with action and coaction given by
\begin{equation}\label{eq-YDmodustr-groualgdiagtype}
\alpha_{i}\cdot x_{j}=q_{ij}x_{j},\  \operatorname{deg}x_{j}=\alpha_{j},\   i,j\in \mathbb{I}.
\end{equation}
The tensor algebra $T(V)$ has a canonical $\mathbb{Z}^{\theta}$-grading with $\operatorname{deg}x_{i}=\alpha_{i}$, and the matrix $\mathfrak{q}$ uniquely determines a  $\mathbb{Z}$-bilinear form $\mathfrak{q}:\mathbb{Z}^{\theta}\times \mathbb{Z}^{\theta}\to\mathbb{C}^{\times}$ satisfying $\mathfrak{q}(\alpha_{i},\alpha_{j})=q_{ij}$ for any $i,j\in \mathbb{I}$. By \cite[Corollary 7.1.15 (1)]{MR4164719}, the canonical $\mathbb{Z}^{\theta}$-grading on $T(V)$ descends to the Nichols algebra $\mathcal{B}(V)$. For any $\alpha,\beta\in \mathbb{Z}^{\theta}$, set $\mathfrak{q}_{\alpha\beta}=\mathfrak{q}(\alpha,\beta)$. Let $N_{\alpha}=\text{ord}(\mathfrak{q}_{\alpha\alpha})$ denote the order of $\mathfrak{q}_{\alpha\alpha}$ in $\mathbb{C}^{\times}$, which is not necessarily finite. For $\alpha=\alpha_{i}$, we write $N_{i}$ in place of $N_{\alpha_{i}}$. 

Assume that $\dim_{\mathbb{C}}\mathcal{B}(V)<+\infty$. All braiding matrices whose associated Nichols algebras are finite-dimensional were classified in \cite{MR2462836}. An explicit list of defining relations for $\mathcal{B}(V)$ is given in \cite[Theorem 3.1]{MR3181554}. Consequently, $\mathcal{J}(V)$ is generated by the relations given in \cite[Theorem 3.1]{MR3181554}. Since $\dim_{\mathbb{C}}\mathcal{B}(V)<+\infty$ and we work over $\mathbb{C}$, it follows from 
\cite[Remark 1.6.19 and Example 1.10.1]{MR4164719} that $q_{ii}\neq 1$ is a root of unity for all $i\in \mathbb{I}$.

 For $i,j\in \mathbb{I}$, define
$$c_{ij}^{\mathfrak{q}}\coloneqq  \begin{cases}
2, & i=j,\\
-\text{min}\left\{
n\in\mathbb{N}
\mid
(n+1)_{q_{ii}}
\left(1-q_{ii}^{n}q_{ij}q_{ji}\right)=0
\right\},
& i\neq j.
\end{cases}
$$
It follows from \cite[Lemmas 13.4.4 and 15.1.5 (1)]{MR4164719} that $(c_{ij}^{\mathfrak{q}})_{\theta\times \theta}\in \text{M}_{\theta}(\mathbb{Z})$ is a generalized Cartan matrix. An index $i\in \mathbb{I}$ is called a \textit{Cartan vertex} of $\mathfrak{q}$ if it satisfies $$q_{ij}q_{ji}=q_{ii}^{c_{ij}^{\mathfrak{q}}}$$ for all $j\in \mathbb{I}$ with $j\neq i$ \cite[Definition 2.6]{MR3181554}.

By \cite[Lemma 15.1.12]{MR4164719}, $\mathfrak{q}$ is of Cartan type if and only if every $i\in \mathbb{I}$ is a Cartan vertex. If $\mathfrak{q}$ is of Cartan type, the finite-dimensionality of $\mathcal{B}(V)$ implies that $(c_{ij}^{\mathfrak{q}})_{\theta\times \theta}$ is a Cartan matrix of finite type. See also \cite[Theorem 15.1.14]{MR4164719}.

For each $i\in \mathbb{I}$, define the reflection $s_{i}^{\mathfrak{q}}:\mathbb{Z}^{\theta}\to \mathbb{Z}^{\theta}$ by
$$s_{i}^{\mathfrak{q}}(\alpha_{j})\coloneqq\alpha_{j}-c_{ij}^{\mathfrak{q}}\alpha_{i}$$
for all $j\in \mathbb{I}$. Clearly, $s_{i}^{\mathfrak{q}}(\alpha_{i})=-\alpha_{i}$ and $(s_{i}^{\mathfrak{q}})^{2}=\text{id}$. Define the matrix $\rho_{i}(\mathfrak{q})$ by 
$$(\rho_{i}(\mathfrak{q}))_{jk}\coloneqq \mathfrak{q}(s_{i}^{\mathfrak{q}}(\alpha_{j}),s_{i}^{\mathfrak{q}}(\alpha_{k}))$$
for $j,k\in \mathbb{I}$. Then $\rho_{i}(\mathfrak{q})$ naturally gives rise to a braided vector space $\rho_{i}(V)$ of diagonal type with braiding matrix $\rho_{i}(\mathfrak{q})$. By \cite[Lemma 13.5.19]{MR4164719}, $c_{ij}^{\rho_{i}(\mathfrak{q})}=c_{ij}^{\mathfrak{q}}$ for all $j\in\mathbb{I}$. Hence, $\rho_{i}^{2}(\mathfrak{q})=\mathfrak{q}$. The set $\mathcal{X}=\{\rho_{j_{1}}\cdots \rho_{j_{n}}(\mathfrak{q})\mid  j_{1},\ldots,j_{n}\in \mathbb{I}, n\geq 1\}$ is called the \textit{Weyl-equivalence class} of $\mathfrak{q}$. By \cite[Proposition 13.6.4 and Lemma 15.1.8 (1)]{MR4164719}, for any matrix $\rho_{j_{1}}\cdots \rho_{j_{n}}(\mathfrak{q})\in \mathcal{X}$, one has
\begin{equation}\label{eq-weylequclas-nichfd-ref}
\dim_{\mathbb{C}}\mathcal{B}(V)=\dim_{\mathbb{C}}\mathcal{B}(\rho_{j_{1}}\cdots \rho_{j_{n}}(V)). 
\end{equation}
By \cite[Example 2.9]{MR3736568}, $\mathcal{B}(V)$ admits a PBW-basis consisting of $\mathbb{Z}^{\theta}$-homogeneous elements. Let $\Delta_{+}^{\mathfrak{q}}$ denote the set consisting of the $\mathbb{Z}^{\theta}$-degrees of the generators of a ($\mathbb{Z}^{\theta}$-homogeneous) PBW-basis of $\mathcal{B}(V)$, counted with multiplicities. The set $\Delta_{+}^{\mathfrak{q}}\subseteq \mathbb{Z}^{\theta}$ is called the \textit{set of positive roots} of $ \mathcal{B}(V)$ and is independent of the choice of the PBW-basis \cite[Remark 2.13]{MR3736568}. Set $\Delta^{\mathfrak{q}}=\Delta_{+}^{\mathfrak{q}}\cup (-\Delta_{+}^{\mathfrak{q}})$. Define $\rho:\mathbb{I}\times \mathcal{X}\to\mathcal{X}$ by $(i,\mathfrak{p})\mapsto \rho_{i}(\mathfrak{p})$. Then 
$$(\mathbb{I},\mathcal{X},\rho,\{(c_{ij}^{\mathfrak{p}})_{\theta\times \theta}\}_{\mathfrak{p}\in \mathcal{X}},\{\Delta^{\mathfrak{p}}\}_{\mathfrak{p}\in \mathcal{X}})$$ defines a finite \textit{generalized root system} \cite[Example 2.30]{MR3736568}. The Weyl groupoid of this generalized root system is denoted by $\mathcal{W}_{\mathfrak{q}}$ (see \cite[Definition 2.28]{MR3736568} or \cite[p.318]{MR4164719} for its definition). By \cite[Theorem 2.29 (b)]{MR3736568} or \cite[Proposition 9.3.9]{MR4164719}, there exists a unique longest element $w_{0}^{\mathfrak{q}}\in \mathcal{W}_{\mathfrak{q}}$ ending at $\mathfrak{q}$. Let $w_{0}^{\mathfrak{q}}=s_{i_{1}}^{\mathfrak{q}}s_{i_{2}}\cdots s_{i_{\ell}}$ be a reduced expression. Set 
$$\beta_{k}\coloneqq s_{i_{1}}^{\mathfrak{q}}\cdots  s_{i_{k-1}}(\alpha_{i_{k}})\in \mathbb{Z}^{\theta}.$$ By \cite[Theorem 2.29 (c)]{MR3736568} or \cite[Corollary 14.5.1 (2)]{MR4164719}, $\Delta_{+}^{\mathfrak{q}}=\{\beta_{k}\mid k\in \mathbb{I}\}\text{ and\ }|\Delta_{+}^{\mathfrak{q}}|=\ell.$ Since $\mathcal{B}(V)$ is finite-dimensional, 
\begin{equation}\label{eq-Nbeta-betarootfini9}
2\leq N_{\beta}=\text{ord}(\mathfrak{q}_{\beta\beta})<+\infty\text{ for all\ }\beta\in \Delta_{+}^{\mathfrak{q}}.
\end{equation}
See \cite[Example 2.9]{MR3736568} or \cite[\S 2.5]{MR3459702}.

\begin{remark}\label{rmk-transinver-nichostifd-0}
 In this case, it follows from the PBW-basis of the Nichols algebra $\mathcal{B}_{\mathfrak{q}}\coloneqq\mathcal{B}(V)$ (see, for example, \cite[\S 2.5]{MR3459702}) that
 \begin{equation}\label{eq-dimforofnichaqlg7}
 \dim_{\mathbb{C}}\mathcal{B}_{\mathfrak{q}}=\prod_{\beta\in\Delta_{+}^{\mathfrak{q}}} N_{\beta}.
 \end{equation}
Set $\mathfrak{q}^{-1}=(q_{ij}^{-1})_{\theta\times\theta}$ and $\mathfrak{q}^{T}=(q_{ji})_{\theta\times\theta}$. By \cite[Proposition 3.17 (d)]{MR2596372} and \eqref{eq-dimforofnichaqlg7}, one has
$$\Delta_{+}^{\mathfrak{q}}=\Delta_{+}^{\mathfrak{q}^{-1}}=\Delta_{+}^{\mathfrak{q}^{T}}\text{ and\ } \dim_{\mathbb{C}}\mathcal{B}_{\mathfrak{q}}= \dim_{\mathbb{C}}\mathcal{B}_{\mathfrak{q}^{-1}}= \dim_{\mathbb{C}}\mathcal{B}_{\mathfrak{q}^{T}}<+\infty. $$
\end{remark}
The set
\begin{equation}\label{eq-cartanroots-set}
\mathcal{O}^{\mathfrak{q}}=\{s_{i_{1}}^{\mathfrak{q}}s_{i_{2}}\cdots s_{i_{k}}(\alpha_{i}) \in \Delta^{\mathfrak{q}}\mid   i\in \mathbb{I}\text{ is a Cartan vertex of\ }\rho_{i_{k}}\cdots \rho_{i_{2}}\rho_{i_{1}}(\mathfrak{q})\}
\end{equation}
is called the set of \textit{Cartan roots} of $\mathfrak{q}$. For example, if $i$ is a Cartan vertex of $\mathfrak{q}$, then $\alpha_{i}\in \mathcal{O}^{\mathfrak{q}}$. In general, the set of Cartan roots $\mathcal{O}^{\mathfrak{q}}$ may be empty (see for example \cite[\S 10.7]{MR3736568}). Set $\mathcal{O}_{+}^{\mathfrak{q}}=\mathcal{O}^{\mathfrak{q}}\cap \mathbb{N}^{\theta}$. Then 
\begin{equation}\label{eq-setcartroot-pomO+}
\mathcal{O}^{\mathfrak{q}}=\mathcal{O}_{+}^{\mathfrak{q}}\sqcup (-\mathcal{O}_{+}^{\mathfrak{q}})
\end{equation}
because a Cartan vertex $i$ of any $\mathfrak{p}\in\mathcal{X}$
remains a Cartan vertex of $\rho_i(\mathfrak{p})$. If $\mathfrak{q}$ is of Cartan type, then any braiding matrix in $\mathcal{X}$ is of Cartan type. In particular, $\mathcal{O}^{\mathfrak{q}}=\Delta^{\mathfrak{q}}$ if $\mathfrak{q}$ is of Cartan type (hence $\mathcal{O}_{+}^{\mathfrak{q}}=\Delta_{+}^{\mathfrak{q}}$). Conversely, the condition $\mathcal{O}^{\mathfrak{q}}=\Delta^{\mathfrak{q}}$ implies that any $i\in \mathbb{I}$ is a Cartan vertex of $\mathfrak{q}$. So $\mathfrak{q}$ is of Cartan type.


Let $\mathcal{I}(V)$ be the ideal of $T(V)$ generated by all the relations appearing in \cite[Theorem 3.1]{MR3181554}, except for the power root vector relations
$$x_{\alpha}^{N_{\alpha}},\ \alpha\in\mathcal{O}_{+}^{\mathfrak q},$$ which are omitted, together with the quantum Serre relations
$$(\operatorname{ad}_{c}x_{i})^{1-c_{ij}^{\mathfrak q}}(x_{j})$$
for all $i\neq j$ such that
$q_{ii}^{c_{ij}^{\mathfrak q}}=q_{ij}q_{ji}=q_{ii}$. The quotient
$$\widetilde{\mathcal{B}}(V)\coloneqq T(V)/\mathcal{I}(V)$$
is called the \textit{distinguished pre-Nichols algebra} of $(V,c)$ \cite{MR3459702}.
\begin{remark}\label{rmk-genuinelyoffincart-disting}
Suppose $\mathfrak{q}$ is \textit{genuinely of finite Cartan type} with Cartan matrix $C=(c_{ij})_{\theta\times \theta}$ (see \cite[p.522]{MR4164719} for the definition) and the
associated Nichols algebra $\mathcal B(V)$ is finite-dimensional. By \cite[Remark 16.3.19 and p.548]{MR4164719}, 
$$\mathcal{I}(V)=((\operatorname{ad}_{c}x_{i})^{1-c_{ij}}(x_{j}),i\neq j).$$
So the corresponding distinguished pre-Nichols algebra is given by
$$\widetilde{\mathcal{B}}(V)\cong  T(V)/((\operatorname{ad}_{c}x_{i})^{1-c_{ij}}(x_{j}),i\neq j).$$
In general, even if $\mathfrak q$ is genuinely of finite Cartan type and $\dim_{\mathbb{C}}\mathcal B(V)<+\infty$, some of the
parameters $q_{ij},i\neq j$ may fail to be roots of unity. For example, let $C$ be of type
$A_{2}$, let $\omega$ be a primitive $3$-th root of unity, and let
$$\mathfrak{p}=(p_{ij})_{2\times 2}=\begin{pmatrix}
\omega& 2\\
(2\omega)^{-1}&\omega
\end{pmatrix}.$$
Then $\mathfrak{p}$ is genuinely of finite Cartan type and the corresponding Nichols algebra $\mathcal B(V)$ is finite-dimensional \cite[Theorem 15.1.14 (6)]{MR4164719}. However,
$p_{12}=2$ is not a root of unity.
\end{remark}

\begin{example}
[\cite{MR3459702,MR4164719}]Retain the assumptions and notation of \S \ref{subsec-DKquanenalgsCY}.\label{eq-positpart-Lztsmaqugrp9-nichols} Let $Q=\mathbb{Z}\alpha_{1}\oplus \cdots \oplus\mathbb{Z} \alpha_{n}$ denote  the root lattice of the root system $\Phi$. Then $\mathbb{C}[Q]\cong \mathbb{C}[\mathbb{Z}^{n}]\cong \mathbb{C}[K_{1}^{\pm 1},\ldots,K_{n}^{\pm 1}]$. Set $\mathfrak{q}=(\epsilon^{d_{i}a_{ij}})_{n\times n}\in \text{M}_{n}(\mathbb{C})$ (recall that $\epsilon$ is a primitive $N$-th root of unity). It is straightforward to verify that $\mathfrak{q}$ is genuinely of finite Cartan type with Cartan matrix $(a_{ij})_{n\times n}$. Define the Yetter-Drinfeld module $V=\mathbb{C}x_{1}\oplus\cdots \oplus \mathbb{C}x_{n}\in {_{\mathbb{C}[Q]}^{\mathbb{C}[Q]}\mathcal{YD}}$ by
$$\alpha_{i}\cdot x_{j}=q_{ij}x_{j} \text{ and\ }\operatorname{deg}x_{j}=\alpha_{j}.$$
Then $V$ is a braided vector space of diagonal type with braiding matrix $\mathfrak{q}$. In this case, the Nichols algebra $\mathcal{B}(V)$ coincides with the positive part $\mathfrak{u}_{\epsilon}^{+}(\mathfrak{g})$ of Lusztig's small quantum group $\mathfrak{u}_{\epsilon}(\mathfrak{g})$, where $x_{j}$ corresponds to the root vector $\overline{E_{j}}\in \mathfrak{u}_{\epsilon}(\mathfrak{g})$. See \cite[\S 16.3]{MR4164719}. We identify $\mathcal{B}(V)$ with $\mathfrak{u}_{\epsilon}^{+}(\mathfrak{g})$ via the canonical isomorphism $\mathcal{B}(V)\cong \mathfrak{u}_{\epsilon}^{+}(\mathfrak{g})$. 

By \eqref{eq-pbw-smallqborel}, $\mathfrak{u}_{\epsilon}^{+}(\mathfrak{g})$ has a PBW-basis $\{\overline{E_{\beta_{1}}}^{m_{1}}\cdots \overline{E_{\beta_{\ell}}}^{m_{\ell}}\mid 0\leq m_{1},\ldots,m_{\ell}\leq N-1\}$, where $\operatorname{deg}E_{\beta_{j}}=\beta_{j}$. It follows that there is a canonical bijection between the set of positive roots $\Delta_{+}^{\mathfrak{q}}$ of $\mathcal{B}(V)$ and the set $\Phi^{+}=\{\beta_{1},\ldots,\beta_{\ell}\}$ of positive roots of the root system $\Phi$. Indeed, the same conclusion holds for any braiding matrix $\mathfrak{q}$ of finite Cartan type; see \cite[Theorem 1 (ii)]{MR2207786}.

Recall that the $\mathbb{Z}$-valued symmetric bilinear form $
(-,-):Q\times Q\to \mathbb{Z},(\alpha_i,\alpha_j)\mapsto d_{i} a_{ij},$ is $W$-invariant, where $W$ is the Weyl group of $\mathfrak{g}$. Hence, for every $\beta\in\Delta_{+}^{\mathfrak{q}}$, there exists $1\leq i\leq n$ such that $\mathfrak{q}_{\beta\beta}=\epsilon^{2d_{i}}$. Therefore, $N_{\beta}=N$ for all $\beta\in\Delta_{+}^{\mathfrak{q}}$. By \cite[Proposition 4.3.12 (1)]{MR4164719},
$$\mathcal{B}(V)=\mathbb{C}\langle E_{1},\ldots,E_{n}\rangle/(E_{\beta_{1}}^{N},\ldots,E_{\beta_{\ell}}^{N},(\operatorname{ad}_{c}E_{i})^{1-a_{ij}}(E_{j}), i\neq j).$$
By the definition of $\widetilde{\mathcal{B}}(V)$ and Remark \ref{rmk-genuinelyoffincart-disting}, we obtain \cite[Remark 2]{MR3459702} that 
$$\widetilde{\mathcal{B}}(V)=\mathbb{C}\langle E_{1},\ldots,E_{n}\rangle/((\operatorname{ad}_{c}E_{i})^{1-a_{ij}}(E_{j}),i\neq j)=\mathcal{U}_{\epsilon}^{+}(\mathfrak{g}).$$
\end{example}

To define the AAY large quantum groups, we follow \cite[\S 4]{MR4600057} and introduce additional notation. Throughout this section, we impose the following assumption on the braiding matrix $\mathfrak{q}$:

\medskip
\noindent
\textup{(\textsc{FCC})}\quad
\begin{minipage}[t]{0.82\textwidth}
\textit{The corresponding Nichols algebra $\mathcal B(V)$ is finite-dimensional, and}
\begin{equation}\label{eq-ffc-cartroot-largech7}
q_{\alpha\beta}^{N_\beta}=1\ \textit{for all }
\alpha\in\Delta^{\mathfrak q}, 
\beta\in\mathcal O^{\mathfrak q}.
\end{equation}
\end{minipage}

\medskip
\begin{remark}\label{rmk-rhiFCCiff-rhoifcc}
By \eqref{eq-Nbeta-betarootfini9}, $N_{\beta}<+\infty$ for every $\beta\in\mathcal{O}^{\mathfrak{q}}\subseteq \Delta^{\mathfrak q}$. Let $i\in \mathbb{I}$. By \eqref{eq-weylequclas-nichfd-ref} and \cite[Remark 4.4 (c)]{MR4600057}, (FCC) holds for $\mathfrak{q}$ if and only if it holds for $\rho_{i}(\mathfrak{q})$. If (FCC) holds for $\mathfrak{q}$, then $q_{\beta\alpha}^{N_{\beta}}=1$ for all $\alpha\in\Delta^{\mathfrak q}$ and $\beta\in\mathcal O^{\mathfrak q}$; see \cite[Lemma 24]{MR3459702}.
\end{remark}

\begin{remark}\label{rmk-frambrma-admothonpa7}
In \cite{MR4600057}, Andruskiewitsch-Angiono-Yakimov impose assumptions on the braiding matrix $\mathfrak{q}$ that are stronger than the assumption (\text{FCC}). 
In addition to the conditions imposed here (see also \cite[Eq. (4.26)]{MR4600057}), it is assumed that the Dynkin diagram associated with $\mathfrak{q}$ is connected and that the hypotheses stated at the beginning of \cite[\S 4.1]{MR4600057} hold. A braiding matrix $\mathfrak{q}$ satisfying these conditions is referred to therein as \textit{belonging to a one-parameter family}. 
These additional assumptions are imposed to ensure that the corresponding large quantum group (see Definition \ref{def-largequbgroup-whole} below) admits a nontrivial Poisson order structure in the sense of Brown-Gordon \cite{MR1989650}. 

All one-parameter families considered by Andruskiewitsch-Angiono-Yakimov are listed in \cite[Appendix A]{MR4600057}. In addition to the braiding matrices listed therein, our framework also applies to other types of braiding matrices; see, for example, \cite[\S 10.7]{MR3736568}.
\end{remark}

Let $\Gamma^{+}$ and $\Gamma^{-}$ be free abelian groups of rank $\theta$, with respective bases
$\{K_i\}_{i=1}^{\theta}$ and $\{L_i\}_{i=1}^{\theta}$. We regard both groups as additive groups. We regard the braided vector space $V$ of diagonal type with braiding matrix $\mathfrak{q}$ as a Yetter-Drinfeld module over $\mathbb{C}[\Gamma^{+}]$ via \eqref{eq-YDmodustr-groualgdiagtype}. Let $W$ be a braided vector space of diagonal type with braiding matrix $\mathfrak{q}^{\prime}=(q_{ij}^{\prime})_{\theta\times \theta}\coloneqq(q_{ji}^{-1})_{\theta\times \theta}$ with respect to the basis $\{y_{i}\}_{i=1}^{\theta}$. Similarly, $W$ can be naturally regarded as a left Yetter-Drinfeld module over
$\mathbb{C}[\Gamma^{-}]$. By Remark \ref{rmk-transinver-nichostifd-0}, we have $\dim_{\mathbb{C}} \mathcal{B}(W)=\dim_{\mathbb{C}}\mathcal{B}(V)<+\infty$. Define
$$U_{\mathfrak{q}}^{\geq 0}\coloneqq \widetilde{\mathcal{B}}(V)\# \mathbb{C}[\Gamma^{+}],\quad U_{\mathfrak{q}}^{\leq 0}\coloneqq \widetilde{\mathcal{B}}(W)\# \mathbb{C}[\Gamma^{-}].$$

\begin{definition}
[\cite{MR4600057}] Assume that the braiding matrix $\mathfrak{q}$ satisfies the assumption (FCC). The Hopf algebras $U_{\mathfrak{q}}^{\geq 0}$ and $U_{\mathfrak{q}}^{\leq 0}$ are called the AAY \textit{large quantized Borel algebras} associated with $\mathfrak{q}$.\label{def-largequborelsubalgs}
\end{definition}
\begin{remark}\label{rmk-largeboreldefinidea-spca2}
Let $E_{i}=x_{i}\in U_{\mathfrak{q}}^{\geq 0}, F_{i}=y_{i}L_{i}^{-1}\in U_{\mathfrak{q}}^{\leq 0}$ for $i\in \mathbb{I}$. As a $\mathbb{C}$-algebra, $U_{\mathfrak{q}}^{\geq 0}$ is defined by the generators $E_{i}, K_{i}^{\pm1}, i\in\mathbb{I},$ with the defining ideal generated by
$$
\mathcal{I}(V),\ K_{i}K_{j}-K_{j}K_{i},\ K_{i}K_{i}^{-1}-1,\ K_{i}^{-1}K_{i}-1,\
K_{i}E_{j}K_{i}^{-1}-q_{ij}E_{j}, i,j\in\mathbb{I},
$$
where $\mathcal{I}(V)$ is identified with its image under $x_{i}\mapsto E_{i}$. The coalgebra structure of $U_{\mathfrak{q}}^{\geq 0}$ is given by $\Delta(E_{i})=E_{i}\otimes 1+K_{i}\otimes E_{i}, \varepsilon(E_{i})=0,\Delta(K_{i})=K_{i}\otimes K_{i}$ and $\varepsilon(K_{i})=1$. Consequently, it follows from Example \ref{eq-positpart-Lztsmaqugrp9-nichols} that, for a complex semisimple Lie algebra $\mathfrak{g}$ of rank $n$ with Cartan matrix $(a_{ij})_{n\times n}$, the DKP quantized Borel algebra
$\mathcal{U}_{\epsilon}^{\geq 0}(\mathfrak{g})$ at a root of unity $\epsilon$ can be regarded
as a special case of $U_{\mathfrak q}^{\geq 0}$ by taking
$$\theta=n\text{ and\ } \mathfrak{q}=(\epsilon^{d_{i}a_{ij}})_{n\times n}.$$
Similarly, $U_{\mathfrak{q}}^{\leq 0}$ is defined by the generators
$F_{i},L_{i}^{\pm1},i\in\mathbb I$, with the defining ideal generated by $
\mathcal{I}(W),\ L_{i}L_{j}-L_{j}L_{i},\ L_{i}L_{i}^{-1}-1,\ L_{i}^{-1}L_{i}-1,\
L_{i}F_{j}L_{i}^{-1}-q_{ji}^{-1}F_{j},$ where $i,j\in\mathbb{I}.$ The quantum group $\mathcal U_{\epsilon}^{\leq 0}(\mathfrak g)$ can be viewed as a special case of $U_{\mathfrak q}^{\leq 0}$.
\end{remark}
By \cite[Theorem 3.7]{MR2732981} or \cite[p.23]{MR4600057} (see also \cite[Proposition 4.3]{MR2596372}), there is a unique Hopf skew-pairing (\cite[Definition 2.8.4]{MR4164719}) $\tau^{AAY}:U_{\mathfrak{q}}^{\leq 0}\times U_{\mathfrak{q}}^{\geq 0}\to\mathbb{C}$ such that
$$\tau^{AAY}(L_{i},K_{j})=q_{ji}^{-1},\ \tau^{AAY}(F_{i},E_{j})=\delta_{ij},\ \tau^{AAY}(L_{i},E_{j})=\tau^{AAY}(F_{i},K_{j})=0,\   i,j\in \mathbb{I}.$$
Note that, in the definition of large quantum groups below, it is sufficient to assume that the Nichols algebra associated with the braiding matrix $\mathfrak{q}$ is finite-dimensional. However, for convenience in subsequent considerations, we impose the stronger assumption (FCC) directly in the definition. As observed in Remark \ref{rmk-frambrma-admothonpa7}, the assumptions on $\mathfrak{q}$ imposed in \cite{MR4600057} for the study of the representation theory of large quantum groups are stronger than assumption (FCC). Consequently, the present framework applies to a broader class
of quantum algebras than that considered in \cite{MR4600057}.

\begin{definition}
[\cite{MR4600057}] Assume that the braiding matrix $\mathfrak{q}$ satisfies the assumption (FCC). The Drinfeld quantum double of $U_{\mathfrak{q}}^{\leq 0}$ and $U_{\mathfrak{q}}^{\geq 0}$ with respect to the  Hopf skew-pairing 
$$\tau^{AAY}: U_{\mathfrak{q}}^{\leq 0}\times U_{\mathfrak{q}}^{\geq 0}\to\mathbb{C},$$
denoted by $U_{\mathfrak{q}}$, is called the AAY \textit{large quantum group} associated with $\mathfrak{q}$. \label{def-largequbgroup-whole}
\end{definition}
\begin{remark}\label{rmk-generare-Uq-largewh}
With the notation of Remark \ref{rmk-largeboreldefinidea-spca2}, the large quantum group $U_{\mathfrak{q}}$, as an algebra, is defined by the generators $E_{i},F_{i}, K_{i}^{\pm1},L_{i} ,i\in\mathbb{I},$ with the defining ideal generated by
$$\mathcal{I}(V),\ \mathcal{I}(W),\ K_{i}E_{j}K_{i}^{-1}-q_{ij}E_{j},\ K_{i}F_{j}K_{i}^{-1}-q_{ij}^{-1}F_{j},$$
$$L_{i}E_{j}L_{i}^{-1}-q_{ji}E_{j},\ L_{i}F_{j}L_{i}^{-1}-q_{ji}^{-1}F_{j},\ K_{i}K_{i}^{-1}-1,\  K_{i}^{-1}K_{i}-1,$$
$$L_{i}L_{i}^{-1}-1,\ L_{i}^{-1}L_{i}-1,\ E_{i}F_{j}-F_{j}E_{i}=\delta_{ij}(K_{i}-L_{i}^{-1}), $$
where $i,j\in \mathbb{I}$. The Hopf algebra structure of $U_{\mathfrak{q}}$ is given by
$$\Delta(K_{i})=K_{i}\otimes K_{i},\  \Delta(L_{i})=L_{i}\otimes L_{i},\  \Delta(E_{i})= E_{i}\otimes 1+K_{i}\otimes E_{i},\ \Delta(F_{i})=F_{i}\otimes L_{i}^{-1}+1\otimes F_{i}, $$
$$\varepsilon(K_{i})=\varepsilon(L_{i})=1,\    \varepsilon(E_{i})=\varepsilon(F_{i})=0,$$
$$S(E_{i})=-K_{i}^{-1}E_{i},\ S(F_{i})=-F_{i}L_{i},\  S(K_{i})=K_{i}^{-1},\ S(L_{i})=L_{i}^{-1},$$
where $i\in \mathbb{I}$. Clearly, both $U_{\mathfrak{q}}^{\geq 0}$ and $U_{\mathfrak{q}}^{\leq 0}$ are Hopf subalgebras of $U_{\mathfrak{q}}$.
\end{remark}
In the rest of this section, \textit{we assume that the braiding matrix $\mathfrak{q}$ satisfies the assumption (FCC)}. We denote by $U_{\mathfrak{q}}^{+}$
the subalgebra of $U_{\mathfrak{q}}$
generated by $E_{i}$, $i\in\mathbb{I}$,
and by $U_{\mathfrak{q}}^{-}$
the subalgebra generated by $F_{i}$, $i\in\mathbb{I}$. The algebras $U_{\mathfrak{q}}^{+}$
and $U_{\mathfrak{q}}^{-}$ are called the positive and negative
\textit{large quantum nilpotent algebras}, respectively.

Example \ref{eg-multiparamqgrp-AAYlarge} below shows that
large quantum groups include a broad class of the multiparameter quantized enveloping algebras discussed in Remark \ref{rmk-multquanenealg-generunim}.

\begin{example}\label{eg-multiparamqgrp-AAYlarge}
Let $\mathfrak{q}$ be genuinely of finite Cartan type with Cartan matrix $C=(c_{ij})_{\theta\times \theta}$, and assume that $\mathfrak{q}$ satisfies the assumption (FCC). By Remark \ref{rmk-genuinelyoffincart-disting} and \cite[Proposition 4.3.12 (1)]{MR4164719}, one can directly verify that
$$\mathcal{I}(V)=\left(\sum_{k=0}^{1-c_{ij}}(-1)^{k}\left(\begin{smallmatrix}1-c_{ij}\\k\end{smallmatrix}\right)_{q_{ii}}q_{ii}^{\frac{k(k-1)}{2}}q_{ij}^{k}E_{i}^{1-c_{ij}-k}E_{j}E_{i}^{k}=0,i\neq j\right),\text{ and}$$
$$\mathcal{I}(W)=\left(\sum_{k=0}^{1-c_{ij}}(-1)^{k}\left(\begin{smallmatrix}1-c_{ij}\\k\end{smallmatrix}\right)_{q_{ii}}q_{ii}^{\frac{k(k-1)}{2}}q_{ij}^{k}F_{i}^{k}F_{j}F_{i}^{1-c_{ij}-k}=0,i\neq j  \right).$$
Consequently, the corresponding large quantum group $U_{\mathfrak{q}}\cong \mathcal{U}_{\mathfrak{q}}(\mathfrak{g}_{C})$ as Hopf algebras, where $ \mathfrak{g}_{C}$ denotes the complex semisimple Lie algebra associated with the Cartan matrix $C$, and
$\mathcal U_{\mathfrak q}(\mathfrak g_C)$ denotes the multiparameter
quantized enveloping algebra of $\mathfrak{g}_{C}$; see Remark \ref{rmk-multquanenealg-generunim}.
\end{example}

By \cite[Proposition 10]{MR3459702} (see also \cite[\S 6]{MR2596372} and \cite[\S 4.4]{MR4600057}), there exist algebra isomorphisms $T_{i}^{\mathfrak{q}}:U_{\rho_{i}(\mathfrak{q})}\to  U_{\mathfrak{q}}$ such that
 $$T_{i}^{\mathfrak{q}}(\underline{K}_{j})=K_{j}K_{i}^{-c_{ij}^{\mathfrak{q}}},\  T_{i}^{\mathfrak{q}}(\underline{L}_{j})=L_{j}L_{i}^{-c_{ij}^{\mathfrak{q}}},$$
$$T_{i}^{\mathfrak{q}}(\underline{E}_{j}) =
\begin{cases}
F_{i} L_{i}, & j=i,\\
(\operatorname{ad}_{c}E_{i})^{-c_{ij}^{\mathfrak{q}}} E_{j},
& j\neq i, \qquad
\end{cases}\
T_{i}^{\mathfrak{q}}(\underline{F}_{j})
=
\begin{cases}
K_{i}^{-1} E_i,
& j=i,\\
(\lambda_{ij}^{\mathfrak{q}})^{-1}
(\operatorname{ad}_{c} f_i)^{-c_{ij}^{\mathfrak{q}}} F_{j},
& j\neq i,
\end{cases}
$$
where the underlined elements denote the generators of $U_{\rho_{i}(\mathfrak{q})}$, and $\lambda_{ij}^{\mathfrak{q}}\in \mathbb{C}^{\times}$ is given by (see also \cite[Lemma 3.18]{MR2596372})
$$\lambda_{ij}^{\mathfrak{q}}=\left(
q_{ii}^{-c_{ij}^{\mathfrak{q}}}
q_{ij}q_{ji}
\right)^{c_{ij}^{\mathfrak{q}}}
\left(-c_{ij}^{\mathfrak{q}}\right)_{q_{ii}}!
\prod_{s=0}^{-c_{ij}^{\mathfrak{q}}-1}
\left(
q_{ii}^{s}q_{ij}q_{ji}-1
\right),\  j\neq i\in \mathbb{I}.$$
Let $w_{0}^{\mathfrak{q}}$ be the longest element in $\mathcal{W}_{\mathfrak{q}}$ ending at $\mathfrak{q}$ and let $w_{0}^{\mathfrak{q}}=s_{i_{1}}^{\mathfrak{q}}s_{i_{2}}\cdots s_{i_{\ell}}$ be a reduced expression of $w_{0}^{\mathfrak{q}}$. By \cite[Theorem 6.20]{MR2596372} or \cite[Eq. (4.16)]{MR4600057},
$$E_{\beta_{k}}\coloneqq  T_{i_{1}}^{\mathfrak{q}}\cdots T_{i_{k-1}}(\underline{E}_{i_{k}})\in U_{\mathfrak{q}}^{+}\text{ and\ }  F_{\beta_{k}}\coloneqq  T_{i_{1}}^{\mathfrak{q}}\cdots T_{i_{k-1}}(\underline{F}_{i_{k}})\in U_{\mathfrak{q}}^{-}.$$
For $\beta\in \Delta_{+}^{\mathfrak{q}}$, define
$$\widetilde{N}_{\beta}\coloneqq 
\begin{cases}
N_{\beta},
& \beta\notin\mathcal{O}^{\mathfrak{q}},\\
+\infty,
& \beta\in\mathcal{O}^{\mathfrak{q}}.
\end{cases}$$
By \cite[Eq. (4.18)]{MR4600057}, the set
$$ \{
E_{\beta_{1}}^{k_1}E_{\beta_{2}}^{k_2}\cdots E_{\beta_{\ell}}^{k_{\ell}}
\mid 0\leq k_{j}<\widetilde{N}_{\beta_j},1\leq j\leq \ell \}$$ 
is a $\mathbb{C}$-basis of $ U_{\mathfrak{q}}^{+}$, and the set
$$ \{F_{\beta_{1}}^{r_1}F_{\beta_{2}}^{r_2}\cdots F_{\beta_{\ell}}^{r_{\ell}}\mid
0\leq r_{j}<\widetilde{N}_{\beta_j},  1\leq j\leq \ell \}$$ 
is a $\mathbb{C}$-basis of $ U_{\mathfrak{q}}^{-}$. It follows that the set
\begin{equation}\label{eq-PBWbasis-Uqlargequ7}
 \{
F_{\beta_{1}}^{r_1} \cdots F_{\beta_{\ell}}^{r_{\ell}}L_{1}^{t_{1}}\cdots L_{\theta}^{t_{\theta}}	K_{1}^{s_{1}}\cdots K_{\theta}^{s_{\theta}}E_{\beta_{1}}^{k_1} \cdots E_{\beta_{\ell}}^{k_{\ell}}
\mid 0\leq k_{j},r_{j}<\widetilde{N}_{\beta_j},
,s_{i},t_{i}\in \mathbb{Z} \}
\end{equation}
is a $\mathbb{C}$-basis of $U_{\mathfrak{q}}$. For $\gamma=s_{1}\alpha_{1}+\cdots+s_{\theta}\alpha_{\theta}\in\mathbb{Z}^{\theta}$, we write
$$K_{\gamma}\coloneqq K_{1}^{s_{1}}\cdots K_{\theta}^{s_{\theta}}\text{ and\ }L_{\gamma}\coloneqq L_{1}^{s_{1}}\cdots L_{\theta}^{s_{\theta}}.$$
Let $N$ be the least common multiple of the integers
$N_{\beta}$, $\beta\in\Delta_{+}^{\mathfrak{q}}$, and let $Z_{\mathfrak{q}} $ be the subalgebra of $U_{\mathfrak{q}}$
generated by the elements \cite[\S 4.5]{MR4600057}:
$$\{E_{\beta}^{N_{\beta}},F_{\beta}^{N_{\beta}},K_{\beta}^{\pm N_{\beta}},L_{\beta}^{\pm N_{\beta}}\mid  \beta\in \mathcal{O}_{+}^{\mathfrak{q}}\}\cup \{K_{\gamma}^{\pm N},L_{\gamma}^{\pm N}\mid \gamma\in \Delta_{+}^{\mathfrak{q}}\}.$$
Define 
$$Z_{\mathfrak{q}}^{\geq 0}\coloneqq \mathbb{C}\langle E_{\beta}^{N_{\beta}},K_{\beta}^{\pm N_{\beta}},K_{\gamma}^{\pm N}\mid \beta\in \mathcal{O}_{+}^{\mathfrak{q}},\gamma\in  \Delta_{+}^{\mathfrak{q}}\rangle\subseteq Z_{\mathfrak{q}},$$
$$Z_{\mathfrak{q}}^{\leq 0}\coloneqq \mathbb{C}\langle F_{\beta}^{N_{\beta}},L_{\beta}^{\pm N_{\beta}},L_{\gamma}^{\pm N}\mid \beta\in \mathcal{O}_{+}^{\mathfrak{q}},\gamma\in  \Delta_{+}^{\mathfrak{q}}\rangle\subseteq Z_{\mathfrak{q}}.$$
\begin{theorem}
[\cite{MR3459702,MR4600057}] Assume\label{thm-largequanmodufi-aay1} that the braiding matrix $\mathfrak{q}$ satisfies the assumption (FCC). Then $Z_{\mathfrak{q}} $ is a central Hopf subalgebra of $U_{\mathfrak{q}}$, and $U_{\mathfrak{q}}$ is a finitely generated free module over $Z_{\mathfrak{q}}$. The same conclusions hold for the pairs
$(U_{\mathfrak{q}}^{\geq 0},Z_{\mathfrak{q}}^{\geq 0})$
and
$(U_{\mathfrak{q}}^{\leq 0},Z_{\mathfrak{q}}^{\leq 0})$.
\end{theorem}
\begin{proof}
See \cite[Theorem 33 and Remark 11]{MR3459702} and \cite[Remark 4.6 (d)]{MR4600057}.
\end{proof}

\begin{remark}
If an algebra $A$ is finitely generated and free as a module over a central subalgebra $C$, one may consider the \textit{regular trace} $\text{tr}_{\text{reg}}:A\to C$ (see, for example, \cite[Eq. (2.1)]{MR4707278} for its definition), which is a $C$-linear map. It follows that the algebras with trace $(U_{\mathfrak{q}},Z_{\mathfrak{q}},\text{tr}_{\text{reg}})$, $(U_{\mathfrak{q}}^{\geq 0},Z_{\mathfrak{q}}^{\geq 0},\text{tr}_{\text{reg}})$, and $(U_{\mathfrak{q}}^{\leq 0},Z_{\mathfrak{q}}^{\leq 0},\text{tr}_{\text{reg}})$ are affine Cayley-Hamilton Hopf algebras in the sense of De Concini-Procesi-Reshetikhin-Rosso \cite{MR2178656}. For general results on the representation theory of Cayley-Hamilton Hopf algebras, see, for example, \cite{MR2178656,MR4891381,HMQW2026chev,HQWZ2026chev}.

In \cite{MR4600057}, Andruskiewitsch-Angiono-Yakimov constructed Poisson order structures on $(U_{\mathfrak{q}},Z_{\mathfrak{q}})$,
$(U_{\mathfrak{q}}^{\geq 0},Z_{\mathfrak{q}}^{\geq 0})$, and
$(U_{\mathfrak{q}}^{\leq 0},Z_{\mathfrak{q}}^{\leq 0})$ when $\mathfrak{q}$ belongs to a one-parameter family \cite[\S 4.1]{MR4600057}, thus facilitating the study of the representation theory of these quantum groups through methods from Poisson geometry. Moreover, these Poisson orders, together with the corresponding regular traces, form Poisson trace orders in the sense of Brown-Yakimov \cite{MR4707278}.
\end{remark}



\begin{corollary}
Assume that the braiding matrix $\mathfrak{q}$ satisfies the assumption (FCC). Then\label{cor-ASgorenlargequan8-eregucart6}
\begin{itemize}
\item[(a)]The large quantum group $U_{\mathfrak{q}}$ is a noetherian affine AS-Gorenstein Hopf algebra of injective dimension $2|\mathcal{O}_{+}^{\mathfrak{q}}|+2\theta=|\mathcal{O}^{\mathfrak{q}}|+2\theta$. 
\item[(b)]The large quantized Borel algebra $U_{\mathfrak{q}}^{\geq 0}$ is a noetherian affine AS-Gorenstein Hopf algebra of injective dimension $|\mathcal{O}_{+}^{\mathfrak{q}}|+ \theta$.  
\item[(c)]If $\mathfrak{q}$ is of Cartan type, then both $U_{\mathfrak{q}}$ and $U_{\mathfrak{q}}^{\geq 0}$ are AS-regular Hopf domains.
\item[(d)]The quantum group $U_{\mathfrak{q}}^{\geq 0}$ is AS-regular if and only if $\mathfrak{q}$ is of Cartan type.
\item[(e)]The quantum group $U_{\mathfrak{q}}$ is AS-regular if and only if $\mathfrak{q}$ is of Cartan type.
\end{itemize} 
\end{corollary}
\begin{proof}
(a) and (b) By Theorem \ref{thm-largequanmodufi-aay1} and \cite[Theorem 0.2 (1)]{MR1938745}, both $U_{\mathfrak{q}}$ and $U_{\mathfrak{q}}^{\geq 0}$ are AS-Gorenstein. By Theorem \ref{thm-noethAS-intebimoduiso-un} (a), for any affine module-finite Hopf algebra $H$, the injective dimension $\text{inj.dim}_{H}H$ is determined by the Gelfand-Kirillov dimension of $H$. Thus, (a) and (b) follow from \cite[Theorem 20]{MR3459702} and \eqref{eq-setcartroot-pomO+}; alternatively, they can be obtained by directly computing the Gelfand-Kirillov dimensions of $Z_{\mathfrak{q}}$ and $Z_{\mathfrak{q}}^{\geq 0}$. 

(c) It suffices to observe that, in this case, $\mathcal{O}_{+}^{\mathfrak{q}}=\Delta_{+}^{\mathfrak{q}}$, and then apply \cite[Corollaries 17 and 18]{MR3459702} together with \cite[Lemma I. 12.12 and Theorem I. 12.14]{MR1898492}.

(d) By (c), it remains only to prove the necessity. Suppose, to the contrary, that there exists a non-Cartan root $\beta  =s_{i_1}^{\mathfrak{q}}\cdots s_{i_{k-1}}(\alpha_{i_k})
\in\Delta_{+}^{\mathfrak{q}}.$ Since $T_{i_{1}}^{\mathfrak{q}}\cdots T_{i_{k-1}}:U_{\rho_{i_{k-1}}\cdots \rho_{i_{1}}(\mathfrak{q})}\to U_{\mathfrak{q}}$ is an algebra isomorphism, \eqref{eq-PBWbasis-Uqlargequ7} implies that $E_{i_{k}}$ is a nilpotent element in $U_{\rho_{i_{k-1}}\cdots \rho_{i_{1}}(\mathfrak{q})}$. This forces $\alpha_{i_{k}}$ to be a non-Cartan root in $\Delta_{\rho_{i_{k-1}}\cdots \rho_{i_{1}}(\mathfrak{q})}^{+}$. Therefore, after replacing $\mathfrak{q}$ by $\rho_{i_{k-1}}\cdots\rho_{i_1}(\mathfrak{q})$ and $U_{\mathfrak{q}}$ by $U_{\rho_{i_{k-1}}\cdots \rho_{i_{1}}(\mathfrak{q})}$, we may assume without loss of generality that the non-Cartan root $\beta \in\Delta_{+}^{\mathfrak{q}}$ is a simple root $\alpha_{j}$
for some $j\in\mathbb{I}$. By \cite[Corollary 9.3.1]{MR4164719}, we may choose a reduced expression
of $w_{0}^{\mathfrak{q}}$ such that the first positive root $\beta_{1}$ in the corresponding sequence is $\alpha_{j}$. Now 
$$\mathscr{A}_{\beta_{1}}\coloneqq \mathbb{C}\langle E_{\beta_{1}}\rangle\subseteq  U_{\mathfrak{q}}^{\geq 0}$$
is isomorphic to the non-semisimple Frobenius algebra $\mathbb{C}[z]/(z^{N_{\beta_{1}}})$. Since $U_{\mathfrak{q}}^{\geq 0}$ is AS-regular, the trivial right $U_{\mathfrak{q}}^{\geq 0}$-module $\mathbb{C}$ admits a finite-length projective resolution $P^{\bullet}\to \mathbb{C}\to 0.$ By \eqref{eq-PBWbasis-Uqlargequ7}, $U_{\mathfrak{q}}^{\geq 0}$ is free as a right $\mathscr{A}_{\beta_1}$-module. Hence, $P^{\bullet}\to \mathbb{C}\to 0$ is a projective resolution of the right $\mathscr{A}_{\beta_1}$-module $\mathbb{C}$. It follows that
$\mathscr{A}_{\beta_1}$ is semisimple, which yields a contradiction.

(e) This follows from an argument similar to that in (d).
\end{proof}

\begin{remark}\label{rmk-ineeleinUq}
In the proof of \cite[Corollary 18]{MR3459702}, Angiono considered an $\mathbb{N}^{2\ell+1}$-filtration on $U_{\mathfrak{q}}$, analogous to the filtration introduced by De Concini-Kac in \cite{MR1103601} for $\mathcal{U}_{\epsilon}(\mathfrak{g})$. Recall that $\ell$ denotes the length of the longest element $w_{0}^{\mathfrak{q}}$. 

If $\mathfrak{q}$ is of Cartan type, \cite[Corollary 18]{MR3459702} implies that the associated graded algebra $\operatorname{gr}U_{\mathfrak{q}}$ is an $\mathbb{N}^{2\ell+1}$-graded domain whose homogeneous component of degree
$0\in\mathbb{N}^{2\ell+1}$ is precisely the commutative algebra $\mathbb{C}[K_{1}^{\pm 1},\ldots,K_{\theta}^{\pm 1},L_{1}^{\pm 1},\ldots,L_{\theta}^{\pm 1}]$. Hence, the invertible elements of $U_{\mathfrak{q}}$ are precisely the nonzero scalar multiples of Laurent monomials in $K_{1},\ldots,K_{\theta},L_{1},\ldots,L_{\theta}$. When $\mathfrak{q}$ is \textit{not} of Cartan type, $U_{\mathfrak{q}}$ has additional invertible elements, as a non-Cartan root $\beta\in \Delta_{+}^{\mathfrak{q}}$ gives rise to a nilpotent element $E_{\beta}$.
\end{remark}

The following statement follows immediately from Example \ref{eg-multiparamqgrp-AAYlarge} and Corollary \ref{cor-ASgorenlargequan8-eregucart6} (c).

\begin{corollary}\label{cor-PHRmulparaquangro-regulardoman}
Assume that the braiding matrix $\mathfrak{q}$ satisfies the assumption (FCC) and is genuinely of finite Cartan type with Cartan matrix $C=(c_{ij})_{\theta\times\theta}$. Let $\mathfrak{g}_{C}$ denote the complex semisimple Lie algebra associated with 
$C$. Then the multiparameter
quantized enveloping algebra $\mathcal{U}_{\mathfrak{q}}(\mathfrak{g}_{C})$ is an affine noetherian AS-regular Hopf algebra and a domain.
\end{corollary}

The following result is straightforward.

\begin{lemma}
Assume that the braiding matrix $\mathfrak{q}$ satisfies the assumption (FCC). Let $\alpha,\beta\in\mathbb{Z}^{\theta}$, and set $u=K_{\alpha}L_{\beta}\in U_{\mathfrak{q}}$. Then the inner automorphism $\iota_{u}: U_{\mathfrak{q}}\to U_{\mathfrak{q}}$, defined by $x\mapsto uxu^{-1}$, satisfies 
$\iota_{u}(E_{j})=\mathfrak{q}(\alpha,\alpha_{j})\mathfrak{q}(\alpha_{j},\beta)E_{j}\text{ and\ } \iota_{u}(F_{j})=\mathfrak{q}(\alpha,\alpha_{j})^{-1}\mathfrak{q}(\alpha_{j},\beta)^{-1}F_{i}$
for all $j\in \mathbb{I}$. \label{lem-innerKaLb-formu}
\end{lemma}

We establish a necessary and sufficient condition for $S^{2}$ of $U_{\mathfrak{q}}$ to be inner when $\mathfrak{q}$ is of Cartan type.

\begin{proposition}\label{prop-UqS2-inner}
Assume that the braiding matrix $\mathfrak{q}$ satisfies the assumption (FCC). Let $S$ be the antipode of $U_{\mathfrak{q}}$. Then $S^{2}$ is inner if there exist $\alpha,\beta\in \mathbb{Z}^{\theta}$ such that 
\begin{equation}\label{eq-S2inner}
\mathfrak{q}(\alpha,\alpha_{j})\mathfrak{q}(\alpha_{j},\beta)=q_{jj}^{-1}
\end{equation}
for all $j\in \mathbb{I}$. Moreover, if $\mathfrak{q}$ is of Cartan type, then the converse also holds.
\end{proposition}

\begin{proof}
By Remark \ref{rmk-generare-Uq-largewh}, $S^{2}(K_{j})=K_{j}$, $S^{2}(L{j})=L_{j}$, $S^{2}(E_{j})=q_{jj}^{-1}E_{j}$, and $S^{2}(F_{j})=q_{jj}F_{j}$. Therefore, the conclusion follows directly from Remark \ref{rmk-ineeleinUq} and Lemma \ref{lem-innerKaLb-formu}.
\end{proof}

For example, the braiding matrix $(\epsilon^{d_{i}a_{ij}})_{n\times n}$
considered in Example \ref{eq-positpart-Lztsmaqugrp9-nichols} satisfies the condition in Proposition \ref{prop-UqS2-inner}. We now provide an example of a braiding matrix that satisfies the assumption (FCC) but does not satisfy \eqref{eq-S2inner} in Proposition \ref{prop-UqS2-inner}.

\begin{example}
Consider the braiding matrix\label{eg-finitegen-maynotS2iner}
$$\mathfrak{q}=
\begin{pmatrix}
-1 & -1 & -1\\
-1 & -1 & 1\\
-1 & 1 & \mathfrak{i}
\end{pmatrix}\in \text{M}_{3}(\mathbb{C}^{\times}).$$
Then $\mathfrak{q}$ is genuinely of finite Cartan type, with Cartan matrix
$$C=\begin{pmatrix}
2 &0 & 0\\
0 & 2 &0\\
0 &0 & 2
\end{pmatrix}.$$
By \cite[Theorem 15.1.14 (6)]{MR4164719}, the Nichols algebra of $\mathfrak{q}$ is finite-dimensional. In this case, $\Delta_{+}^{\mathfrak{q}}=\mathcal{O}_{+}^{\mathfrak{q}}=\{\alpha_{1},\alpha_{2},\alpha_{3}\}\subseteq \mathbb{Z}^{3}$, with $N_{1}=N_{2}=2$ and $N_{3}=4$. Hence, $q_{ij}^{N_{j}}=1$ for all $1\leq i,j\leq 3$. Therefore, the braiding matrix $\mathfrak{q}$ satisfies the assumption (FCC). We now show that, for the above braiding matrix $\mathfrak q$, there do not
exist $\alpha,\beta\in\mathbb{Z}^{3}$ such that \eqref{eq-S2inner} holds for all $1\leq j\leq 3$. Suppose that there exist $\alpha=(a_{1},a_{2},a_{3}),\beta=(b_{1},b_{2},b_{3})\in \mathbb{Z}^{3}$ such that \eqref{eq-S2inner} holds for all $1\leq j\leq 3$. Since $\mathfrak{q}$ is symmetric, $\mathfrak{q}(\alpha+\beta,\alpha_{j})=q_{jj}^{-1}$ for all $1\leq j\leq 3$. Set $m_{i}=a_{i}+b_{i}$ for $1\leq i\leq 3$. Then, for each $1\leq j\leq 3$, $\prod_{i=1}^{3}q_{ij}^{m_{i}}=q_{jj}^{-1}$. Consequently,
$$(-1)^{m_{1}+m_{2}+m_{3}}=-1,\quad  (-1)^{m_{1}+m_{2}}=-1,\quad (-1)^{m_{1}}\mathfrak{i}^{m_{3}}=-\mathfrak{i}.$$
Observe that $(-1)^{m_{3}}=1$, implying that $m_{3}$ is even. Therefore, the equality $(-1)^{m_{1}}\mathfrak{i}^{m_{3}}=-\mathfrak{i}$ contradicts the evenness of $m_{3}$.
\end{example}


In the setting of Theorem \ref{thm-largequanmodufi-aay1}, define
\begin{equation}\label{eq-AAYsmallqgrp}
\mathfrak{u}_{\mathfrak{q}} \coloneqq U_{\mathfrak{q}}/(Z_{\mathfrak{q}})^{+}U_{\mathfrak{q}}.
\end{equation}
The finite-dimensional Hopf algebra $\mathfrak{u}_{\mathfrak{q}}$ is called the AAY \textit{small quantum group}. There is a canonical exact sequence of Hopf algebras:
$$\mathbb{C}\to Z_{\mathfrak{q}}\to U_{\mathfrak{q}}\to \mathfrak{u}_{\mathfrak{q}}\to\mathbb{C}.$$

\begin{example}
[\cite{MR4506530}] Retain\label{eg-GGsmquangroupmu-9} the notation of \S \ref{subsec-DKquanenalgsCY}, and denote the Cartan matrix $(a_{ij})_{n\times n}$ by $A$. Recall that a braiding matrix $\mathfrak{p}=(p_{ij})_{n\times n}$ of Cartan type with Cartan matrix $A$ is said to be of \textit{integral type} if there exist $p\in\mathbb C^\times$ and
$(b_{ij})_{n\times n}\in M_{n}(\mathbb{Z})$ such that
$$
b_{ij}+b_{ji}=2d_{i}a_{ij} \text{ and\ }p_{ij}=p^{b_{ij}}
$$
for all $1\leq i,j\leq n$. Consider $\mathfrak{q}=(\epsilon^{b_{ij}})_{n\times n}$, where $b_{ij}+b_{ji}=2d_{i}a_{ij}$ for all $1\leq i,j\leq n$. Recall that $\epsilon$ is a primitive $N$-th root of unity, where $N\geq 3$ is odd. Then  $\mathfrak{q}=(\epsilon^{b_{ij}})_{n\times n}$ is of integral type and is genuinely of finite Cartan type. For any $\beta\in \Delta_{+}^{\mathfrak{q}}$, there exists $1\leq i\leq n$ such that $\mathfrak{q}(\beta,\beta)=\epsilon^{2d_{i}}$. Therefore, $N_{\beta}=N$ for all $\beta\in \Delta_{+}^{\mathfrak{q}}$. In this case,
$$\mathfrak{u}_{\mathfrak{q}}=  \mathcal{U}_{\mathfrak{q}}(\mathfrak{g})/(E_{\beta}^{N},F_{\beta}^{N},K_{i}^{ N}-1,L_{i}^{N}-1\mid\beta\in \Delta_{+}^{\mathfrak{q}},1\leq i\leq n)$$
is the small multiparameter quantum group studied by Garc\'ia-Gavarini in \cite[\S 7.3.2]{MR4506530}. If $\mathfrak{q}=(\epsilon^{d_{i} a_{ij}})_{n\times n}$, then Lusztig's small quantum group is a Hopf quotient of $\mathfrak{u}_{\mathfrak{q}}$. 
\end{example}

\begin{remark}\label{rmk-GGsmallquantumgrp-lusz}
Just as Lusztig's small quantum group of a complex semisimple Lie algebra $\mathfrak{g}$ can be realized both as a Hopf quotient of the DK (unrestricted) quantized enveloping algebra of $\mathfrak{g}$ at a root of unity and as a Hopf subalgebra of the Lusztig (restricted) quantized enveloping algebra of $\mathfrak{g}$ at a root of unity, the small multiparameter quantum group of $\mathfrak{g}$ can also be defined using the restricted multiparameter quantized enveloping algebra of $\mathfrak{g}$ (see \cite[\S 7.3.2 and Theorem 7.3.14]{MR4506530}).
\end{remark}

\subsection{Large quantum groups are unimodular}\label{subsec-UqwhenCY}
The aim of this subsection is to prove that, for any braiding matrix
$\mathfrak{q}$ satisfying the assumption (FCC), the associated large quantum group $U_{\mathfrak{q}}$ is unimodular. Consequently, we derive a necessary and sufficient condition for $U_{\mathfrak{q}}$ to be Calabi-Yau. Throughout this subsection, we fix a braiding matrix $\mathfrak{q}=(q_{ij})_{\theta\times \theta}$ satisfying the assumption (FCC) and freely use the notation introduced in the previous subsection.

Before proving the main result of this subsection, we introduce the following notation:
$$\mathscr{R}^{\geq 0}\coloneqq  \{\alpha\in \mathbb{Z}^{\theta}\mid \mathfrak{q}(\alpha,\beta)=1,\forall \beta\in \mathbb{Z}^{\theta}\},$$
$$\mathscr{R}^{\leq 0}\coloneqq  \{\alpha\in \mathbb{Z}^{\theta}\mid \mathfrak{q}(\beta,\alpha)=1,\forall \beta\in \mathbb{Z}^{\theta}\}.$$
Both $\mathscr{R}^{\geq 0}$ and $\mathscr{R}^{\leq 0}$ are subgroups of $\mathbb{Z}^{\theta}$. By Remark \ref{rmk-rhiFCCiff-rhoifcc}, one has $N_{\beta}\beta\in \mathscr{R}^{\geq 0}\cap \mathscr{R}^{\leq 0} $ for any $\beta\in \mathcal{O}^{\mathfrak{q}}$ and $N\gamma\in  \mathscr{R}^{\geq 0}\cap \mathscr{R}^{\leq 0} $ for all $\gamma\in \Delta^{\mathfrak{q}}$. Recall that $N$ is the least common multiple of the integers $N_{\beta}$ for $\beta\in\Delta_{+}^{\mathfrak{q}}$. In particular, both $\mathscr{R}^{\geq 0}$ and $\mathscr{R}^{\leq 0}$ are finite groups. Via the canonical isomorphism $\Gamma^{+}\cong \mathbb {Z}^{\theta},K_{i}\mapsto \alpha_{i}$, the quotient group $\mathbb{Z}^{\theta}/\mathscr{R}^{\geq 0}$ induces a quotient group $\overline{\Gamma^{+}}$ of $\Gamma^{+}$. Moreover, the group algebra
$\mathbb{C}[\overline{\Gamma^{+}}]$ is the Hopf algebra quotient of $\mathbb{C}[\Gamma^{+}]$ by the ideal generated by $\{K_{\alpha}-1\mid \alpha\in \mathscr{R}^{\geq 0}\}$. Similarly, one obtains a quotient group $\overline{\Gamma^{-}}$ of $\Gamma^{-}$. Then $\mathcal{B}(V)$ and $\mathcal{B}(W)$ can be naturally regarded as Hopf algebras in
$$
{}_{\mathbb{C}[\overline{\Gamma^{+}}]}^{\mathbb{C}[\overline{\Gamma^{+}}]}
\mathcal{YD}\text{ and\ }
{}_{\mathbb{C}[\overline{\Gamma^{-}}]}^{\mathbb{C}[\overline{\Gamma^{-}}]}
\mathcal{YD},
$$
respectively. Let $C_{\mathfrak{q}}^{\geq 0}$ and $C_{\mathfrak{q}}^{\leq 0}$ be the
subalgebras of $U_{\mathfrak{q}}$ generated, respectively, by
$$
Z_{\mathfrak q}^{\geq 0}
\cup
\{K_{\alpha}\mid \alpha\in \mathscr{R}^{\geq 0}\}\text{ and\ }
Z_{\mathfrak q}^{\leq 0}\cup
\{L_{\alpha}\mid \alpha\in \mathscr{R}^{\leq 0}\}.$$
By Theorem \ref{thm-largequanmodufi-aay1}, $C_{\mathfrak q}^{\geq 0}$ is a central Hopf subalgebra of $U_{\mathfrak{q}}^{\geq 0}$, and $C_{\mathfrak{q}}^{\leq 0}$ is a central Hopf subalgebra of $U_{\mathfrak q}^{\leq 0}$. Set
$$\overline{U_{\mathfrak{q}}^{\geq 0}}\coloneqq U_{\mathfrak{q}}^{\geq 0}/(C_{\mathfrak q}^{\geq 0})^{+}U_{\mathfrak{q}}^{\geq 0}\text{ and\ }\overline{U_{\mathfrak{q}}^{\leq 0}}\coloneqq U_{\mathfrak{q}}^{\leq 0}/(C_{\mathfrak q}^{\leq 0})^{+}U_{\mathfrak{q}}^{\leq 0}.$$
Then, by the definition of the distinguished pre-Nichols algebra, there are canonical Hopf algebra isomorphisms
\begin{equation}\label{eq-canoiso-oflargborel-BVW}
\overline{U_{\mathfrak q}^{\geq 0}}\cong\mathcal{B}(V)\#\mathbb{C}[\overline{\Gamma^{+}}] \text{ and\ } \overline{U_{\mathfrak q}^{\leq 0}}
\cong \mathcal{B}(W)\#\mathbb{C}[\overline{\Gamma^{-}}].
\end{equation}

\begin{lemma}\label{lem-nondegenpair-ubarleq7bore}
The Hopf skew-pairing $
\tau^{AAY}:U_{\mathfrak{q}}^{\leq 0}\times U_{\mathfrak{q}}^{\geq 0}\to\mathbb{C}$
descends to a skew-pairing $\overline{\tau^{AAY}}:\overline{U_{\mathfrak{q}}^{\leq 0}}\times
\overline{U_{\mathfrak{q}}^{\geq 0}}\to\mathbb{C}$,
which is nondegenerate.
\end{lemma}

\begin{proof}
We identify $\overline{U_{\mathfrak q}^{\geq 0}}\cong \mathcal{B}(V)\#\mathbb{C}[\overline{\Gamma^{+}}] $ and $\overline{U_{\mathfrak q}^{\leq 0}}
\cong \mathcal{B}(W)\#\mathbb{C}[\overline{\Gamma^{-}}]$ via the canonical isomorphisms \eqref{eq-canoiso-oflargborel-BVW}. Then the existence of $\overline{\tau^{AAY}}:\overline{U_{\mathfrak{q}}^{\leq 0}}\times\overline{U_{\mathfrak{q}}^{\geq 0}}\to\mathbb{C}$ follows from \cite[Theorem 3.7 (1)]{MR2732981}. Clearly, $\overline{\tau^{AAY}}( \overline{L_{\beta}}, \overline{K_{\alpha}})=\mathfrak{q}(\alpha,\beta)^{-1}$ for any $\alpha,\beta\in \mathbb{Z}^{\theta}$. It follows from the definitions of $\mathscr{R}^{\geq 0}$ and $\mathscr{R}^{\leq 0}$, together with Dedekind's lemma on the linear independence of group characters, that the restriction of $\overline{\tau^{AAY}}$ to $\mathbb{C}[\overline{\Gamma^{-}}]\times
\mathbb{C}[\overline{\Gamma^{+}}]$ is also nondegenerate. In particular, $|\overline{\Gamma^{-}}|=|\overline{\Gamma^{+}}|$.  Let $\{\delta_{1},\ldots,\delta_{t}\}$ be a set of representatives for the cosets of $\mathscr{R}^{\leq 0}$ in $\mathbb{Z}^{\theta}$, and let
$\{\partial_{1},\ldots,\partial_{t}\}$ be a set of representatives for the cosets of
$\mathscr{R}^{\geq 0}$ in $\mathbb{Z}^{\theta}$. Then
$$
\{\overline{L_{\delta_{1}}},\ldots,\overline{L_{\delta_t}}\}\text{ and\ }
\{\overline{K_{\partial_{1}}},\ldots,\overline{K_{\partial_{t}}}\}
$$
are $\mathbb{C}$-bases of $\mathbb{C}[\overline{\Gamma^{-}}]$ and $\mathbb{C}[\overline{\Gamma^{+}}]$, respectively. Consequently,
$$
\{
\overline{\tau^{AAY}}(-,\overline{K_{\partial_{1}}}),
\ldots,
\overline{\tau^{AAY}}(-,\overline{K_{\partial_{t}}})
\}
$$
forms a $\mathbb{C}$-basis of
$(\mathbb{C}[\overline{\Gamma^{-}}])^{*}$, since the restriction of $\overline{\tau^{AAY}}$ to $\mathbb{C}[\overline{\Gamma^{-}}]\times
\mathbb{C}[\overline{\Gamma^{+}}]$ is nondegenerate.

We now prove that the skew-pairing $\overline{\tau^{AAY}}:
\overline{U_{\mathfrak q}^{\leq 0}}\times\overline{U_{\mathfrak q}^{\geq 0}}\to \mathbb{C}$ is nondegenerate. 

Suppose that
$u^{\leq 0}\in \overline{U_{\mathfrak{q}}^{\leq 0}}$ satisfies
\begin{equation}\label{eq-uleq0tau-0foralriel5}
\overline{\tau^{AAY}}(u^{\leq 0},v^{\geq 0})=0
\end{equation}
for all $v^{\geq 0}\in \overline{U_{\mathfrak{q}}^{\geq 0}}$.
We prove that $u^{\leq 0}=0$. The proof for the opposite side is analogous.

By \cite[Theorem 3.11 (1)]{MR2732981},
\begin{equation}\label{eq-ARStaumultspli-0}
\overline{\tau^{AAY}}(p^{-}\overline{L_{\beta}},p^{+} \overline{K_{\alpha}}) =\overline{\tau^{AAY}}(p^{-} ,p^{+})\overline{\tau^{AAY}}( \overline{L_{\beta}}, \overline{K_{\alpha}})
\end{equation}
for any $p^{+}\in \mathcal{B}(V)$, $p^{-}\in \mathcal{B}(W)$, and $\alpha,\beta\in \mathbb{Z}^{\theta}$. Write
$$
u^{\leq 0}=\sum_{k=1}^{t}p_{k}^{-}\overline{L_{\delta_{k}}}\in \overline{U_{\mathfrak{q}}^{\leq 0}},
$$
where $p_{k}^{-}\in\mathcal{B}(W)$ for $1\leq k\leq t$. Then, by \eqref{eq-uleq0tau-0foralriel5} and \eqref{eq-ARStaumultspli-0}, one has
$$0=\sum_{k=1}^{t}\overline{\tau^{AAY}}(p_{k}^{-}\overline{L_{\delta_{k}}},p^{+}\overline{K_{\partial_{j}}})=\sum_{k=1}^{t}\overline{\tau^{AAY}}(p_{k}^{-},p^{+})\overline{\tau^{AAY}}( \overline{L_{\delta_{k}}}, \overline{K_{\partial_{j}}})$$
for any $p^{+}\in \mathcal{B}(V)$ and any $1\leq j\leq t$. Hence
$$ \sum_{k=1}^{t}\overline{\tau^{AAY}}(p_{k}^{-},p^{+}) \overline{L_{\delta_{k}}}=0$$
for any $p^{+}\in  \mathcal{B}(V)$. It follows that $\overline{\tau^{AAY}}(p_{k}^{-},p^{+}) =0$ for any $p^{+}\in \mathcal{B}(V)$ and any $1\leq k\leq t$. By \cite[Theorem 3.11 (2)(3)]{MR2732981}, the restriction of $\overline{\tau^{AAY}}$ to $\mathcal {B}(W)\times \mathcal{B}(V)$ is nondegenerate. Therefore, $p_{k}^{-}=0$ for any $1\leq k\leq t$. Hence, $u^{\leq 0}=0$.
\end{proof}

\begin{theorem}
Assume that the braiding matrix $\mathfrak{q}$ satisfies the assumption (FCC).\label{thm-largqungrp-CY}
\begin{itemize}
\item[(a)]The large quantum group $U_{\mathfrak{q}}$ is unimodular.
\item[(b)]The Frobenius extension $Z_{\mathfrak{q}} \subseteq U_{\mathfrak{q}}$ is symmetric if there exist $\alpha,\beta\in \mathbb{Z}^{\theta}$ such that \eqref{eq-S2inner} holds for all $j\in \mathbb{I}$.
\item[(c)]The large quantum group $U_{\mathfrak{q}}$ is Calabi-Yau if and only if $\mathfrak q$ is of Cartan type and there exist $\alpha,\beta\in \mathbb{Z}^{\theta}$ such that \eqref{eq-S2inner} holds for all $j\in\mathbb{I}$.
\end{itemize}
\end{theorem}

\begin{proof}
(a) Let $C_{\mathfrak{q}}$ be the subalgebra of $U_{\mathfrak{q}}$ generated by
$C_{\mathfrak q}^{\leq 0}$ and $C_{\mathfrak q}^{\geq 0}$. Then $C_{\mathfrak{q}}$ is a central Hopf subalgebra of $U_{\mathfrak{q}}$ containing $Z_{\mathfrak{q}}$. Set
$$\overline{U_{\mathfrak q}}=U_{\mathfrak q}/(C_{\mathfrak q})^{+}U_{\mathfrak{q}}.$$
By Lemma \ref{lem-nondegenpair-ubarleq7bore}, $\overline{U_{\mathfrak q}}$ coincides with Drinfeld's quantum double of $\overline{U_{\mathfrak{q}}^{\leq 0}}$ and $\overline{U_{\mathfrak{q}}^{\geq 0}}$ with respect to the nondegenerate Hopf skew-pairing $\overline{\tau^{AAY}}:\overline{U_{\mathfrak{q}}^{\leq 0}}\times \overline{U_{\mathfrak{q}}^{\geq 0}}\to\mathbb{C}$. Thus $\overline{U_{\mathfrak q}}$ is a finite-dimensional unimodular Hopf algebra by \cite[Theorem 10.3.12]{MR1243637}. By Theorem \ref{thm-modulefinitehopf-integral} (b), the large quantum group $U_{\mathfrak{q}}$ is unimodular.

(b) This follows from (a), Theorem \ref{thm-modulefinitehopf-integral} (c) and Proposition \ref{prop-UqS2-inner}.

(c) This is a direct consequence of (b), Theorem \ref{thm-modulefinitehopf-integral} (c) and Corollary \ref{cor-ASgorenlargequan8-eregucart6} (e).
\end{proof}
\begin{remark}
By the proof of Theorem \ref{thm-largqungrp-CY} (a), $\overline{U_{\mathfrak{q}}}$ is the Drinfeld quantum double of $\overline{U_{\mathfrak{q}}^{\leq 0}}$ and $\overline{U_{\mathfrak{q}}^{\geq 0}}$ with respect to the nondegenerate Hopf skew-pairing $\overline{\tau^{AAY}}$. Then $\overline{U_{\mathfrak{q}}}$ is a symmetric Frobenius algebra, as it is unimodular and admits a quasi-triangular structure. However, as shown in Example \ref{eg-finitegen-maynotS2iner}, $U_{\mathfrak{q}}$ need not be a symmetric Frobenius extension of $C_{\mathfrak{q}}$. This shows that there exist an affine Hopf algebra $H$ and a central Hopf subalgebra $C$ of $H$ such that $H$ is finitely generated as a module over $C$ and $H/C^{+}H$ is a symmetric Frobenius algebra, while $H$ is not a symmetric Frobenius extension of $C$.
\end{remark}

\begin{remark}
The results of \cite[\S 2.12]{MR3736568} can be used to compute the integral of $\overline{U_{\mathfrak q}^{\geq 0}}$ and, consequently, the homological integral and the Nakayama automorphism of the large quantized Borel algebra $U_{\mathfrak q}^{\geq 0}$. We leave the details to the interested reader.
\end{remark}

By Theorem \ref{thm-largqungrp-CY} (c) and Corollary \ref{cor-PHRmulparaquangro-regulardoman}, we obtain the following.

\begin{corollary}\label{cor-mulQGrp-whenCY}
Assume that the braiding matrix $\mathfrak{q}$ satisfies the assumption (FCC) and is genuinely of finite Cartan type with Cartan matrix $C=(c_{ij})_{\theta\times\theta}$. Let $ \mathfrak{g}_{C}$ be the complex semisimple Lie algebra associated with 
 $C$. Then the multiparameter
quantized enveloping algebra $\mathcal{U}_{\mathfrak{q}}(\mathfrak{g}_{C})$ is Calabi-Yau if and only if there exist $\alpha,\beta\in \mathbb{Z}^{\theta}$ such that \eqref{eq-S2inner} holds for all $j\in\mathbb{I}$.
\end{corollary}


\section{Unimodularity of module-finite Hopf algebras via Yadav's unimodular module categories}\label{sec-yadaunimodular-cat}
Throughout this section, $\mathbbm{k}$ is an algebraically closed field.
\subsection{Unimodular tensor categories and module categories}
In this subsection, we recall the notions of unimodular tensor categories \cite{MR2097289,MR3242743} and unimodular module categories in the sense of Yadav \cite{MR4632574}, and fix notation. For basic terminology and further background on tensor and module categories, see \cite{MR3242743}.


Let $\mathcal{C}$ be a finite tensor category, and let $P(X)$ denote the projective cover of an object $X$ in $\mathcal{C}$. By \cite[Lemma 6.4.1]{MR3242743}, there exists an invertible object $X_{\rho}$ in $\mathcal{C}$, unique up to isomorphism, such that $P(X_{\rho})\cong P({\bf 1})^{*}$. Here, $P({\bf 1})^{*}$ denotes the left dual of the projective cover of the unit object ${\bf 1}$. The invertible object $X_{\rho}$ is called the \textit{distinguished invertible object} of $\mathcal{C}$. A finite tensor category $\mathcal{C}$ is \textit{unimodular} if $X_{\rho} \cong {\bf 1}$. 

\begin{remark}
For finite multi-tensor categories, the distinguished invertible object may be defined via the \textit{Nakayama functor}, which in turn permits one to define the unimodularity for finite multi-tensor categories. 
For further details, see \cite[\S 2.3.4]{MR4632574}.
\end{remark}

All fusion tensor categories and all factorizable finite tensor categories are unimodular; see \cite[Corollary 6.4]{MR2097289} and \cite[Proposition 8.10.10]{MR3242743}. It is well known that semisimple MTCs give rise to three-dimensional TQFTs \cite{MR1091619,MR1292673}. By \cite[Theorem 1.1]{MR3996323}, non-semisimple MTCs in the sense of Kerler-Lyubashenko \cite{MR1862634} can be equivalently defined as factorizable ribbon finite tensor categories. Moreover, by \cite[Theorem 0.4.2]{MR1862634}, any non-semisimple MTC $\mathcal{C}$ admits an extended TQFT whose circle category is $\mathcal{C}$.


When $\mathcal{C}=H\text{-mod}$ is the category of finite-dimensional modules over a finite-dimensional Hopf algebra $H$, the distinguished invertible object of $\mathcal{C}$ is precisely the $1$-dimensional $H$-module induced by the distinguished group-like element of the dual Hopf algebra $H^{\circ}$ \cite[Proposition 6.5.5]{MR3242743}. In this case, the tensor category $H\text{-mod}$ is unimodular if and only if the Hopf algebra $H$ is unimodular.

\begin{remark}\label{rmk-unimocatiff-revuni8}
Let $\mathcal{C}^{\text{rev}}$ denote the \textit{reverse tensor category} of $\mathcal{C}$, i.e., the opposite tensor category defined in \cite[Definition 2.1.5]{MR3242743}. The underlying category of $\mathcal{C}^{\text{rev}}$ is the same as that of $\mathcal{C}$, while its tensor product is reversed: for any objects $X,Y$ in $\mathcal{C}$, one defines $X\otimes^{\text{rev}}Y\coloneqq Y\otimes X$. For example, let $H$ be a Hopf algebra with bijective antipode $S$, and let $V$ be a finite-dimensional right $H$-module. Define a left $H$-module $F(V)=V$ by $h\cdot v=vS(h)$ for any $h\in H$ and $v\in V$. Then $F:\text{mod-}H\to H\text{-mod}$ induces a tensor equivalence $\text{mod-}H\cong (H\text{-mod})^{\text{rev}}$. For any finite tensor category $\mathcal{C}$, $\mathcal{C}$ is unimodular if and only if $\mathcal{C}^{\text{rev}}$ is unimodular. 
\end{remark}

Let $\mathcal{M}$ and $\mathcal{N}$ be finite left $\mathcal{C}$-module categories. The category of right exact $\mathcal{C}$-module functors from $\mathcal{M}$ to $\mathcal{N}$ is denoted by $\operatorname{Rex}_{\mathcal{C}}(\mathcal{M},\mathcal{N})$. If $\mathcal{M}$ is exact, then any $\mathcal{C}$-module functor from $\mathcal{M}$ to $\mathcal{N}$ is exact \cite[Proposition 7.6.9]{MR3242743}. Moreover, by \cite[Proposition 7.11.6]{MR3242743},  $\operatorname{Rex}_{\mathcal{C}}(\mathcal{M},\mathcal{N})$ is a finite $\mathbbm{k}$-linear abelian category. When $\mathcal{M}=\mathcal{N}$, we abbreviate $\operatorname{Rex}_{\mathcal{C}}(\mathcal{M},\mathcal{M})$ as $\operatorname{Rex}_{\mathcal{C}}(\mathcal{M})$. The category $\operatorname{Rex}_{\mathcal{C}}(\mathcal{M})$ has a canonical monoidal structure, with tensor product given by the functor composition, and unit object given by the identity functor. If $\mathcal{M}$ is an indecomposable exact $\mathcal{C}$-module category, $\operatorname{Rex}_{\mathcal{C}}(\mathcal{M})$ is a finite tensor category. In this case, the left/right dual of any module functor $F:\mathcal{M}\to\mathcal{M}$ is given by its left/right adjoint functor, see \cite[\S 7.12]{MR3242743} for details.

\begin{example}
[\cite{MR2331768}] Let $H$ be a finite-dimensional Hopf algebra.\label{eg-AMcomputeRexcmalg} For any finite-dimensional left $H$-comodule algebra $B$, the category $B\text{-mod}$ naturally admits the structure of a left module category over $H\text{-mod}$, induced by the coaction $\delta_{B}: B\to H\otimes B, b\mapsto \sum_{(b)}b_{(-1)}\otimes b_{(0)}$. Conversely, by \cite[Proposition 1.19]{MR2331768}, any finite left module category over $H\text{-mod}$ is equivalent to the category of finite-dimensional modules over a finite-dimensional left $H$-comodule algebra. In what follows, we fix two finite-dimensional left $H$-comodule algebras $A$ and $B$.

Let $\mathcal{C}=H\text{-mod}$. Then $A\text{-mod}$ and $B\text{-mod}$ are finite left $\mathcal{C}$-module categories. Let $\text{Bimod}^{H}(A,B)$ denote the category of $A$-$B$ bimodules in  $H\text{-comod}$, the category of finite-dimensional left $H$-comodules. The objects of $\text{Bimod}^{H}(A,B)$ are precisely the finite-dimensional equivariant $(A,B)$-bimodules considered in \cite[Definition 1.22]{MR2331768}. If $A=B$, then $(\text{Bimod}^{H}(A,A),\otimes_{A})$ is a $\mathbbm{k}$-linear monoidal category by \cite[Exercise 7.8.28]{MR3242743}.

The category $\text{Rex}_{\mathcal{C}}(B\text{-mod},A\text{-mod})$ was explicitly described by Andruskiewitsch-Mombelli in \cite{MR2331768}. For any $P\in\text{Bimod}^{H}(A,B)$, the functor $ P\otimes_{B}- $, with the natural isomorphism 
$$s_{X,V}:P\otimes_{B}(X\otimes V)\to  X\otimes (P\otimes_{B}V), p\otimes_{B}\otimes x\otimes v\mapsto \sum_{(p)}p_{(-1)}x\otimes p_{(0)}\otimes_{B} v,$$
for finite-dimensional $H$-modules $X$ and $B$-modules $V$, defines a right exact $\mathcal{C}$-module functor from $B\text{-mod}$ to $A\text{-mod}$. Hence, there is a $\mathbbm{k}$-linear functor
\begin{equation}\label{eq-AMfunceq-RexBA}
\text{Bimod}^{H}(A,B)\to \text{Rex}_{\mathcal{C}}(B\text{-mod},A\text{-mod}),P\mapsto (P\otimes_{B}-,s).
\end{equation}
By \cite[Proposition 1.23]{MR2331768}, the functor \eqref{eq-AMfunceq-RexBA} is an equivalence of categories. Moreover, \eqref{eq-AMfunceq-RexBA} clearly gives a $\mathbbm{k}$-linear monoidal equivalence
\begin{equation}\label{eq-AMfunceq-RexAssc}
\text{Bimod}^{H}(A,A)\cong\text{Rex}_{\mathcal{C}}(A\text{-mod}).
\end{equation}
\end{example}

As noted in the introduction, Yadav \cite{MR4632574} introduced the notion of a unimodular module category for exact module categories over finite tensor categories, providing a functorial construction of Frobenius algebras in the Drinfeld center of $ \mathcal{C} $. 

For our purposes, we restrict the definition of unimodularity to indecomposable exact module categories. The definition for general exact module categories is analogous, except that $\operatorname{Rex}_{\mathcal{C}}(\mathcal{M})$ is a multi-tensor category in this case. Unimodularity for indecomposable exact module categories is sufficient for the subsequent results.

\begin{definition}
[\cite{MR4632574}]An indecomposable exact left $\mathcal{C}$-module category $\mathcal{M}$ is called \textit{unimodular} if the finite tensor category $\operatorname{Rex}_{\mathcal{C}}(\mathcal{M})$ is unimodular.\label{def-unimmodu-excat}
\end{definition}

\begin{remark}
See \cite[Definition 3.2]{MR4632574} for an equivalent definition of unimodularity for exact module categories. There exists a non-unimodular finite-dimensional Hopf algebra $H$ such that the finite tensor category $H\text{-mod}$ admits a unimodular module category \cite[Example 3.6 (iii)]{MR4632574}. On the other hand, there exists a finite-dimensional Hopf algebra $H$ (e.g., Taft algebras \cite[Theorem 4.45]{MR4632574}) such that the finite tensor category $H\text{-mod}$ admits no unimodular module category.
\end{remark}

Let $H$ be a finite-dimensional Hopf algebra. A left $H$-comodule algebra $B$ is called \textit{exact} (resp., \textit{indecomposable}) if $B\text{-mod}$ is exact (resp., indecomposable) as a module category over $H\text{-mod}$. For example, the trivial $H$-comodule algebra $\mathbbm{k}$ is exact and indecomposable.
\begin{definition}
[\cite{MR4632574}] Let $H$ be a finite-dimensional Hopf algebra and let $B$ be a finite-dimensional, indecomposable, exact left $H$-comodule algebra. The comodule algebra $B$ is called \textit{unimodular} if $B\text{-mod}$ is unimodular as a module category over $H\text{-mod}$.
\end{definition}

Let $H$ be a finite-dimensional Hopf algebra, and let $B$ be a finite-dimensional left $H$-comodule algebra. A two-sided ideal $I$ of $B$ is called an \textit{$H$-ideal} if it is an $H$-subcomodule of $B$. A left $H$-comodule algebra $B$ is called \textit{$H$-simple} if it has no nontrivial $H$-ideals. By \cite[Propositions 1.18 and 1.20]{MR2331768}, every $H$-simple comodule algebra is exact and indecomposable. Conversely, for any indecomposable exact module category $\mathcal{M}$ over $H\text{-mod}$, there exists a finite-dimensional $H$-simple comodule algebra $B$ such that $\mathcal{M}$ is equivalent, as a module category, to $B\text{-mod}$.

\subsection{Unimodularity of bi-Galois objects}
Let $H$ and $L$ be Hopf algebras, and let $A$ be an $H$-$L$-bicomodule algebra. The bicomodule algebra $A$ is called an \textit{$H$-$L$-bi-Galois object} \cite{MR1408508,MR2075600} if it is both a left $H$-Galois object and a right $L$-Galois object. By \cite[Corollary 3.2.4]{MR2075600}, $H$ and $L$ are \textit{monoidally Morita-Takeuchi equivalent} if and only if there exists an $H$-$L$-bi-Galois object. Thus, bi-Galois objects provide an approach to constructing monoidal equivalences between comodule categories over Hopf algebras. In \cite[Theorem 3.5]{MR1408508}, Schauenburg proved that for any right $L$-Galois object $A$, there exists a Hopf algebra $H$, unique up to isomorphism, such that $A$ is an $H$-$L$-bi-Galois object. The following observation is well known (see, for example, \cite[\S 3.1]{MR2678630}).

\begin{lemma}
Let $H$ be a finite-dimensional Hopf algebra and let $A$ be a left $H$-Galois object. Then $A$ is $H$-simple as a left $H$-comodule algebra. In particular, $A\text{-mod}$ is an indecomposable exact left module category over $H\text{-mod}$.\label{lem-galoobj-simindeomex}
\end{lemma}

For the remainder of this subsection, assume that $H$ and $L$ are finite-dimensional Hopf algebras. Consequently, any left $H$-Galois object $A$ is finite-dimensional \cite[Theorem 8.3.1]{MR1243637}; in particular, $\dim_{\mathbbm{k}}A=\dim_{\mathbbm{k}}H$. By Lemma \ref{lem-galoobj-simindeomex}, the left $H$-Galois object $A$ is an indecomposable exact comodule algebra. Therefore, we may consider the unimodularity of $A$ as a left $H$-comodule algebra. The unimodularity of $H$-$L$-bi-Galois objects is closely related to that of the Hopf algebra $L$.

\begin{proposition}
Let $H$ and $L$ be finite-dimensional Hopf algebras admitting an $H$-$L$-bi-Galois object. Then the following conditions are equivalent.\label{prop-bigaloisunimoduiff}
\begin{itemize}
\item[(i)] $L$ is a unimodular Hopf algebra;
\item[(ii)] Every $H$-$L$-bi-Galois object is unimodular as an $H$-comodule algebra;
\item[(iii)] There exists an $H$-$L$-bi-Galois object that is unimodular as an $H$-comodule algebra.
\end{itemize}
In particular, if $H$ is a finite-dimensional unimodular Hopf algebra, then every $H$-$H$-bi-Galois object is unimodular as a left $H$-comodule algebra.
\end{proposition}

\begin{proof}
Let $A$ be an $H$-$L$-bi-Galois object and set $\mathcal{C}=H\text{-mod}$.  By Lemma \ref{lem-galoobj-simindeomex}, $A\text{-mod}$ is an indecomposable exact left module category. We claim that
\begin{equation}\label{eq-rexAAtens-Lmodrev}
\text{Rex}_{\mathcal{C}}(A\text{-mod} )\cong (L\text{-mod})^{\text{rev}}
\end{equation}
as tensor categories. By the proof of \cite[Theorem 3.3.1]{MR2075600}, there is a tensor equivalence $\text{Bimod}^{H}(A,A)\cong \text{mod-}L\cong (L\text{-mod})^{\text{rev}}$. Hence, \eqref{eq-rexAAtens-Lmodrev} follows from \eqref{eq-AMfunceq-RexAssc}.

Now, (i) $\Rightarrow$ (ii) and (iii) $\Rightarrow$ (i) follow immediately from Remark \ref{rmk-unimocatiff-revuni8} and \eqref{eq-rexAAtens-Lmodrev}. The implication (ii) $\Rightarrow$ (iii) is immediate.
\end{proof}

\begin{remark}
By \cite[Definition 7.12.17]{MR3242743}, the tensor equivalence \eqref{eq-rexAAtens-Lmodrev} implies that, under the assumptions of Proposition \ref{prop-bigaloisunimoduiff}, the categories $L\text{-mod}$ and $H\text{-mod}$ are categorically Morita equivalent. In the case $L=H$, the tensor equivalence \eqref{eq-rexAAtens-Lmodrev} follows directly from \cite[Lemma 4.1 (1)]{MR3209810} and \cite[Proposition 4.2 (iv)]{MR2677836}.
\end{remark}

\subsection{Unimodularity of fiber algebras}\label{subsec-fiberalgg-mfhopf}
In this subsection, we fix an affine Hopf algebra $H$ over an algebraically closed field $\mathbbm{k}$ and assume that $H$ is finitely generated as a module over a central Hopf subalgebra $C$. Since the antipode of $H$ is bijective, $H\text{-mod}$ is a tensor category under the tensor product. For any finite-dimensional $H$-module $V$, both the left dual $V^{*}$ and the right dual $^{*}{V}$ of $V$ in $H\text{-mod}$ are given by the ordinary dual space of $V$, equipped with the respective $H$-actions 
$$(h\xi)(v)=\xi(S(h)v) \text{ and }(h\xi)(v)=\xi(S^{-1}(h)v)$$ 
for all $h\in H$, $v\in V$, and $\xi$ in the dual space of $V$.

Let $G(C^{\circ})$ denote the group of group-like elements of the dual Hopf algebra $C^{\circ}$. There is a canonical bijection $G(C^{\circ})\cong \operatorname{maxSpec}C$ given by $\chi\mapsto \mathfrak{m}_{\chi}\coloneqq \text{Ker}\chi$. For $\chi,\psi\in G(C^{\circ})$, their convolution product 
is denoted by $\chi\ast \psi$. Write $\mathfrak{m}_{\chi\ast \psi}$ as $\mathfrak{m}_{\chi}\ast\mathfrak{m}_{\psi}$. It follows that $\operatorname{maxSpec}C$ is a group under the convolution product. If, moreover, $C$ is smooth (e.g., when $\mathbbm{k}$ is of characteristic zero), then $\operatorname{maxSpec}C$ is an affine algebraic group over $\mathbbm{k}$.  The identity element of $ \operatorname{maxSpec}C$ with respect to the convolution product is $\mathfrak{m}_{\ovep}$, where $\ovep=\varepsilon|_{C}:C\to\mathbbm{k}$. For any $\chi\in G(C^{\circ})$, the inverse of $\chi$ with respect to the convolution product is $\chi S$. Thus, the inverse of $\mathfrak{m}\in \operatorname{maxSpec}C$ is $\mathfrak{m}_{\chi S}$. If $V$ is a finite-dimensional $H/\mathfrak{m}_{\chi}H$-module, then both $V^{*}$ and $^{*}V$ are $H/\mathfrak{m}_{\chi S}H$-modules. For any $\mathfrak{m}\in \operatorname{maxSpec}C$, the finite-dimensional algebra $H/\mathfrak{m}H$ is called the \textit{fiber algebra} of $H$ at $\mathfrak{m}$. Hence, one obtains a family of finite-dimensional $\mathbbm{k}$-algebras parameterized by $\operatorname{maxSpec}C$:
$$
\mathcal{B}_{H}=\{H/\mathfrak{m}H\mid  \mathfrak{m}\in \operatorname{maxSpec}C\}.
$$
The family of fiber algebras $ \mathcal{B}_{H}$ plays a central role in the study of irreducible representations of $H$, as it is a standard fact that the set of isomorphism classes of irreducible $H$-modules is the disjoint union of the sets of isomorphism classes of irreducible modules over the fiber algebras in $ \mathcal{B}_{H}$. For related background, see, for example, \cite{MR1898492,MR1676211,MR1288995,MR1989650,MR4891381,HMQW2026chev,HQWZ2026chev}. In \cite[Proposition 3.1 (2)]{HMQW2026chev}, the authors established that $H$ is a finitely generated projective $C$-module of constant rank $r$. Consequently, each fiber algebra $H/\mathfrak{m}H$ in $\mathcal{B}_{H}$ satisfies $\dim_{\mathbbm{k}} H/\mathfrak{m}H = r.$ 

\begin{remark}\label{rmk-bundleofalgs}
If $C$ is smooth and $\widetilde{H}$ denotes the sheaf on $\operatorname{maxSpec} C$ associated with the $C$-module $H$, then $\widetilde{H}$ is a vector bundle of rank $r$ over the algebraic group $\operatorname{maxSpec} C$. For each
$\mathfrak{m}\in \operatorname{maxSpec} C$, the fiber of the vector bundle $\widetilde{H}$ at $\mathfrak{m}$ is
$$\widetilde{H}(\mathfrak{m})=H_{\mathfrak{m}}\otimes_{C_{\mathfrak{m}}}C_{\mathfrak{m}}/\mathfrak{m}C_{\mathfrak{m}}\cong H/\mathfrak{m}H,$$
which is precisely the fiber algebra of $H$ at $\mathfrak{m}$.
\end{remark}

For any two characters $\chi,\psi\in G(C^{\circ})$, the coproduct $\Delta:H\to H\otimes H$ induces an algebra homomorphism
\begin{equation}\label{eq-Del-fibers}
\Delta_{\chi,\psi}:H/\mathfrak{m}_{\chi\ast \psi}H\to H/\mathfrak{m}_{\chi}H\otimes H/\mathfrak{m}_{\psi}H,\overline{h}\mapsto \sum_{(h)}\overline{h_{(1)}}\otimes \overline{h_{(2)}}.
\end{equation}
Consequently, the identity fiber algebra $ H/\mathfrak{m}_{\ovep}H$ is a finite-dimensional Hopf algebra with coproduct $\Delta_{\ovep,\ovep}$. It follows from \eqref{eq-Del-fibers} that each fiber algebra $H/\mathfrak{m}H$ is a bicomodule algebra over $H/\mathfrak{m}_{\ovep}H$. Moreover, \eqref{eq-Del-fibers} induces a bifunctor
\begin{equation}\label{eq-bifunct-conv-fiberalgsten}
\otimes : H/\mathfrak{m}_{\chi}H\text{-mod} \times H/\mathfrak{m}_{\psi}H\text{-mod}\to H/\mathfrak{m}_{\chi\ast\psi}H\text{-mod}.
\end{equation}
In particular, for any $\mathfrak{m}\in \operatorname{maxSpec}C$, \eqref{eq-bifunct-conv-fiberalgsten} endows $H/\mathfrak{m}H\text{-mod}$ with the structures of both a left and a right module category over the finite tensor category $H/\mathfrak{m}_{\ovep}H\text{-mod}$.

\begin{remark}
The identity fiber algebra $H/\mathfrak{m}_{\ovep}H$ is also called the \textit{restricted Hopf algebra} \cite{MR1898492} associated with the pair $(H,C)$. This terminology originates from the case where $H=\mathcal{U}(\mathfrak{g})$ is the universal enveloping algebra of a finite-dimensional restricted Lie algebra $(\mathfrak{g},[p])$ over a field $\mathbbm{k}$ of characteristic $p>0$, and $C=\mathbbm{k}[x^{p}-x^{[p]}\mid x\in \mathfrak{g}]$ is the $p$-center of $\mathcal{U}(\mathfrak{g})$. In this setting, the identity fiber algebra is precisely the \textit{restricted enveloping algebra} $\mathfrak{u}(\mathfrak{g})$ of $(\mathfrak{g},[p])$, while the remaining fiber algebras are the \textit{reduced enveloping algebras} of $(\mathfrak{g},[p])$. 
For many big quantum groups at roots of unity, the corresponding small quantum group arises as the identity fiber of the big quantum group with respect to an appropriate central Hopf subalgebra.
\end{remark}


For any fiber algebra $H/\mathfrak{m}H$, the algebra $H/\mathfrak{m}H$ is an $H/\mathfrak{m}_{\ovep}H$-$H/\mathfrak{m}_{\ovep}H$-bi-Galois object by \cite[Lemma 2.1]{HQWZ2026chev}. Hence, by Lemma \ref{lem-galoobj-simindeomex}, $H/\mathfrak{m}H \text{-mod}$ is an indecomposable exact module category over the tensor category $H/\mathfrak{m}_{\ovep}H\text{-mod}$; see also \cite[Proposition 3.8]{HMQW2026chev} for a direct proof. Theorem \ref{thm-modulefinitehopf-integral} and Proposition \ref{prop-bigaloisunimoduiff} yield:
 
\begin{theorem}
Let $H$ be an affine Hopf algebra, and let $C$ be a central Hopf subalgebra of $H$ such that $H$ is finitely generated as a $C$-module. Then the following conditions are equivalent:\label{thm-unimodularcomodualg-fiber}
\begin{itemize}
\item[(1)] $H$ is unimodular;
\item[(2)]$H/\mathfrak{m}_{\ovep}H$ is a finite-dimensional unimodular Hopf algebra; 
\item[(3)]For any $\mathfrak{m}\in \operatorname{maxSpec}C$, $H/\mathfrak{m}H$ is a unimodular comodule algebra;
\item[(4)]There exists $\mathfrak{m}\in \operatorname{maxSpec}C$ such that $H/\mathfrak{m}H$ is a unimodular comodule algebra.
\end{itemize}
\end{theorem}

\begin{remark}\label{rmk-unimocat-Hopfcogralg}
Theorem \ref{thm-unimodularcomodualg-fiber} extends to the more general setting of Hopf group coalgebras of finite type in the sense of Turaev-Virelizier \cite{MR2674592,MR1903398} (see \cite[\S 1.3]{MR1903398} for the definition). We briefly sketch the argument below. 

Let $G$ be a group with identity element $e$, and let $(\{H_{g}\}_{g\in G},\{\Delta_{g,k}:H_{gk}\to H_{g}\otimes H_{k}\}_{g,k\in G},\varepsilon_{e}: H_{e}\to\mathbbm{k},\{S_{g}:H_{g}\to H_{g^{-1}}\}_{g\in G})$ be a Hopf $G$-coalgebra. Here, $\{H_{g}\}_{g\in G}$ is a family of $\mathbbm{k}$-algebras, $\{\Delta_{g,k}:H_{gk}\to H_{g}\otimes H_{k}\}_{g,k\in G}$ is a family of algebra homomorphisms, and $\{S_{g}:H_{g}\to H_{g^{-1}}\}_{g\in G}$ is a family of algebra anti-homomorphisms satisfying certain compatibility conditions. It can be shown that $(H_{e},\Delta_{e,e},\varepsilon_{e})$ is a Hopf algebra and that each $H_{g}$ is an $H_{e}$-$H_{e}$-bimodule coalgebra with respect to $\Delta_{e,g}$ and $\Delta_{g,e}$. A Hopf $G$-coalgebra $(\{H_{g}\}_{g\in G},\{\Delta_{g,k} \}_{g,k\in G},\varepsilon_{e} ,\{S_{g} \}_{g\in G})$ is said to be \textit{of finite type} if all $H_{g}$ are finite-dimensional over $\mathbbm{k}$. In this case, all $S_{g}$ are bijective \cite[Corollary 3.7]{MR1903398}. For any affine Hopf algebra $H$ admitting a large central Hopf subalgebra $C$, it is straightforward to show that $$(\{H/\mathfrak{m}_{\chi}H\}_{\chi\in G(C^{\circ})},\{\Delta_{\chi,\psi}\}_{\chi,\psi\in G(C^{\circ})},\overline{\varepsilon},\{S_{\chi}\}_{\chi\in G(C^{\circ})})$$ is a finite-type Hopf $G(C^{\circ})$-coalgebra, where $S_{\chi}:H/\mathfrak{m}_{\chi}H\to H/\mathfrak{m}_{\chi S}H$ is the canonical homomorphism induced by the antipode $S$ of $H$.

For any finite-type Hopf $G$-coalgebra $(\{H_{g}\}_{g\in G},\{\Delta_{g,k} \}_{g,k\in G},\varepsilon_{e} ,\{S_{g} \}_{g\in G})$ with $H_{g}\neq 0$, by the proof of \cite[Lemma 2.1]{HQWZ2026chev}, the algebra $H_{g}$ is an $H_{e}$-$H_{e}$-bi-Galois object. Hence, Proposition \ref{prop-bigaloisunimoduiff} applies. Consequently, 
$H_{e}$ is a unimodular Hopf algebra if and only if $H_{g}$ is unimodular as a left $H_{e}$-comodule algebra. 
\end{remark}

\begin{corollary}
Let $H$ be an affine Calabi-Yau Hopf algebra, and let $C$ be a central Hopf subalgebra of $H$ such that $H$ is finitely generated as a $C$-module. Then, for any $\mathfrak{m}\in \operatorname{maxSpec}C$, $H/\mathfrak{m}H$ is a unimodular $H/\mathfrak{m}_{\ovep}H$-comodule algebra.
\end{corollary}

By Theorem \ref{thm-largqungrp-CY} (1), any AAY large quantum group $U_{\mathfrak{q}}$ associated with a braiding matrix $\mathfrak{q}$ satisfying the assumption (FCC) is a unimodular AS-Gorenstein Hopf algebra.
Therefore, Theorem \ref{thm-unimodularcomodualg-fiber} applies, yielding the following:

\begin{corollary}
Assume that the braiding matrix $\mathfrak{q}$ satisfies the assumption (FCC), and let $(U_{\mathfrak{q}}, Z_{\mathfrak{q}})$ be as in Theorem \ref{thm-largequanmodufi-aay1}. Then, for any $\mathfrak{m}\in \operatorname{maxSpec}Z_{\mathfrak{q}}$, the quotient $U_{\mathfrak{q}}/\mathfrak{m}U_{\mathfrak{q}}$ is a unimodular $\mathfrak{u}_{\mathfrak{q}}$-comodule algebra, where $\mathfrak{u}_{\mathfrak{q}}$ is the small quantum group defined in \eqref{eq-AAYsmallqgrp}.
\end{corollary}

\section*{Acknowledgments}
AI was used solely for English grammar and spell checking during the preparation of the manuscript. The authors take full responsibility for the content of the paper.
The authors thank Yuhang Cao, Zhongkai Mi, Xiao Xu, and Jing Yu for helpful discussions and comments. 


\end{document}